\newif\ifLTEX
\LTEXtrue

\documentclass[a4paper]{amsart}

\usepackage{mathtools,amssymb}
\usepackage[abbrev]{amsrefs}
\usepackage{xparse}

\ifLTEX
\usepackage{graphicx}
\else
\usepackage[dvipdfmx]{graphicx}
\fi
 
\usepackage{hyperref}
\usepackage{subcaption}
\usepackage{enumitem}
\usepackage{amsthm}
\usepackage{pdfpages}
\usepackage[all]{xy} 
\usepackage{float}
\usepackage{marginfix}
\usepackage{bm}

\newtheorem{thm}{Theorem}[section]
\newtheorem{lem}[thm]{Lemma}

\newtheorem{prop}[thm]{Proposition}

\theoremstyle{definition}

\newcommand{\R}{\mathbb{R}}

\newcommand{\Z}{\mathbb{Z}}

\newcommand{\M}{\mathrm{Mod}_{0,2n+2}}

\newcommand{\LM}{\mathrm{LMod}_{2n+2}}

\newcommand{\PM}{\mathrm{PMod}_{0,2n+2}}

\newcommand{\Mg}{\mathrm{Mod}_{g}}

\newcommand{\SM}{\mathrm{SMod}_{g;k}}

\newcommand{\Hg}{\mathcal{H}_{g}}
\newcommand{\SH}{\mathcal{SH}_{g;k}}
\newcommand{\SHomH}{\mathrm{SHomeo}_{+}(H_g)}
\newcommand{\SHomS}{\mathrm{SHomeo}_{+}(\Sigma _g)}
\newcommand{\Hil}{\bm{\mathrm{H}}_{2n+2}}
\newcommand{\LH}{\bm{\mathrm{LH}}_{2n+2}}
\newcommand{\PH}{\bm{\mathrm{PH}}_{2n+2}}
\newcommand{\A}{\mathcal{A}}
\newcommand{\B}{\mathcal{B}}

\numberwithin{equation}{section}

\allowdisplaybreaks
\title[Finite presentation for balanced superelliptic handlebody group]{A finite presentation for the balanced superelliptic handlebody group}

\author[G.~Omori]{Genki Omori}
\address{
(Genki Omori)
Department of Mathematics, Faculty of Science and Technology, Tokyo University of Science, 2641 Yamazaki, Noda-shi, Chiba, 278-8510 Japan
}
\email{omori\_genki@ma.noda.tus.ac.jp}

\author[Y.~Yoshida]{Yuya Yoshida}
\address{
(Yuya Yoshida)
Future Architect, Inc., 1-2-2 Osaki, Shinagawa-ku, Tokyo, 141-0032 Japan
}
\email{mathmathyuya@yahoo.co.jp}


\date{\today}

\begin{document}
\maketitle
\begin{abstract}
The balanced superelliptic handlebody group is the normalizer of the transformation group of the balanced superelliptic covering space in the handlebody group of the total space.  
We give a finite presentation for the balanced superelliptic handlebody group. 
To give this presentation, we construct a finite presentation for the liftable Hilden group. 
\end{abstract}

\section{Introduction}


Let $H_{g}$ be a oriented 3-dimensional handlebody of genus $g\geq 0$ and $B^3=H_0$ a 3-ball. 
For integers $n\geq 1$ and $k\geq 2$ with $g=n(k-1)$, the \textit{balanced superelliptic covering map} $p=p_{g,k}\colon H_g\to B^3$ is a $k$-fold branched covering map with the covering transformation group generated by the \textit{balanced superelliptic rotation} $\zeta =\zeta _{g,k}$ of order $k$ (precisely defined in Section~\ref{section_bscov} and see Figure~\ref{fig_bs_periodic_map}). 
The branch points set $\A \subset B^3$ of $p$ is the disjoint union of $n+1$ proper arcs in $B^3$ and the restriction of $p$ to the preimage $\widetilde{\A }$ of $\A $ is injective (i.e. $\widetilde{\A }$ is the fixed points set of $\zeta $). 
We denote $\Sigma _g=\partial H_g$ and $S^2=\partial B^3$. 
The restriction $p|_{\Sigma _g}\colon \Sigma _g\to S^2$ is also a $k$-fold branched covering with the branch points set $\B =\partial \A $ and we also call $p|_{\Sigma _g}$ the balanced superelliptic covering map. 
When $k=2$, $\zeta |_{\Sigma _g}$ coincides with the hyperelliptic involution, and for $k\geq 3$, the the balanced superelliptic covering space was introduced by Ghaswala and Winarski~\cite{Ghaswala-Winarski2}. 
We often abuse notation and denote simply $p|_{\Sigma _g}=p$ and $\zeta |_{\Sigma _g}=\zeta $.  

The mapping class group $\Mg $ of $\Sigma _g$ is the group of isotopy classes of orientation-preserving self-homeomorphisms on $\Sigma _{g}$. 
Finite presentations for $\Mg $ were given by Hatcher-Thurston~\cite{Hatcher-Thurston}, Wajnryb~\cite{Wajnryb1}, and Matsumoto~\cite{Matsumoto}. 
The \textit{handlebody group} $\mathcal{H}_{g}$ is the group of isotopy classes of orientation-preserving self-homeomorphisms on $H_{g}$. 
Wajnryb~\cite{Wajnryb2} gave a finite presentation for $\Hg $. 
We have the will-defined injective homomorphism $\mathcal{H}_{g}\hookrightarrow \Mg $ by restricting the actions of elements in $\mathcal{H}_{g}$ to $\Sigma _g$ and using the irreducibility of $H_g$.  
By this injective homomorphism, we regard $\Hg$ as the subgroup of $\Mg $ whose elements extend to $H_g$.  
Let $\M $ be the group of isotopy classes of orientation-preserving self-homeomorphisms on $S^2$ fixing $\B $ setwise and $\Hil $ the group of isotopy classes of orientation-preserving self-homeomorphisms on $B^3$ fixing $\A $ setwise. 
The group $\Hil $ is introduced by Hilden~\cite{Hilden} and is called the \textit{Hilden group}. 
Tawn~\cite{Tawn1} gave a finite presentation for $\Hil $. 
By restricting the actions of elements in $\Hil $ to $S^2$, we also have the injective homomorphism $\Hil \hookrightarrow \M $ (see \cite[p.~157]{Brendle-Hatcher} or \cite[p.~484]{Hilden}) and regard $\Hil $ as the subgroup of $\M $ whose elements extend to the homeomorpism on $B^3$ which preserve $\A $ by this injective homomorphism. 

For $g=n(k-1)\geq 1$, an orientation-preserving self-homeomorphism $\varphi $ on $\Sigma _{g}$ or $H_g$ is \textit{symmetric} for $\zeta =\zeta _{g,k}$ if $\varphi \left< \zeta \right> \varphi ^{-1}=\left< \zeta \right> $. 
The \textit{balanced superelliptic mapping class group} (or the \textit{symmetric mapping class group}) $\SM $ is the subgroup of $\Mg $ which consists of elements represented by symmetric homeomorphisms. 
Birman and Hilden~\cite{Birman-Hilden3} showed that $\SM $ coincides with the group of symmetric isotopy classes of symmetric homeomorphisms on $\Sigma _g$. 
We call the intersection $\SH =\SM \cap \Hg $ the \textit{balanced superelliptic handlebody group} (or the \textit{symmetric handlebody group}). 
By Lemma~1.21 in \cite{Iguchi-Hirose-Kin-Koda} and Lemma~\ref{lem_interior_symm-liftability}, 
we show that $\SH $ is also isomorphic to the group of  symmetric isotopy classes of symmetric homeomorphisms on $H_g$. 
Since elements in $\SM $ (resp. in $\SH $) are represented by elements which preserve $\B $ (resp. $\A $) by the definitions, we have the homomorphisms $\theta \colon \SM \to \M $ and $\theta |_{\SH}\colon \SH \to \Hil $ that are introduced by Birman and Hilden~\cite{Birman-Hilden2}. 
They also proved that $\theta (\mathrm{SMod}_{g;2} )=\M $ and Hirose and Kin~\cite{Hirose-Kin} showed that $\theta (\mathcal{SH}_{g;2})=\Hil $. 

A self-homeomorphism $\varphi $ on $\Sigma _{0}$ (resp. on $B^3$) is \textit{liftable} with respect to $p=p_{g,k}$ if there exists a self-homeomorphism $\widetilde{\varphi }$ on $\Sigma _{g}$ (resp. on $H_g$) such that $p\circ \widetilde{\varphi }=\varphi \circ p$, namely, the following diagrams commute: 
\[
\xymatrix{
\Sigma _g \ar[r]^{\widetilde{\varphi }} \ar[d]_p &  \Sigma _{g} \ar[d]^p & H_g \ar[r]^{\widetilde{\varphi }}\ar[d]_{p}  &  H_{g}\ar[d]^{p} \\
\Sigma _{0}  \ar[r]_{\varphi } &\Sigma _{0}, \ar@{}[lu]|{\circlearrowright} & B^3  \ar[r]_{\varphi } &B^3. \ar@{}[lu]|{\circlearrowright}
}
\] 
The \textit{liftable mapping class group} $\mathrm{LMod}_{2n+2;k}$ is the subgroup of $\M $ which consists of elements represented by liftable homeomorphisms on $S^2$ for $p_{g,k}|_{S^2}$, and the \textit{liftable Hilden group} $\bm{\mathrm{LH}}_{2n+2;k}$ is the subgroup of $\Hil $ which consists of elements represented by liftable homeomorphisms on $B^3$ for $p_{g,k}$. 
As a homomorphic image in $\M $, we have $\bm{\mathrm{LH}}_{2n+2;k}=\mathrm{LMod}_{2n+2;k}\cap \Hil $ by Lemma~\ref{liftability_surf_hand}. 
By the definitions, we have the homomorphisms $\theta \colon \SM \to \mathrm{LMod}_{2n+2;k}$ and $\theta |_{\SH }\colon \SH \to \bm{\mathrm{LH}}_{2n+2;k}$. 
Birman and Hilden~\cite{Birman-Hilden2} proved that $\theta $ induces the isomorphism $\mathrm{LMod}_{2n+2;k}\cong \SM /\left< \zeta \right> $, and for the case of the handlebody group, Hirose and Kin~\cite{Hirose-Kin} showed that $\theta |_{\SH }$ induces the isomorphism $\bm{\mathrm{LH}}_{2n+2;2}\cong \mathcal{SH}_{g;2} /\left< \zeta _{g,2}\right> $.  
By Lemma~\ref{lem_exact_SH_handlebody}, for $k\geq 3$, we also have the surjective homomorphism $\theta |_{\SH }\colon \SH \to \bm{\mathrm{LH}}_{2n+2;k}$ which induces the isomorphism $\bm{\mathrm{LH}}_{2n+2;k}\cong \SH /\left< \zeta _{g,k}\right> $. 

When $k=2$, $\mathrm{SMod}_{g,2}$ is called the \textit{hyperelliptic mapping class groups} and coincides with the center of $\zeta =\zeta _{g,2}$ in $\Mg $.  
In this case, $\mathrm{LMod}_{2n+2;2}$ is equal to $\M $ and a finite presentation for $\mathrm{SMod}_{g;2}$ was given by Birman and Hilden~\cite{Birman-Hilden1}. 
For the case of the handlebody group, Theorem~2.11 of~\cite{Hirose-Kin} implies that $\bm{\mathrm{LH}}_{2n+2;2}$ is equal to $\Hil $ and Hirose and Kin gave a finite presentation for $\mathcal{SH}_{g;2}$ in Theorem~A.8 of~\cite{Hirose-Kin}. 

When $k\geq 3$, in Lemma~3.6 of \cite{Ghaswala-Winarski1}, Ghaswala and Winarski gave a necessary and sufficient condition for lifting a homeomorphism on $S^2$ with respect to $p_{g,k}$ (see also Lemma~\ref{lem_GW}). 
By their necessary and sufficient condition, we show that $\mathrm{LMod}_{2n+2;k}$ is a proper subgroup of $\M $ and the liftability of a self-homeomorphism on $(S^2, \B )$ does not depend on $k\geq 3$ (the liftability depends on only the action on $\B $). 
Hence we omit ``$k$'' in the notation of the liftable mapping class group and the liftable Hilden group for $k\geq 3$ (i.e. we express $\mathrm{LMod}_{2n+2;k}=\LM $ and $\bm{\mathrm{LH}}_{2n+2;k}=\LH $ for $k\geq 3$).  
Ghaswala and Winarski~\cite{Ghaswala-Winarski1} constructed a finite presentation for $\LM $. 
After that, Hirose and the first author~\cite{Hirose-Omori} gave a finite presentation for $\LM $ in a different generating set from Ghaswala-Winarski's presentation in~\cite{Ghaswala-Winarski1} and a finite presentation for $\SM $. 
In~\cite{Hirose-Omori}, Hirose and the first author also constructed the finite presentations for the liftable mapping class groups and the balanced superelliptic mapping class groups fixing either one point or one 2-disk in $\Sigma _g$. 

In this paper, we give a finite presentation for $\LH $ and $\SH $ (Theorems~\ref{thm_pres_LH} and \ref{thm_pres_SH}), and calculate their integral first homology groups. 
The integral first homology group $H_1(G)$ of a group $G$ is isomorphic to the abelianization of $G$. 
Put $\Z _l=\Z /l\Z $ for an integer $l\geq 2$. 
The results about the integral first homology groups of $\LH $ and $\SH $ are as follows:

\begin{thm}\label{thm_abel_lmod}
For $n\geq 1$, $H_1(\LH )\cong \Z \oplus \Z _2\oplus \Z _{2}$. 
\end{thm}

\begin{thm}\label{thm_abel_smod}
For $n\geq 1$ and $k\geq 3$ with $g=n(k-1)$, 
\[
H_1(\SH )\cong \left\{ \begin{array}{ll}
 \Z \oplus \Z _2\oplus \Z _{2}\oplus \Z _{2}&\text{if }  n \text{ is odd and }k \text{ is even},   \\
 \Z \oplus \Z _{2}\oplus \Z _{2}& \text{otherwise}.
 \end{array} \right.
\]
\end{thm}

The explicit generators of the first homology groups in Theorem~\ref{thm_abel_lmod}~and~\ref{thm_abel_smod} are given in their proofs in Section~\ref{section_abel-lmod} and~\ref{section_abel-smod}. 

Contents of this paper are as follows. 
In Section~\ref{Preliminaries}, we review the definition of the balanced superelliptic covering map $p=p_{g,k}$ and introduce some liftable homeomorphisms on $B^3$ for $p$ and the relations among these liftable homeomorphisms. 
To define these liftable homeomorphisms, we review the \textit{spherical wicket group} in Section~\ref{section_liftable-element}. 
In Section~\ref{section_exact-seq}, we observe the group structure of the group $\LH $ via a group extension (Propositions~\ref{prop_exact_LH}). 
In Section~\ref{section_lmod}, we give a finite presentation for $\LH $ in Theorem~\ref{thm_pres_LH} and its proof, and calculate $H_1(\LH )$ in Section~\ref{section_abel-lmod}. 
To prove Theorem~\ref{thm_pres_LH}, in Section~\ref{section_pure_Hilden}, we give a finite presentation for the pure Hilden group which is obtained from Tawn's finite presentation for the pure Hilden group of one marked disk case in~\cite{Tawn2}. 
Finally, in Section~\ref{section_smod}, we give a finite presentation for $\SH $ (Theorem~\ref{thm_pres_SH}) and calculate $H_1(\SH )$ in Section~\ref{section_abel-smod}. 
In Section~\ref{section_lifts}, we construct explicit lifts of generators for $\LH $ with respect to $p$ which are introduced in Section~\ref{section_liftable-element}. 

\section{Preliminaries}\label{Preliminaries}

\subsection{The balanced superelliptic covering space}\label{section_bscov}

For integers $n\geq 1$, $k\geq 2$, and $g=n(k-1)$, we describe the handlebody $H_{g}$ as follows. 
We take the unit 3-ball $B(1)$ in $\R ^3$ and $n$ mutually disjoint parallel copies $B(2),\ B(3),\ \dots ,\ B(n+1)$ of $B(1)$ by translations along the x-axis such that 
\[
\max \bigl( B(i)\cap (\R \times \{ 0\}\times \{ 0\} )\bigr) <\min \bigl( B(i+1)\cap (\R \times \{ 0\}\times \{ 0\} )\bigr)
\]
in $\R =\R \times \{ 0\}\times \{ 0\} $ for $1\leq i\leq n$ (see Figure~\ref{fig_bs_periodic_map}). 
Let $\zeta$ be the $(-\frac{2\pi }{k})$-rotation of $\R ^3$ on the $x$-axis. 
Then for each $1\leq i\leq n$, we connect $B(i)$ and $B(i+1)$ by $k$ 3-dimensional 1-handles such that the union of the $k$ 3-dimensional 1-handles are preserved by the action of $\zeta $ as in Figure~\ref{fig_bs_periodic_map}. 
Since the union of $B(1)\cup B(2)\cup \cdots \cup B(n+1)$ and the attached $n\times k$ 3-dimensional 1-handles is homeomorphic to $H_{g=n(k-1)}$, we regard this union as $H_g$. 

\begin{figure}[h]
\includegraphics[scale=1.35]{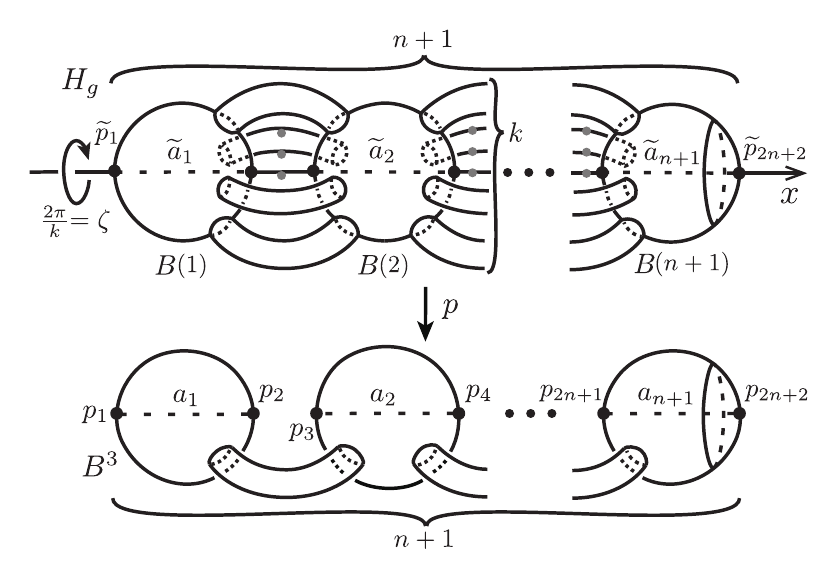}
\caption{The balanced superelliptic covering map $p=p_{g,k}\colon H_g\to B^3$.}\label{fig_bs_periodic_map}
\end{figure}

By the construction above, the action of $\zeta $ on $\R ^3$ induces the action on $H_g$ and the fixed points set of $\zeta $ is $\widetilde{\A }=H_g\cap (\R \times \{ 0\}\times \{ 0\})$. 
We call $\zeta $ the \textit{balanced superelliptic rotation} on $H_g$. 
We can see that the intersection $\widetilde{a}_i=B(i)\cap \widetilde{\A }$ for $1\leq i\leq n+1$ is a proper simple arc in $H_g$ and $\widetilde{\A }=\widetilde{a}_1\sqcup \widetilde{a}_2\sqcup \cdots \sqcup \widetilde{a}_{n+1}$ (see Figure~\ref{fig_bs_periodic_map}). 
The quotient space $H_g/\left< \zeta \right>$ is homeomorphic to $B^3$ and the induced quotient map $p=p_{g,k}\colon H_g\to B^3$ is a branched covering map with the branch points set $\A =p(\widetilde{\A })=p(\widetilde{a}_1)\sqcup p(\widetilde{a}_2)\sqcup \cdots \sqcup p(\widetilde{a}_{n+1})\subset B^3$. 
We call $p$ 
 the \textit{balanced superelliptic covering map}. 
Put 
\begin{itemize}
\item $a_i=p(\widetilde{a}_i)$ \quad for $1\leq i\leq n+1$, 
\item $\widetilde{p}_{2i-1}=\min \widetilde{a}_i$ and $\widetilde{p}_{2i}=\max \widetilde{a}_i$ \quad in $\R =\R \times \{ 0\}\times \{ 0\}$ for $1\leq i\leq n+1$, 
\item $p_i=p(\widetilde{p}_{i})$ \quad for $1\leq i\leq 2n+2$, 
\item $\B =\partial \A =\{ p_1, p_2, \dots , p_{2n+2}\}$, 
\item $\Sigma _g=\partial H_g$, \quad and \quad $S^2=\partial B^3$ (see Figure~\ref{fig_bs_periodic_map}).
\end{itemize}  
Then we also call the restriction $p|_{\Sigma _g}\colon \Sigma _g\to S^2$ the balanced superelliptic covering map and we often denote simply $p|_{\Sigma _g}=p$. 
We note that the branch points set of $p\colon \Sigma _g\to S^2$ coincides with $\B $. 

At the end of this subsection, we will define simple closed curves which are used frequently throughout this paper. 
Let $l_i$ $(1\leq i\leq 2n+1)$ be an oriented simple arc on $S^2$ whose endpoints are $p_i$ and $p_{i+1}$ as in Figure~\ref{fig_path_l}, and let $\widetilde{l}_i^l$ for $1\leq i\leq 2n+1$ and $1\leq l\leq k$ be a lift of $l_i$ with respect to $p$ such that $\zeta (\widetilde{l}_i^l)=\widetilde{l}_i^{l+1}$ for  $1\leq l\leq k-1$ and $\zeta (\widetilde{l}_i^k)=\widetilde{l}_i^{1}$ (see Figure~\ref{fig_isotopy_surface_3-handles}). 
Put $L=l_1\cup l_2\cup \cdots \cup l_{2n+1}\subset S^2$ and $\widetilde{L}=p^{-1}(L)\subset \Sigma _g$. 
Note that the isotopy class of a homeomorphism $\varphi $ on $H_g$ (resp. on $B^3$ relative to $\A $) is determined by the isotopy class of $\varphi (\widetilde{L})$ in $\Sigma _g$ (resp. of $\varphi (L)$ in $S^2$ relative to  $\B $). 
We identify $B^3$ with the 3-manifold with a sphere boundary on the lower side in Figure~\ref{fig_path_l}, and also identify $H_g$ with the 3-manifold with boundary as on the lower side in Figure~\ref{fig_isotopy_surface_3-handles} by some homeomorphisms. 

\begin{figure}[h]
\includegraphics[scale=1.3]{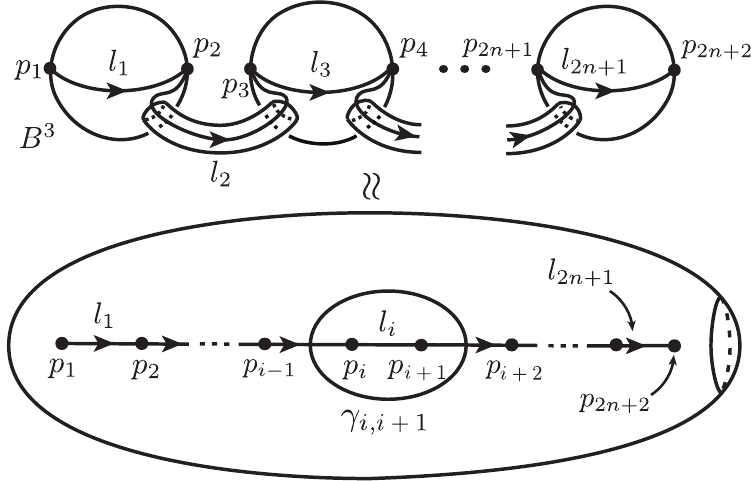}
\caption{Arcs $l_1,\ l_2,\ \dots ,\ l_{2n+1}$ and a simple closed curve $\gamma _{i,i+1}$ on $S^2$ for $1\leq i\leq 2n+1$, and a natural homeomorphism of $B^3$.}\label{fig_path_l}
\end{figure}

\begin{figure}[h]
\includegraphics[scale=0.85]{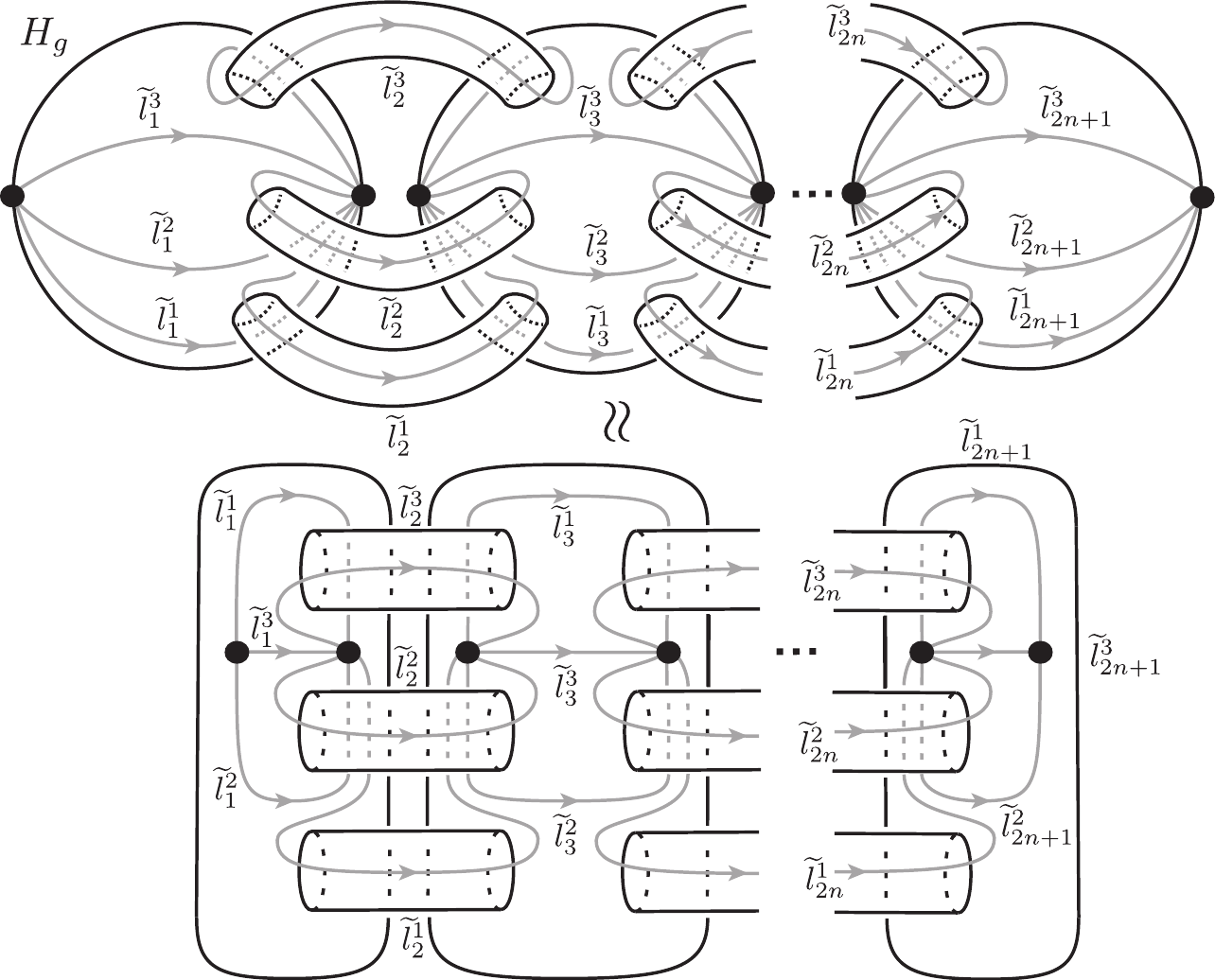}
\caption{A natural homeomorphism of $H_g$ when $k=3$.}\label{fig_isotopy_surface_3-handles}
\end{figure}

Let $\gamma _{i,i+1}$ for $1\leq i\leq 2n+1$ be a simple closed curve on $S^2-\B $ which surrounds the two points $p_i$ and $p_{i+1}$ as in Figure~\ref{fig_path_l} and $\gamma _{i}^l$ for $1\leq i\leq 2n+1$ and $1\leq l\leq k$ a simple closed curve on $\Sigma _g$ which is a lift of $\gamma _{i,i+1}$ and is equipped with $\zeta (\gamma _{i}^l)=\gamma _{i}^{l+1}$ for $1\leq l\leq k-1$ as in Figure~\ref{fig_scc_c_il}. 
We remark that $\gamma _{2i-1}^l$ for $1\leq i\leq n+1$ and $1\leq l\leq k$ bounds a disk in $H_g$. 

\begin{figure}[h]
\includegraphics[scale=0.85]{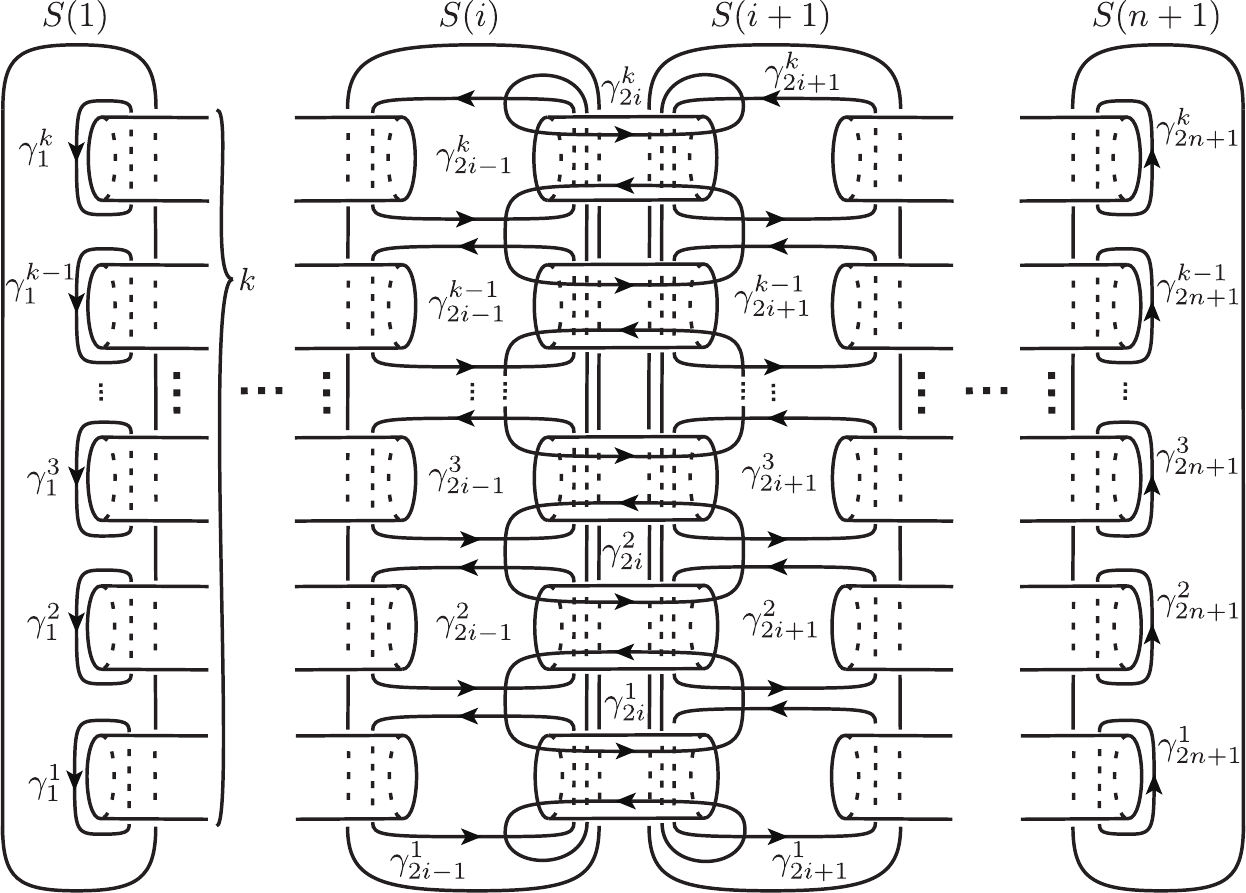}
\caption{Simple closed curves $\gamma _i^l$ on $\Sigma _g$ for $1\leq i\leq 2n+1$ and $1\leq l\leq k$.}\label{fig_scc_c_il}
\end{figure}

\subsection{Notes on the Birman-Hilden correspondence for the balanced superelliptic handlebody group}\label{section_notes_BH-corresp}

In this section, we review the Birman-Hilden correspondence for the balanced superelliptic handlebody group as an analogy of Section~1.8 in \cite{Iguchi-Hirose-Kin-Koda}. 
Let $p=p_{g,k}\colon H_g\to B^3$ be the balanced superelliptic covering map for fixed $g=n(k-1)$ with $k\geq 3$ and $n\geq 1$, and $\zeta $ the balanced superelliptic rotation on $H_g$. 
Then, we denote by $\SHomH $ the group of orientation-preserving symmetric homeomorpoisms on $H_g$ for $\zeta $, and by $\mathrm{Homeo}_+(B^3, \A )$ the group of orientation-preserving self-homeomorpoisms on $B^3$ which preserve $\A $ setwise. 
Remark that $\pi _0(\mathrm{Homeo}_+(B^3, \A ))=\Hil $. 
Since $\SHomH $ preserves $\widetilde{\A }=p^{-1}(\A )$ setwise, the natural homorphism $\SHomH \to \mathrm{Homeo}_+(B^3, \A )$ induces the surjective homomorphism 
\begin{eqnarray}\label{theta_bar}
\bar{\theta }\colon \pi _0(\SHomH )\to \LH 
\end{eqnarray}
with the kernel $\left< \zeta \right>$. 
Then, by an analogy of Lemma~1.21 in \cite{Iguchi-Hirose-Kin-Koda}, we have the following lemma. 

\begin{lem}\label{lem_interior_symm-liftability}
\begin{enumerate}
\item Let $\varphi $ be a symmetric homeomorphism on $\Sigma _g$ for $\zeta |_{\Sigma _g}$ which extends to $H_g$. 
Then there exists a symmetric homeomorphism $\hat{\varphi }$ on $H_g$ for $\zeta $ such that $\hat{\varphi }|_{\Sigma _g}=\varphi $.  
\item If symmetric homeomorphisms $\varphi _1$ and $\varphi _2$ on $H_g$ are isotopic, then $\varphi _1$ and $\varphi _2$ are symmetrically isotopic with respect to $\zeta $.  
\end{enumerate}
\end{lem}

\begin{proof}
(1): 
For $1\leq i\leq n+1$ and $1\leq l\leq k$, we take a proper disk $D_{2i-1}^l$ in $H_g$ such that $\partial D_{2i-1}^l=\gamma _{2i-1}^l$, $\zeta (D_{2i-1}^l)=D_{2i-1}^{l+1}$ for $1\leq l\leq k-1$, and $D_{2i-1}^l$ is disjoint from $\widetilde{\A }$ and $D_{2j-1}^{l^\prime }$ for $(i,l)\not= (j,l^\prime )$. 
Note that the fixed points set of $\zeta $ (resp. $\zeta |_{\Sigma _g}$) is $\widetilde{\A }$ (resp. $\widetilde{\B }=p^{-1}(\B )$). 
Since $\varphi $ is symmetric for $\zeta |_{\Sigma _g}$, the group $\left< \zeta |_{\Sigma _g}\right> $ preserves the set $\widetilde{\B } $ and acts on the set of disjoint simple closed curves $\{ \varphi (\gamma _{2i-1}^1),\dots , \varphi (\gamma _{2i-1}^k)\}$ as a bijection. 
Since $\varphi $ extends to $H_g$, the curve $\varphi (\gamma _{2i-1}^l)$ for $1\leq i\leq n+1$ and $1\leq l\leq k$ is null-homotopic in $H_g$. 
By Edmonds~\cite{Edmonds}, there exist proper disks $\bar{D}_{2i-1}^1$ $(1\leq i\leq n+1)$ such that $\partial \bar{D}_{2i-1}^1=\varphi (\gamma _{2i-1}^1)$, the disk $\bar{D}_{2i-1}^1$ is disjoint from $\widetilde{\A }$, and the disks $\bar{D}_{2i-1}^1, \zeta (\bar{D}_{2i-1}^1), \dots , \zeta ^{k-1}(\bar{D}_{2i-1}^1)$ are mutually disjoint.  

Put $\bar{D}_{2i-1}^l=\zeta ^{l-1}(\bar{D}_{2i-1}^1)$ for $2\leq l\leq k$, and $E_{2i-1}=p(D_{2i-1}^1)(=p(D_{2i-1}^l))$ and $\bar{E}_{2i-1}=p(\bar{D}_{2i-1}^1)(=p(\bar{D}_{2i-1}^l))$ for $1\leq i\leq n+1$. 
Since $\partial \bar{E}_{2i-1}$ is disjoint from $\partial \bar{E}_{2j-1}$ for $i\not =j$, by the irreducibility of $B^3-\A $, we retake $\bar{E}_{2i-1}$ such that $\bar{E}_{1}, \bar{E}_{3}, \dots , \bar{E}_{2n+1}$ are mutually disjoint by using an isotopy in $B^3-\A $ relative to $S^2$. 

Since $\varphi $ preserves $\widetilde{\B }$, $\varphi $ induces the self-homeomorphism $\psi $ on $B^3$ which preserves $\B $. 
By the definition, for each $1\leq i\leq n+1$, the disk $E_{2i-1}$ cuts off from $B^3$ a 3-ball which contains the arc $a_i$. 
Since $\bar{D}_{2i-1}^l$ is disjoint from $\widetilde{\A }$ and $\varphi (\gamma _{2i-1}^1)\cup \cdots \cup \varphi (\gamma _{2i-1}^k)$ cut off a punctured 2-sphere which contains the points $\widetilde{p}_{2i-1}$ and $\widetilde{p}_{2i}$, the disk $\bar{E}_{2i-1}$ also cuts off from $B^3$ a 3-ball which contains the arc $\psi (a_i)$. 
Thus, there exists $\hat{\psi }\in \mathrm{Homeo}_+(B^3, \A )$ such that $\hat{\psi }(E_{2i-1})=\bar{E}_{2i-1}$ and $\hat{\psi }|_{S^2}=\psi $.  
Therefore, by the sujectivity of $\bar{\theta }\colon \pi _0(\SHomH )\to \LH $, there exists a lift $\hat{\varphi }\in \SHomH $ of $\hat{\psi }$ such that $\varphi =\hat{\varphi }|_{\Sigma _g}$. 

(2): Let $\psi _i\in \mathrm{Homeo}_+(B^3, \A )$ for $i=1,2$ be the element which is induced by $\varphi _i$. 
Since $\varphi _1|_{\Sigma _g}$ and $\varphi _2|_{\Sigma _g}$ are isotopic, $\varphi _1|_{\Sigma _g}$ and $\varphi _2|_{\Sigma _g}$ are symmetrically isotopic by Birman and Hilden~\cite{Birman-Hilden3}. 
The isotopy between $\varphi _1|_{\Sigma _g}$ and $\varphi _2|_{\Sigma _g}$ induces the isotopy between $\psi _1|_{S^2}$ and $\psi _2|_{S^2}$. 
By Proposition~A.4 in~\cite{Hirose-Kin}, this isotopy extends to the isotopy between $\psi _1$ and $\psi _2$. 
Then, the isotopy between $\psi _1$ and $\psi _2$ uniquely lifts to the isotopy in $\SHomH $ whose origin is $\varphi _1$. 
The terminal of this isotopy coincides with $\zeta ^l\varphi _2$ for some $0\leq l\leq k-1$.  
Since $\varphi _1$ is isotopic to $\varphi _2$, if $1\leq l\leq k-1$, then $\varphi _1$ is not isotopic to $\zeta ^l\varphi _2$. 
Therefore, we have $l=0$, namely $\varphi _1$ is symmetrically isotopic to $\varphi _2$, and we have completed the proof of Lemma~\ref{lem_interior_symm-liftability}. 
\end{proof}

By Lemma~\ref{lem_interior_symm-liftability}, the group $\pi _0(\SHomH )$ is isomorphic to $\SH $ and we have the surjective homomorphism
\[
\bar{\theta }\colon \SH \to \LH .
\]
We regard $\SH $ and $\Hil$ as the subgroups of $\SM $ and $\M$, respectively, by the natural injections $\SH \hookrightarrow \SM $ and $\Hil \hookrightarrow \M $. 
Since if $\varphi \in \mathrm{Homeo}_+(B^3, \A )$ is liftable with respect to $p$, then $\varphi |_{S^2}$ is also liftable with respect to $p|_{\Sigma _g}$, we have $\LH \subset \LM \cap \Hil $. 
By Lemma~\ref{lem_interior_symm-liftability}~(1), if $\varphi \in \mathrm{Homeo}_+(B^3, \A )$ lifts to a self-homeomolphism $\widetilde{\varphi }$ on $H_g$ 
such that $\widetilde{\varphi }|_{\Sigma _g}$ is symmetric for $\zeta |_{\Sigma _g}$, then $\varphi $ has a lift $\widetilde{\varphi }^\prime \in \SHomH $. 
Thus, we have the following lemma:  

\begin{lem}\label{liftability_surf_hand}
For $n \geq 1$, we have $\LH =\LM \cap \Hil $.
\end{lem}

By Birman and Hilden~\cite{Birman-Hilden2}, we have the following exact sequence: 
\begin{eqnarray*}\label{exact_BH}
1\longrightarrow \left< \zeta \right> \longrightarrow \SM \stackrel{\theta }{\longrightarrow }\LM \longrightarrow 1. 
\end{eqnarray*}
By restricting this exact sequence to $\SH $, we have the following lemma: 

\begin{lem}\label{lem_exact_SH_handlebody}
We have the following exact sequence: 
\begin{eqnarray}\label{exact_SH_handlebody}
1\longrightarrow \left< \zeta \right> \longrightarrow \SH \stackrel{\theta }{\longrightarrow }\LH \longrightarrow 1. 
\end{eqnarray}
\end{lem}

\subsection{The liftable condition for the balanced superelliptic covering map}\label{section_liftable-condition}

In this section, we review Ghaswala-Winarski's necessary and sufficient condition for lifting a homeomorphism on $S^2$ with respect to $p=p_{g,k}$ for $k\geq 3$ in Lemma~3.6 of \cite{Ghaswala-Winarski1}. 

Let $l$ be a simple arc on $S^2$ whose endpoints lie in $\B $. 
A regular neighborhood $\mathcal{N}$ of $l$ in $S^2$ is homeomorphic to a 2-disk. 
Then the half-twist $\sigma [l]$ is a self-homeomorphism on $S^2$ which is described as the result of anticlockwise half-rotation of $l$ in $\mathcal{N}$ as in Figure~\ref{fig_sigma_l}. 
Recall that $l_i$ $(1\leq i\leq 2n+1)$ is an oriented simple arc on $S^2$ whose endpoints are $p_i$ and $p_{i+1}$ as in Figure~\ref{fig_path_l}. 
We define $\sigma _i=\sigma [l_i]$ for $1\leq i\leq 2n+1$. 
As a well-known result, $\M $ is generated by $\sigma _1$, $\sigma _2,\ \dots $, $\sigma _{2n+1}$. 
Let $\mathrm{Map}(\B )$ be the group of self-bijections on $\B $. 
Since $\mathrm{Map}(\B )$ is naturally identified with the symmetric group $S_{2n+2}$ of degree $2n+2$, the action of $\M $ on $\B$ induces the surjective homomorphism
\[
\Psi \colon \M \to S_{2n+2}
\]
given by $\Psi (\sigma _i)=(i\ i+1)$, where for maps or mapping classes $f$ and $g$, the product $gf$ means that $f$ apply first.  

\begin{figure}[h]
\includegraphics[scale=1.1]{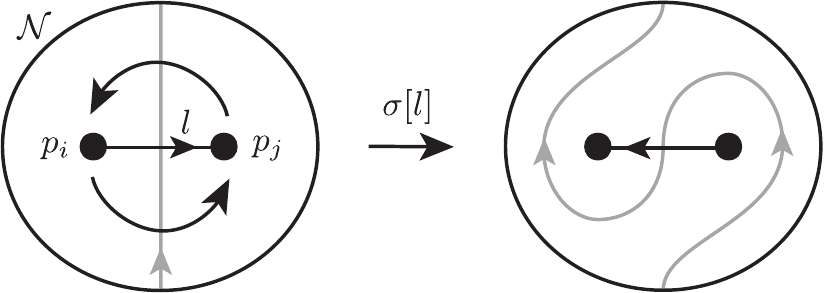}
\caption{The half-twist $\sigma [l]$ along an arc $l$.}\label{fig_sigma_l}
\end{figure}

Put $\B _o=\{ p_1,\ p_3,\ \dots ,\ p_{2n+1}\}$ and $\B _e=\{ p_2,\ p_4,\ \dots ,\ p_{2n+2}\}$. 
An element $\sigma $ in $S_{2n+2}$ is \textit{parity-preserving} if $\sigma (\B _o)=\B _o$, and is \textit{parity-reversing} if $\sigma (\B _o)=\B _e$. 
An element $f$ in $\M $ is \textit{parity-preserving} (resp. \textit{parity-reversing}) if $\Psi (f)$ is \textit{parity-preserving} (resp. \textit{parity-reversing}). 
Let $W_{2n+2}$ be the subgroup of $S_{2n+2}$ which consists of parity-preserving or parity-reversing elements, $S_{n+1}^o$ (resp. $S_{n+1}^e$) the subgroup of $S_{2n+2}$ which consists of elements whose restriction to $\B _e$ (resp. $\B _o$) is the identity map.  
Note that $S_{n+1}^o$ (resp. $S_{n+1}^e$) is a subgroup of $W_{2n+2}$, is isomorphic to $S_{n+1}$, and is generated by transpositions $(1\ 3)$, $(3\ 5),\ \dots $, $(2n-1\ 2n+1)$ (resp. $(2\ 4)$, $(4\ 6),\ \dots $, $(2n\ 2n+2)$). 
Then we have the following exact sequence:
\begin{eqnarray}\label{exact1}
1\longrightarrow S_{n+1}^o\times S_{n+1}^e\longrightarrow W_{2n+2}\stackrel{\pi }{\longrightarrow }\Z _2\longrightarrow 1, 
\end{eqnarray}
where the homomorphism $\pi \colon W_{2n+2}\to \Z _2$ is defined by $\pi (\sigma )=0$ if $\sigma $ is parity-preserving and $\pi (\sigma )=1$ if $\sigma $ is parity-reversing. 
Ghaswala and Winarski~\cite{Ghaswala-Winarski1} proved the following lemma. 

\begin{lem}[Lemma~3.6 in \cite{Ghaswala-Winarski1}]\label{lem_GW}
Let $\LM $ be the liftable mapping class group for the balanced superelliptic covering map $p_{g,k}$ for $n\geq 1$ and $k\geq3$ with $g=n(k-1)$. 
Then we have
\[
\LM =\Psi ^{-1}(W_{2n+2}).
\]
\end{lem}
Lemma~\ref{lem_GW} implies that a mapping class $f\in \M $ lifts with respect to $p_{g,k}$ if and only if $f$ is parity-preserving or parity-reversing (in particular, when $k\geq 3$, the liftability of a homeomorphism on $S^2$ does not depend on $k$). 

We regard $\Hil $ as the subgroup of $\M $. 
By Lemmas~\ref{liftability_surf_hand} and \ref{lem_GW}, we have the following lemma: 

\begin{lem}\label{liftable_condition_handlebody}
For $f\in \Hil $, the element $f$ lies in $\LH $ if and only if $\Psi (f)\in W_{2n+2}$. 
\end{lem}

\subsection{The spherical wicket group and the pure subgroups of the braid group and the mapping class group}\label{section_wicket-group}

In this section, we review the spherical wicket group and the pure subgroups of the braid group and the mapping class group to define explicit elements of the liftable Hilden group $\LH $ in the next section. 

Let $SB_{2n+2}$ be the \textit{spherical braid group} of $2n+2$ strands. 
We regard an element in $SB_{2n+2}$ as a $(2n+2)$-tangle in $S^2\times [0,1]$ which consists of $2n+2$ simple proper arcs whose one of the endpoints lies in $\B \times \{ 0\}$ and the other one lies in $\B \times \{ 1\}$, and we also regard $\A $ is a $(n+1)$-tangle in $B^3 $. 
Such $\A $ is called a \textit{wicket}. 
For $b\in SB_{2n+2}$, denote by $^b\! \A$ the $(n+1)$-tangle in $B^3 $ which is obtained from $b$ by attaching a copy of $B^3 $ with $\A $ to $S^2\times \{ 0\}$ such that $p_i\in \partial B^3$ is attached to the end of $i$-th strand in $b$ (see Figure~\ref{fig_wicket}), where we regard $S^2\times [0,1]$ attached the copy of $B^3 $ as $B^3$. 
The \textit{spherical wicket group} $SW_{2n+2}$ is the subgroup of $SB_{2n+2}$ whose element $b$ satisfies the condition that $^b\! \A$ is isotopic to $\A $ relative to $\partial B^3=S^2\times \{ 1\}$. 
Brendle and Hatcher~\cite{Brendle-Hatcher} introduced the group $SW_{2n+2}$ and gave its finite presentation. 

\begin{figure}[h]
\includegraphics[scale=0.5]{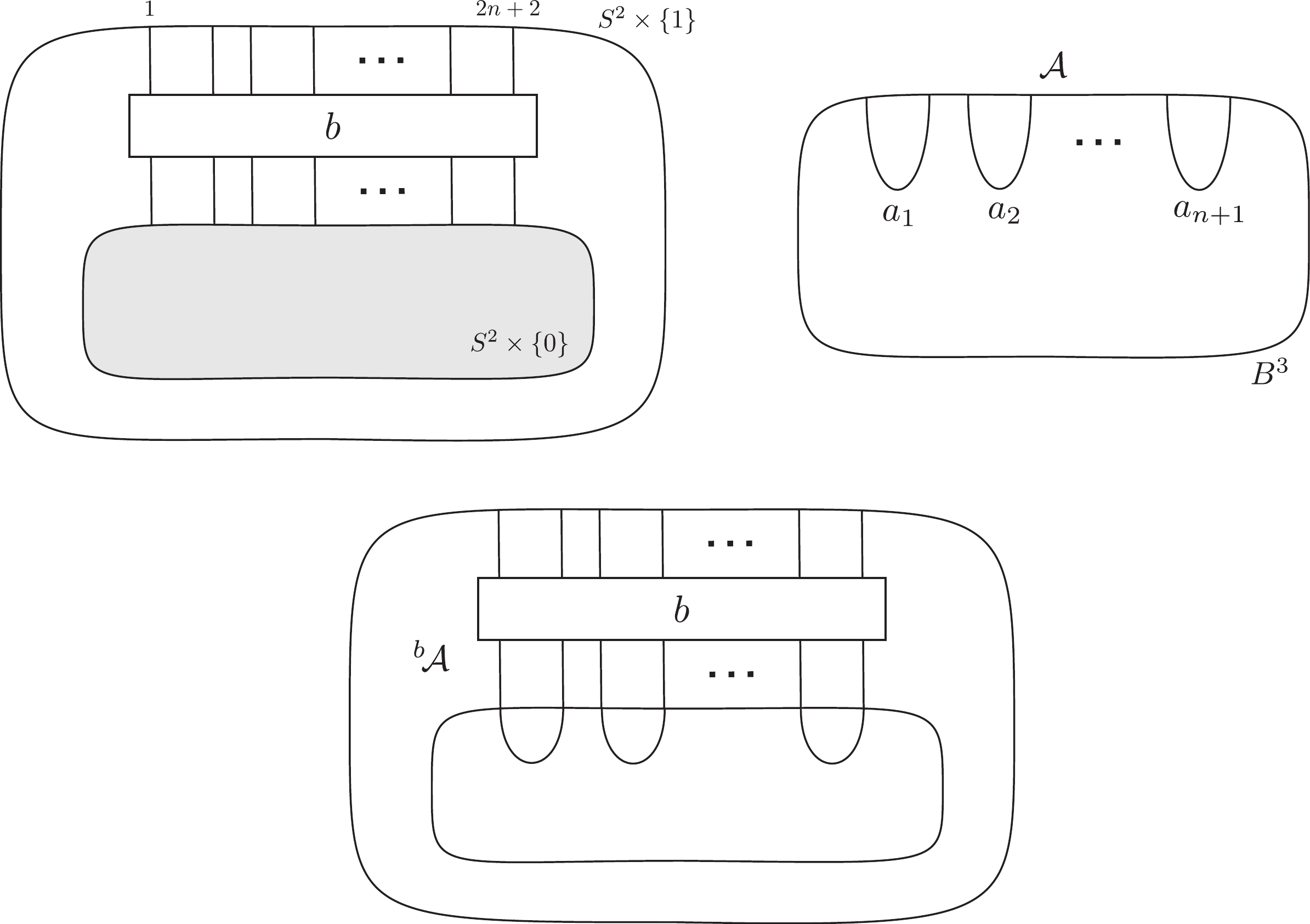}
\caption{A spherical braid $b$ and the tangles $\A $ and $^b\! \A $ in $B^3$.}\label{fig_wicket}
\end{figure}

For $b_1$, $b_2\in SB_{2n+2}$, the product $b_1b_2$ is a braid as on the right-hand side in Figure~\ref{fig_braid_product_def}. 
Then we have the surjective homomorphism 
\[
\Gamma \colon SB_{2n+2}\to \M 
\]
which maps the half-twist about $i$-th and $(i+1)$-st strands as on the left-hand side in Figure~\ref{fig_braid_product_def} to the half-twist $\sigma _i\in \M $.
$\Gamma $ has the kernel with order 2 which is generated by the full twist braid (see for instance Section~9.1.4 in \cite{Farb-Margalit}). 
We abuse notation and denote simply the half-twist in $SB_{2n+2}$ about $i$-th and $(i+1)$-st strands by $\sigma _i$. 
By Theorem~2.6 in \cite{Hirose-Kin}, we have $\Gamma (SW_{2n+2})=\Hil $. 

\begin{figure}[h]
\includegraphics[scale=0.75]{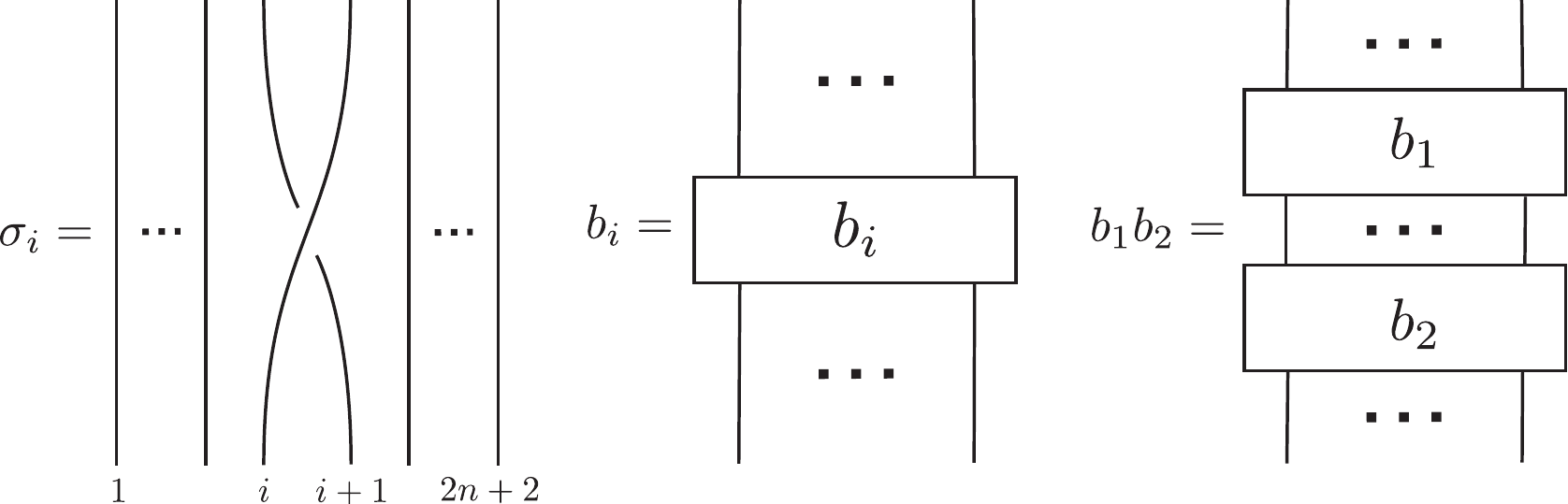}
\caption{The half-twist $\sigma_i\in SB_{2n+2}$ $(1\leq i\leq 2n+1)$ and the product $b_1b_2$ in $SB_{2n+2}$ for $b_i$ $(i=1,\ 2)$.}\label{fig_braid_product_def}
\end{figure}

The \textit{pure mapping class group} $\PM $ is the kernel of the homomorphism $\Psi \colon \M \to S_{2n+2}$. 
Since all elements in $\PM $ is parity-preserving, $\PM $ is a subgroup of $\LM $ and we have the following exact sequence:  
\begin{eqnarray}\label{exact2}
1\longrightarrow \PM \longrightarrow \LM \stackrel{\Psi }{\longrightarrow }W_{2n+2}\longrightarrow 1. 
\end{eqnarray}
Denote by $PSB_{2n+2}$ the kernel of $\Psi \circ \Gamma \colon SB_{2n+2}\to S_{2n+2}$, $PSW_{2n+2}=SW_{2n+2}\cap PSB_{2n+2}$, and $\PH =\Hil \cap \PM $. 
We call the groups $\PH $, $PSB_{2n+2}$, $PSW_{2n+2}$ the \textit{pure Hilden group}, the \textit{pure spherical braid group}, and the \textit{pure spherical wicket group}, respectively. 
By the fact that $\Gamma (SW_{2n+2})=\Hil $ 
and $\Gamma (PSB_{2n+2})=\PM $, we show that $\Gamma (PSW_{2n+2})=\PH $, in particular, we remark that $\PH \subset \LH$.

\subsection{Liftable elements for the balanced superelliptic covering map}\label{section_liftable-element}

In this section, we introduce some liftable homeomorphisms on $B^3$ with respect to $p=p_{g,k}\colon H_g\to B^3$ for $k\geq 3$ as images of elements in $SW_{2n+2}$ by $\Gamma $. 
We often abuse notation and denote a homeomorphism and its isotopy class by the same symbol.  
We review the generators for $SW_{2n+2}$ by Brendle and Hatcher~\cite{Brendle-Hatcher}. 

Let $s_i$ and $r_i$ for $1\leq i\leq n$ be the $2n+2$ strands braids as in Figure~\ref{fig_r_i-s_i}. 
We can show that $s_i$ and $r_i$ for $1\leq i\leq n$ are elements in $SW_{2n+2}$. 
Remark that the relations 
\begin{eqnarray*}
s_i=\sigma _{2i}\sigma _{2i+1}\sigma _{2i-1}\sigma _{2i} \quad  \text{and} \quad r_i=\sigma _{2i}^{-1}\sigma _{2i+1}^{-1}\sigma _{2i-1}\sigma _{2i}
\end{eqnarray*}
for $1\leq i\leq n$ hold in $SB_{2n+2}$. 
Generators of Brendle-Hatcher's finite presentation for $SW_{2n+2}$ in~\cite{Brendle-Hatcher} are $s_i$, $r_i$ for $1\leq i\leq n$, and $\sigma _{2i-1}$ for $1\leq i\leq n+1$. 
We abuse notation and denote simply $\Gamma (b)$ for $b\in SB_{2n+2}$ by $b$ (i.e. we express $s_i, r_i\in \Hil $). 
Since we can check that $\Psi (s_i)=\Psi (r_i)=(2i-1\ 2i+1)(2i\ 2i+2)$ and $s_i$ and $r_i$ are parity-preserving for $1\leq i\leq n$, by Lemma~\ref{liftable_condition_handlebody}, we have the following lemma. 

\begin{figure}[h]
\includegraphics[scale=0.75]{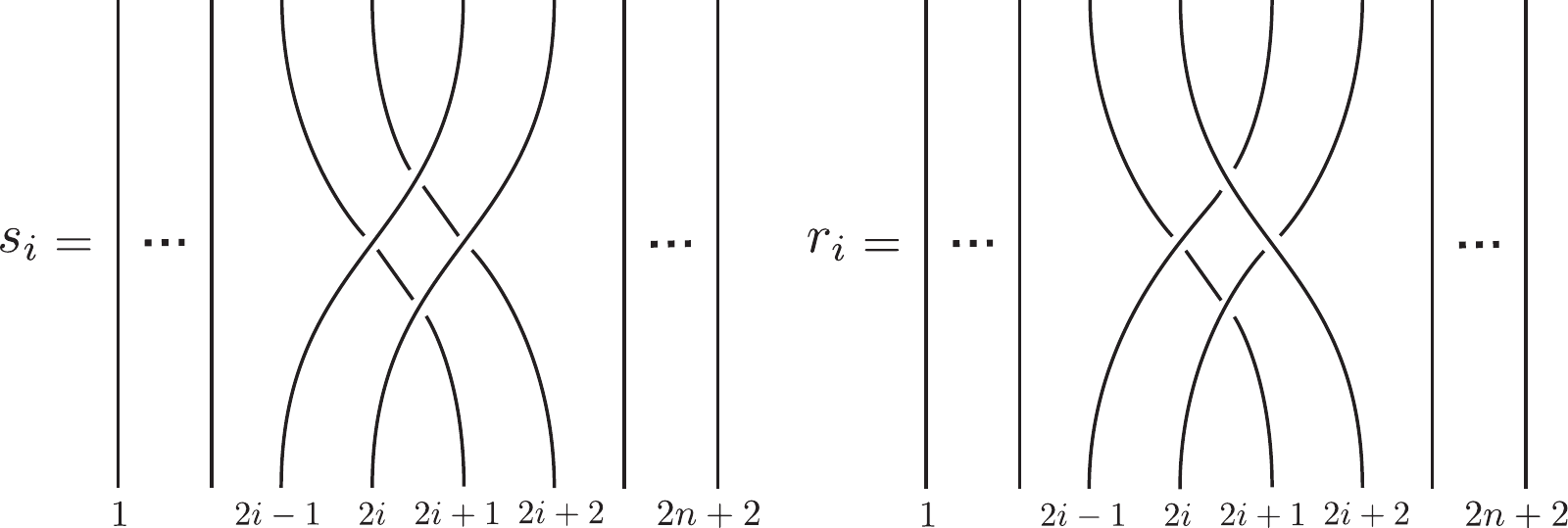}
\caption{The braids $s_i$ and $r_i$ in $SB_{2n+2}$ for $1\leq i\leq n$.}\label{fig_r_i-s_i}
\end{figure}

\begin{lem}\label{lem_s_i-r_i_liftable}
For $1\leq i\leq n$, $s_i$ and $r_i$ are liftable with respect to $p$, namely, $s_i$ and $r_i$ lie in $\LH $. 
\end{lem} 

For each $1\leq i\leq n+1$, $\sigma _{2i-1}\in SB_{2n+2}$ lies in $SW_{2n+2}$, however, $\sigma _{2i-1}\in \Hil $ is not parity-preserving and not parity-reversing. 
Denote $r=\sigma _1\sigma _3\cdots \sigma _{2n+1}\in \Hil $. 
Since $\Psi (r)=(1\ 2)(3\ 4)\cdots (2n+1\ 2n+2)$, the mapping class $r$ is parity reversing. 
Thus, by Lemma~\ref{liftable_condition_handlebody}, we have the following lemma. 

\begin{lem}\label{lem_r_liftable}
The mapping class $r$ is liftable with respect to $p$, namely, $r$ lies in $\LH $. 
\end{lem}

For a simple closed curve $\gamma $ on $\Sigma _g$ $(g\geq 0)$, we denote by $t_\gamma $ the right-handed Dehn twist along $\gamma $. 
Recall that $\gamma _{i,i+1}$ for $1\leq i\leq 2n+1$ is a simple closed curve on $S^2$ as in Figure~\ref{fig_path_l}. 
Then we define $t_{i}=t_{\gamma _{2i-1,2i}}$ for $1\leq i\leq n+1$. 
Since we have $t_i=\sigma _{2i-1}^2$ and $t_{i}$ preserves $\B$ pointwise (i.e. $t_{i,j}$ lies in $\PH $), by Lemma~\ref{liftable_condition_handlebody}, we have the following lemma. 

\begin{lem}\label{lem_t_i_liftable}
For $1\leq i\leq n+1$, $t_i$ is liftable with respect to $p$, moreover, $t_i$ lies in $\PH $. 
\end{lem} 

For $1\leq i,j\leq n+1$ and $i\not =j$, let $p_{i,j}=p_{j,i}$, $x_{i,j}=x_{j,i}$, and $y_{i,j}=y_{j,i}$ be elements in $SB_{2n+2}$ as in Figure~\ref{fig_p_ij-x_ij-y_ij} (or its image by $\Gamma$). 
We can see that $^{p_{i,j}}\! \A $, $^{x_{i,j}}\! \A $, and $^{y_{i,j}}\! \A $ are isotopic to $\A $ relative to $\partial B^3$ and lie in $PSB_{2n+2}$. 
Thus, by Lemma~\ref{liftable_condition_handlebody}, we have the following lemma. 

\begin{figure}[h]
\includegraphics[scale=0.65]{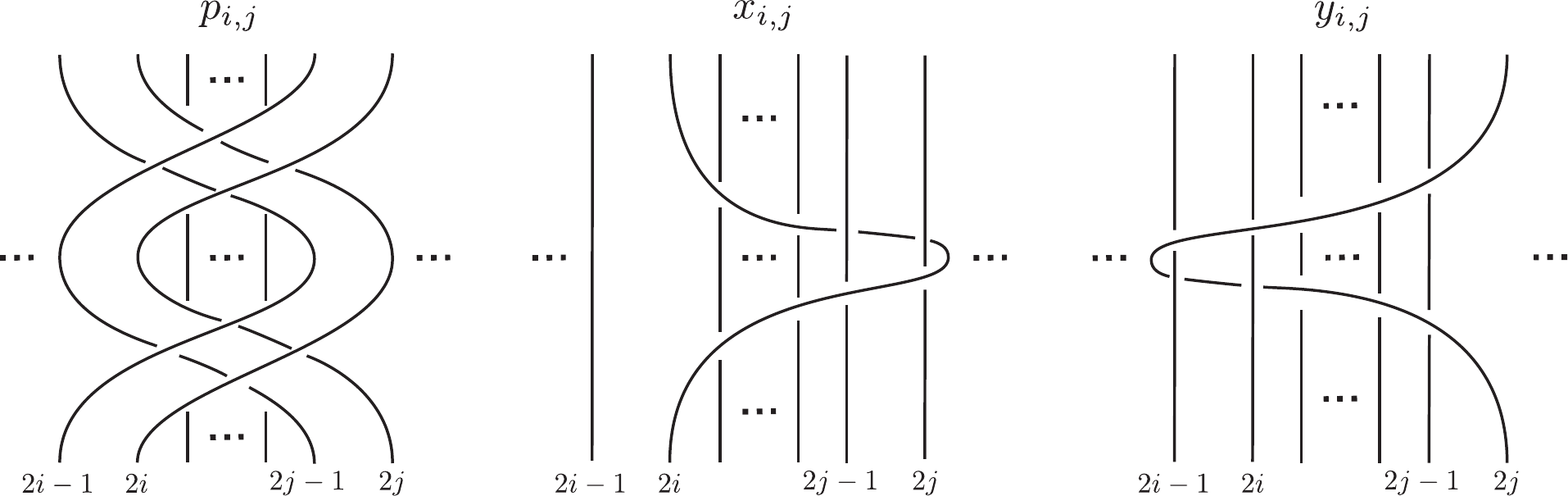}
\caption{The braids $p_{i,j}$, $x_{i,j}$, and $y_{i,j}$ in $SB_{2n+2}$ for $1\leq i,j\leq n+1$ and $i\not =j$.}\label{fig_p_ij-x_ij-y_ij}
\end{figure}

\begin{lem}\label{lem_p_ij_liftable}
For $1\leq i,j\leq n+1$ and $i\not=j$, $p_{i,j}$, $x_{i,j}$, and $y_{i,j}$ are liftable with respect to $p$, moreover, $p_{i,j}$, $x_{i,j}$, and $y_{i,j}$ lie in $\PH $. 
\end{lem}

\subsection{Basic relations among liftable elements for the balanced superelliptic covering map}\label{section_relations_liftable-elements}

In this section, we review relations in $\LH $ among liftable elements 
introduced in Section~\ref{section_liftable-element}. 
Recall that for mapping classes $f$ and $g$, the product $gf$ means that $f$ apply first. 

\paragraph{\emph{Commutative relations}}

For elements $f$ and $h$ in a group $G$, $f\rightleftarrows h$ if $f$ commutes with $h$, namely, the relation $fh=hf$ holds in $G$. 
We call such a relation a \textit{commutative relation}. 
For mapping classes $f$ and $h$, if their representatives have mutually disjoint supports, then we have $f\rightleftarrows h$ in the mapping class group. 
We have the following lemma. 

\begin{lem}\label{lem_comm_rel}
For $n\geq 1$, the following relations hold in $SW_{2n+2}$ and $\Hil $:
\begin{enumerate}
\item $\alpha _i \rightleftarrows \beta _j$ \quad for $j-i\geq 2$ and $\alpha , \beta \in \{ s, r\}$,
\item $\alpha _{i} \rightleftarrows t_{j}$ \quad for $j\not =i, i+1$ and $\alpha \in \{ s, r\}$, 
\item $t_i\rightleftarrows t_{j}$ \quad for $1\leq i<j\leq n+1$, 
\item $s_i \rightleftarrows r$ \quad for $1\leq i\leq n$, 
\item $t_i \rightleftarrows r$ \quad for $1\leq i\leq n+1$. 
\end{enumerate}
\end{lem}

\paragraph{\emph{Conjugation relations}}

For elements $f_1$, $f_2$, and $h$ in a group $G$, we call the relation $hf_1h^{-1}=f_2$, which is equivalent to $hf_1=f_2h$, a \textit{conjugation relation}. 
In particular, in the case of $f_1=t_\gamma $ or $f_1=\sigma [l]$, and $h$ is a mapping class, the relations $ht_\gamma =t_{h(\gamma )}h$ and $h\sigma [l]=\sigma [h(l)]h$ hold in the mapping class group. 
By these relations or calculations of braids, we have the following lemma (for the relations (4) and (7) in Lemma~\ref{lem_conj_rel}, see Figure~\ref{fig_proof_lem_conj-rel}). 

\begin{lem}\label{lem_conj_rel}
For $n\geq 1$, the following relations hold in $SW_{2n+2}$ and $\Hil $:
\begin{enumerate}
\item $\alpha _i\alpha _{i+1}\alpha _i=\alpha _{i+1}\alpha _i\alpha _{i+1}$ \quad for $1\leq i\leq n-1$ and $\alpha \in \{ s, r\}$, 
\item $s_i^{\varepsilon }s_{i+1}^{\varepsilon }r_{i}=r_{i+1}s_i^{\varepsilon }s_{i+1}^{\varepsilon }$ \quad for $1\leq i\leq n-1$ and $\varepsilon \in \{ 1, -1\}$,
\item $r_ir_{i+1}s_{i}=s_{i+1}r_ir_{i+1}$ \quad for $1\leq i\leq n-1$,
\item $r_irs_{i}=rs_{i}r_{i}^{-1}$ \quad for $1\leq i\leq n$,
\item $s_i^{\varepsilon }t_{i}=t_{i+1}s_{i}^{\varepsilon }$ \quad for $1\leq i\leq n$ and $\varepsilon \in \{ 1, -1\}$,
\item $r_it_{i}=t_{i+1}r_{i}$ \quad for $1\leq i\leq n$,
\item $t_is_{i}^2r_i=r_is_i^2t_{i+1}$ \quad for $1\leq i\leq n$.
\end{enumerate}
\end{lem}

\begin{figure}[h]
\includegraphics[scale=0.9]{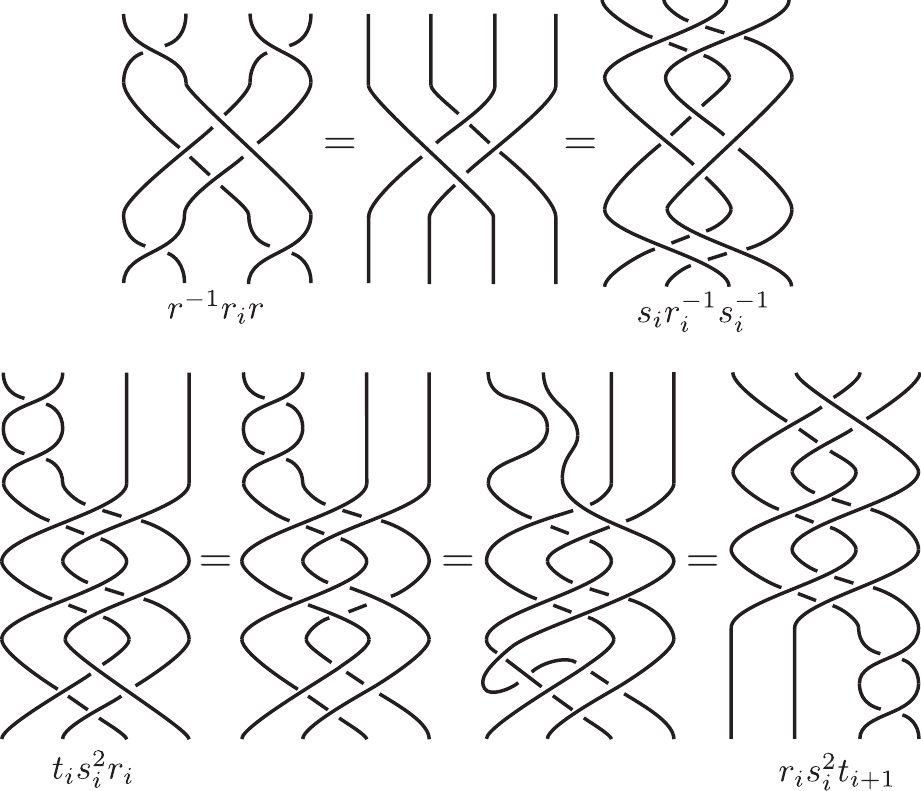}
\caption{The relations (4) and (7) in Lemma~\ref{lem_conj_rel}.}\label{fig_proof_lem_conj-rel}
\end{figure}

\paragraph{\emph{A Lift of a relation in the symmetric group}}

Recall that we have the surjective homomorphism $\Psi \colon \LM \to W_{2n+2}$ by Lemma~\ref{lem_GW} and $\LH$ is a subgroup of $\LM $. 
Since $\PH =\ker \Psi |_{\LH }$ and $\Psi (r^2)=1$ in $W_{2n+2}\subset S_{2n+2}$, we can express $r^2$ by a product of elements in $\PH $. 
The next lemma follows from some calculations in $SW_{2n+2}$. 

\begin{lem}\label{lem_lift_W}
For $n\geq 1$, the relation 
\[
r^2=t_1t_2\cdots t_{n+1}
\]
holds in $SW_{2n+2}$ and $\Hil $. 
\end{lem}

\paragraph{\emph{Relations among pure elements and non-pure elements}}

The next lemma follows from some calculations in $SW_{2n+2}$ (see Figure~\ref{fig_proof_lem_pxy}). 

\begin{lem}\label{lem_p_ij_relations}
For $n\geq 1$, the following relations hold in $SW_{2n+2}$ and $\Hil $:
\begin{enumerate}
\item $p_{i,i+1}=s_{i}^2$, \quad $x_{i,i+1}=s_{i}r_i^{-1}$, \quad $y_{i,i+1}=r_i^{-1}s_{i}$ \quad for $1\leq i\leq n$, 
\item $s_{j-1}^{-1}\alpha _{i,j}s_{j-1}=\alpha _{i,j-1}$ \quad for $j-i\geq 2$ and $\alpha \in \{ p, x, y\}$, 
\item $s_{i-1}^{-1}\alpha _{i,j}s_{i-1}=\alpha _{i-1,j}$ \quad for $2\leq i<j\leq n+1$ and $\alpha \in \{ p, x, y\}$. 
\end{enumerate}
\end{lem} 

\begin{figure}[h]
\includegraphics[scale=0.75]{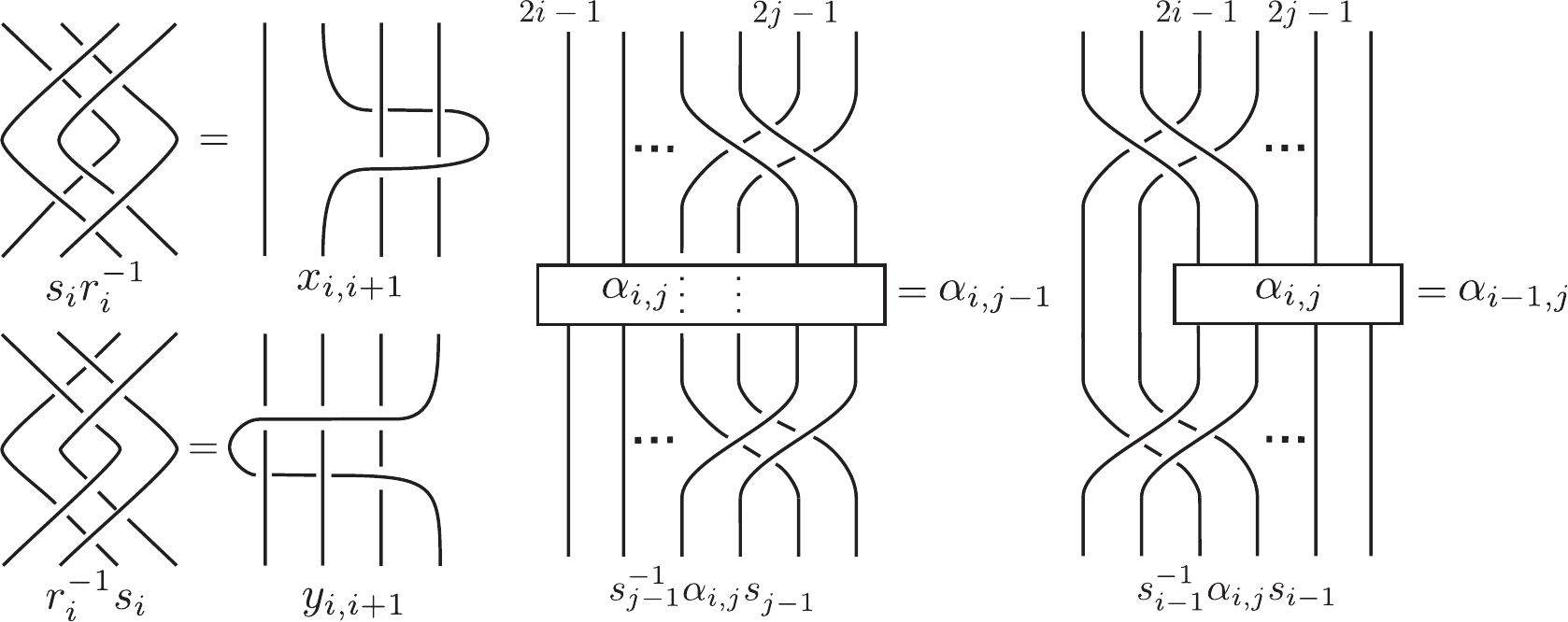}
\caption{The relations $x_{i,i+1}=s_{i}r_i^{-1}$, $y_{i,i+1}=r_i^{-1}s_{i}$ for $1\leq i\leq n$, $s_{j-1}^{-1}\alpha _{i,j}s_{j-1}=\alpha _{i,j-1}$ for $j-i\geq 2$ and $\alpha \in \{ p, x, y\}$, and $s_{i-1}^{-1}\alpha _{i,j}s_{i-1}=\alpha _{i-1,j}$ for $2\leq i<j\leq n+1$ and $\alpha \in \{ p, x, y\}$.}\label{fig_proof_lem_pxy}
\end{figure}

\section{The group structure of the liftable Hilden group via a group extension}\label{section_exact-seq}

In this section, we observe the group structure of $\LH $ via a group extension. 
Recall that the image of the liftable mapping class group $\LM $ with respect to $\Psi \colon \M \to S_{2n+2}$ coincides with the group $W_{2n+2}\cong (S_{n+1}^o\times S_{n+1}^e)\rtimes \Z _2$ by Lemma~\ref{lem_GW}, and we identify the symmetric group $S_{2n+2}$ with the group $\mathrm{Map}(\B )$ of self-bijections on $\B =\{ p_1, p_2, \dots ,p_{2n+2}\}$ (see Section~\ref{section_liftable-condition}). 
For $f\in \M $, we define $\bar{f}=\Psi (f)$. 

Put $\widehat{\B }=\{ \{ p_1, p_2\} , \{ p_3, p_4\} , \dots , \{ p_{2n+1}, p_{2n+2}\} \}$ and 
\[
\sigma (\widehat{\B })=\{ \sigma (\{ p_1, p_2\} ), \sigma (\{ p_3, p_4\} ), \dots , \sigma (\{ p_{2n+1}, p_{2n+2}\} )\}
\]
for $\sigma \in S_{2n+2}$. 
We remark that $\widehat{\B }=\{ \partial a_1, \partial a_2, \dots ,\partial a_{n+1}\}$. 
Denote by $V_{2n+2}$ the subgroup of $S_{2n+2}$ which consists of elements preserving $\widehat{\B }$, i.e. $\sigma (\widehat{\B })=\widehat{\B }$ for each $\sigma \in V_{2n+2}$. 
Then we have the following lemma. 

\begin{lem}\label{lem_Hil-V}
For $n\geq 1$, we have $\Psi (\Hil )=V_{2n+2}$.
\end{lem}

\begin{proof}
For any $f\in \Hil $, since the mapping class $f$ preserves $\A =a_1\sqcup a_2\sqcup \cdots \sqcup a_{2n+2}$ (i.e. $\{ f(a_1), f(a_2), \dots , f(a_{2n+2})\} =\{ a_1, a_2, \dots , a_{2n+2}\}$), we have $f(\widehat{\B })=\{ f(\partial a_1), f(\partial a_2), \dots , f(\partial a_{n+1})\} =\{ \partial a_1, \partial a_2, \dots ,\partial a_{n+1}\} =\widehat{\B }$. 
Thus we also have $\bar{f}(\widehat{\B })=\widehat{\B }$ and $\Psi (\Hil )\subset V_{2n+2}$. 

Since $V_{2n+2}$ acts on $\widehat{\B }$, we have the natural homomorphism $\Pi \colon V_{2n+2}\to S_{n+1}$, where we regard $S_{n+1}$ as the group $\mathrm{Map}(\widehat{\B })$. 
Remark that $S_{n+1}$ is generated by the transpositions $\tau _i$ $(1\leq i\leq n)$ such that $\tau _i$ transposes $\partial a_i$ and $\partial a_{i+1}$. 
For $1\leq i\leq n$, we have $s_i(a_i)=a_{i+1}$, $s_i(a_{i+1})=a_{i}$, and $s_i(a_l)=a_{l}$ for $l\not =i, i+1$. 
Thus we have $\Pi (\bar{s}_i)=\tau _i$ for $1\leq i\leq n$ and $\Pi $ is surjective. 

Recall that $t_{i}=\sigma _{2i-1}$ and $\bar{t}_{i}=(2i-1\ 2i)\in V_{2n+2}$ for $1\leq i\leq n+1$. 
For $\sigma \in V_{2n+2}$, $\Pi (\sigma )=1$ in $S_{n+1}$ if and only if the pair $(\sigma (p_{2i-1}), \sigma (p_{2i}))$ is equal to either $(p_{2i-1}, p_{2i})$ or $(p_{2i}, p_{2i-1})$. 
Hence, if $\Pi (\sigma )=1$ in $S_{n+1}$, then $\sigma =1$ or there exists integers $1\leq i_1<i_2<\cdots <i_l\leq n+1$ such that $\sigma =\bar{t}_{i_1}\bar{t}_{i_2}\cdots \bar{t}_{i_l}$. 
Thus $\ker \Pi $ is generated by $\bar{t}_1, \bar{t}_2, \dots , \bar{t}_{n+1}$ and we have the following exact sequence:
\begin{eqnarray}\label{exact_V}
1\longrightarrow \Z _2^{n+1} \longrightarrow V_{2n+2} \stackrel{\Pi }{\longrightarrow }S_{n+1}\longrightarrow 1. 
\end{eqnarray}
By arguments above, we show that $V_{2n+2}$ is generated by $\bar{s}_{i}=\Psi (s_i)$ for $1\leq i\leq n$ and $\bar{t}_{i}=\Psi (t_i)$ for $1\leq i\leq n+1$. 
Therefore, we have $\Psi (\Hil )\supset V_{2n+2}$ and we have completed the proof of Lemma~\ref{lem_Hil-V}.  
\end{proof}

We remark that the preimage $\Psi ^{-1}(V_{2n+2})$ does not coincide with $\Hil $. 
For instance, $\bar{\sigma }_2^2=1\in V_{2n+2}$, however, $^{\sigma _2^2}\! \A $ is not isotopic to $\A $ relative to $\partial B^3$ (i.e. $\sigma _2^2\not \in \Hil $). 
By Lemma~\ref{lem_Hil-V} and restricting the exact sequence~(\ref{exact2}) in Section~\ref{section_liftable-element} to $\Hil $, we have the following proposition. 

\begin{prop}\label{prop_exact_Hil}
For $n\geq 1$, we have the following exact sequence: 
\begin{eqnarray}\label{exact_Hil}
1\longrightarrow \PH \longrightarrow \Hil \stackrel{\Psi }{\longrightarrow }V_{2n+2}\longrightarrow 1. 
\end{eqnarray}
\end{prop}

Let $VW_{2n+2}$ be the intersection of $V_{2n+2}$ and $W_{2n+2}$, i.e. $VW_{2n+2}$ consists of elements in $S_{2n+2}$ which preserve $\widehat{\B }$ and are parity-preserving or parity-reversing, and let $S_{n+1}^{oe}$ be the subgroup of $S_{2n+2}$ which consists of elements that preserve $\widehat{\B }$ and are parity-preserving. 
By the definitions, we have $S_{n+1}^{oe}=VW_{2n+2}\cap \ker \pi $ (for the definition of $\pi \colon W_{2n+2}\to \Z _2$, see Section~\ref{section_liftable-condition}). 
We have the following proposition. 

\begin{prop}\label{prop_exact_VW}
For $n\geq 1$, we have the following exact sequence: 
\begin{eqnarray}\label{exact_VW}
1\longrightarrow S_{n+1}^{oe} \longrightarrow VW_{2n+2} \stackrel{\pi }{\longrightarrow }\Z _2\longrightarrow 1. 
\end{eqnarray}
\end{prop}

\begin{proof}
Recall that $\bar{r}=(1\ 2)(3\ 4)\cdots (2n+1\ 2n+2)$ in $S_{2n+1}$. 
We can see that $\bar{r}$ is parity-reversing and $\pi (\bar{r})=1\in \Z _2$. 
Thus the restriction $\pi |_{VW_{2n+2}}\colon VW_{2n+2}\to \Z _2$ is surjective.  
By restricting the exact sequence~(\ref{exact1}) in Section~\ref{section_liftable-condition} to $VW_{2n+2}$, we have the exact sequence
\[
1\longrightarrow VW_{2n+2}\cap \ker \pi  \longrightarrow VW_{2n+2} \stackrel{\pi }{\longrightarrow }\Z _2\longrightarrow 1. 
\] 
By an argument above, we have $S_{n+1}^{oe}=VW_{2n+2}\cap \ker \pi $ and have completed the proof of Proposition~\ref{prop_exact_VW}. 
\end{proof}

We review the precise definition of $\Pi \colon V_{2n+2}\to S_{n+1}$ which appear in the proof of Lemma~\ref{lem_Hil-V}. 
Put $\frac{\B }{2}=\{ 1, 2, \dots , n+1\}$ and we define $\beta \colon \B \to \frac{\B }{2}$ by $\beta (p_i)=\frac{i+1}{2}$ for odd $1\leq i\leq 2n+1$ and $\beta (p_i)=\frac{i}{2}$ for even $2\leq i\leq 2n+2$. 
We regard $S_{n+1}$ as the group of bijections on $\frac{\B }{2}$. 
Then we redefine $\Pi \colon V_{2n+2}\to S_{n+1}$ by $\Pi (\sigma )(i)=\beta (\sigma (p_{2i}))$ for $\sigma \in V_{2n+2}$ and $i\in \frac{\B }{2}$. 
Since $V_{2n+2}$ preserves $\widehat{\B }$, $\Pi $ is a well-defined homomorphism with the following commutative diagram:  
\[
\xymatrix{
 \B \ar[r]^{\sigma }\ar[d]_{\beta }  &  \B \ar[d]^{\beta } \\
 \frac{\B }{2}  \ar[r]_{\Pi (\sigma )} &\frac{\B }{2}. \ar@{}[lu]|{\circlearrowright}
}
\] 
Then, we have the following lemma. 

\begin{lem}\label{lem_S^oe}
For $n\geq 1$, the restriction $\Pi |_{S_{n+1}^{oe}}\colon S_{n+1}^{oe}\to S_{n+1}$ is an isomorphism.
\end{lem}

\begin{proof}
For $\sigma \in S_{n+1}^{oe}$, we suppose that $\Pi (\sigma )=1$ in $S_{n+1}$. 
By the commutative diagram above, we have $\beta \circ \sigma =\beta $. 
Since $\beta (\sigma (p_{2i-1}))=\beta (p_{2i-1})=i$ and $\beta (\sigma (p_{2i}))=\beta (p_{2i})=i$ for $1\leq i\leq n+1$, we have 
\[
\{ \sigma (p_{2i-1}), \sigma (p_{2i})\} =\beta ^{-1}(i)=\{ p_{2i-1}, p_{2i}\}
\]
for $1\leq i\leq n+1$ (the condition $\{ \sigma (p_{2i-1}), \sigma (p_{2i})\} \supset \beta ^{-1}(i)$ follows from the injectivity of $\sigma $). 
Since $\sigma $ is parity-preserving, we have $(\sigma (p_{2i-1}), \sigma (p_{2i}))=(p_{2i-1}, p_{2i})$ for $1\leq i\leq n+1$. 
Thus we have $\sigma =1$ in $S_{n+1}^{oe}$ and $\Pi |_{S_{n+1}^{oe}}$ is injective. 

For $\sigma \in S_{n+1}$, we define $\widetilde{\sigma }\colon \B \to \B$ by $\widetilde{\sigma }(p_{2i-1})=p_{2\sigma (i)-1}$ and $\widetilde{\sigma }(p_{2i})=p_{2\sigma (i)}$ for $1\leq i\leq n+1$. 
By the bijectivity of $\sigma $, $\widetilde{\sigma }$ is also bijective (i.e. $\widetilde{\sigma }\in S_{2n+2}$). 
Since $\widetilde{\sigma }$ is parity-preserving and $\widetilde{\sigma }(\widehat{\B })=\widehat{\B }$ by the definition of $\widetilde{\sigma }$, we have $\widetilde{\sigma }\in  S_{n+1}^{oe}$. 
For $1\leq i\leq n+1$, we have 
\begin{eqnarray*}
\Pi (\widetilde{\sigma })(i)=\beta (\widetilde{\sigma }(p_{2i}))=\beta (p_{2\sigma (i)})=\sigma (i). 
\end{eqnarray*}
Thus we have $\Pi (\widetilde{\sigma })=\sigma $ and $\Pi |_{S_{n+1}^{oe}}$ is surjective. 
Therefore we have completed the proof of Lemme~\ref{lem_S^oe}. 
\end{proof}

Recall that $\bar{s}_{i}=(2i-1\ 2i+1)(2i\ 2i+2)$ in $S_{n+1}^{oe}$ and $\Pi (\bar{s}_{i})=(i\ i+1)=\tau _i$ in $S_{n+1}$ for $1\leq i\leq n$. 
Thus, by Lemma~\ref{lem_S^oe}, we have the following Lemma. 

\begin{lem}\label{lem_S^oe_gen}
For $n\geq 1$, the restriction $S_{n+1}^{oe}$ is generated by $\bar{s}_{i}$ for $1\leq i\leq n$.
\end{lem}

By the next proposition, we can see the group $\LH$ as a group extension. 

\begin{prop}\label{prop_exact_LH}
For $n\geq 1$, we have the following exact sequence: 
\begin{eqnarray}\label{exact_LH}
1\longrightarrow \PH \longrightarrow \LH \stackrel{\Psi }{\longrightarrow }VW_{2n+2}\longrightarrow 1. 
\end{eqnarray}
\end{prop}

\begin{proof}
By Lemmas~\ref{lem_GW} and \ref{lem_Hil-V}, we have $\Psi (\LM )=W_{2n+2}$ and $\Psi (\Hil )=V_{2n+2}$. 
Hence, since $\LH =\LM \cap \Hil $ by Lemma~\ref{liftability_surf_hand}, we have 
\[
\Psi (\LH )=\Psi (\LM )\cap \Psi (\Hil )=W_{2n+2}\cap V_{2n+2}=VW_{2n+2}.
\] 
Recall that $\PH =\PM \cap \Hil $. 
Therefore, by restricting the exact sequence~(\ref{exact2}) in Section~\ref{section_liftable-element} to $\LH $, we obtain the exact sequence~\ref{exact_LH}, and have completed the proof of Proposition~\ref{prop_exact_LH}. 
\end{proof}

\section{A presentation for the liftable Hilden group}\label{section_lmod}


Recall that for elements $f$ and $h$ in a group $G$, $f\rightleftarrows h$ means the commutative relation $fh=hf$ in $G$. 
The main theorem of this paper for the liftable Hilden group is as follows. 

\begin{thm}\label{thm_pres_LH}
For $n\geq 1$, $\LH $ admits the presentation with generators $s_i$ for $1\leq i\leq n$, $r_i$ for $1\leq i\leq n$, $t_{i}$ for $1\leq i\leq n+1$, and $r$, and the following defining relations: 
\begin{enumerate}
\item commutative relations
\begin{enumerate}
\item $\alpha _i \rightleftarrows \beta _j$ \quad for $j-i\geq 2$ and $\alpha , \beta \in \{ s, r\}$,
\item $\alpha _{i} \rightleftarrows t_{j}$ \quad for $j\not =i, i+1$ and $\alpha \in \{ s, r\}$, 
\item $t_i\rightleftarrows t_{j}$ \quad for $1\leq i<j\leq n+1$, 
\item $s_i \rightleftarrows r$ \quad for $1\leq i\leq n$, 
\item $t_i \rightleftarrows r$ \quad for $1\leq i\leq n+1$, 
\end{enumerate}
\item conjugation relations
\begin{enumerate}
\item $\alpha _i\alpha _{i+1}\alpha _i=\alpha _{i+1}\alpha _i\alpha _{i+1}$ \quad for $1\leq i\leq n-1$ and $\alpha \in \{ s, r\}$, 
\item $s_i^{\varepsilon }s_{i+1}^{\varepsilon }r_{i}=r_{i+1}s_i^{\varepsilon }s_{i+1}^{\varepsilon }$ \quad for $1\leq i\leq n-1$ and $\varepsilon \in \{ 1, -1\}$,
\item $r_ir_{i+1}s_{i}=s_{i+1}r_ir_{i+1}$ \quad for $1\leq i\leq n-1$,
\item $r_irs_{i}=rs_{i}r_{i}^{-1}$ \quad for $1\leq i\leq n$,
\item $s_i^{\varepsilon }t_{i}=t_{i+1}s_{i}^{\varepsilon }$ \quad for $1\leq i\leq n$ and $\varepsilon \in \{ 1, -1\}$,
\item $r_it_{i}=t_{i+1}r_{i}$ \quad for $1\leq i\leq n$,
\item $t_is_{i}^2r_i=r_is_i^2t_{i+1}$ \quad for $1\leq i\leq n$,
\end{enumerate}
\item $r^2=t_1t_2\cdots t_{n+1}$, 
\item $r_1r_2\cdots r_ns_n\cdots s_2s_1t_{1}=1$, 
\item $t_{1}t_2\cdots t_{n+1}\bigl( s_1(s_2s_1)\cdots (s_{n-1}\cdots s_2s_1)(s_n\cdots s_2s_1)\bigr) ^2=1$. 
\end{enumerate}
\end{thm}

We will give the proof of Theorem~\ref{thm_pres_LH} in Section~\ref{section_proof_lmod}.

\subsection{A finite presentation for the pure Hilden group}\label{section_pure_Hilden}

In this section, we give a finite presentation for $\PH $ which is obtained from Tawn's finite presentation for the pure Hilden group with one marked disk in \cite{Tawn2}. 
Let $D$ be a disk in $S^2-\B $, $\Hil ^1$ the group of isotopy classes of orientation-preserving self-homeomorphisms on $B^3$ fixing $\A $ setwise and $D$ pointwise, and $\PH ^1$ the group of isotopy classes of orientation-preserving self-homeomorphisms on $B^3$ fixing $\A \cup D$ pointwise. 
Tawn~\cite{Tawn2} gave the following presentation for $\PH ^1$. 

\begin{thm}[\cite{Tawn2}]\label{thm_pres_PH^1}
For $n\geq 1$, $\PH ^1$ admits the presentation with generators $p_{i,j}$, $x_{i,j}$, $y_{i,j}$ for $1\leq i,j\leq n+1$ with $i\not =j$, and $t_i$ for $1\leq i\leq n+1$, and the following defining relations: 
\begin{enumerate}
\item[(0)] $\alpha _{i,j}=\alpha _{j,i}$  \quad for $1\leq i,j\leq n+1$ with $i\not =j$ and $\alpha \in \{ p, x, y\}$, 
\item[(C-pt)] $p_{i,j} \rightleftarrows t_k$ \quad for $1\leq i<j\leq n+1$ and $1\leq k\leq n+1$,
\item[(C-tt)] $t_{i} \rightleftarrows t_{j}$ \quad for $1\leq i<j\leq n+1$, 
\item[(C-xt)] $x_{i,j} \rightleftarrows t_{k}$ \quad for $i<j$ and $k\not =i$, 
\item[(C-yt)] $y_{i,j} \rightleftarrows t_{k}$ \quad for $i<j$ and $k\not =j$, 
\item[(C1)] $\alpha _{i,j} \rightleftarrows \beta _{k,l}$ \quad for cyclically ordered $(i,j,k,l)$ and $\alpha , \beta \in \{ p, x, y\}$, 
\item[(C2)] $\alpha _{i,j} \rightleftarrows \beta _{i,k}\gamma _{j,k}$ \quad for cyclically ordered $(i,j,k)$ and $(\alpha , \beta , \gamma )$ as in Table~\ref{rel_C2_PH^1}, 
\item[(C3)] $\alpha _{i,k} \rightleftarrows p_{j,k}\beta _{j,l}p_{j,k}^{-1}$ \quad for cyclically ordered $(i,j,k,l)$ and $\alpha , \beta \in \{ p, x, y\}$, 
\item[(M-x)] $x_{i,j} \rightleftarrows p_{i,j}t_{i}$ \quad for $i<j$, 
\item[(M-y)] $y_{i,j} \rightleftarrows p_{i,j}t_{j}$ \quad for $i<j$. 
\end{enumerate}
\end{thm}

\begin{table}[h]
  \centering
  \begin{tabular}{|c||cccc|}
    \hline
    $i<j<k$ & $(p, p, p)$ & $(p, y, y)$ & $(x, p, p)$ & $(x, x, p)$ \\
     & $(x, y, y)$ & $(y, p, p)$ & $(y, p, x)$ & $(y, y, y)$ \\
    \hline
    $j<k<i$ & $(p, p, p)$ & $(p, x, y)$ & $(x, p, p)$ & $(x, p, x)$ \\
     & $(x, x, y)$ & $(y, p, p)$ & $(y, x, y)$ & $(y, y, p)$ \\
    \hline
    $k<i<j$ & $(p, p, p)$ & $(p, x, x)$ & $(x, p, p)$ & $(x, x, x)$ \\
     & $(x, y, p)$ & $(y, p, p)$ & $(y, p, y)$ & $(y, x, x)$ \\
    \hline
    \end{tabular}
\caption{The values of $(\alpha , \beta , \gamma )$ for the relation (C2) in Theorem~\ref{thm_pres_PH^1}. }\label{rel_C2_PH^1}
\end{table}

Let $B_{2n+2}$ be the ``classical'' braid group of $2n+2$ strands. 
$B_{2n+2}$ is isomorphic to the group of isotopy classes of orientation-preserving self-homeomorphisms on $S^2$ fixing $\B $ setwise and $D$ pointwise. 
By restricting the action of $\Hil ^1$ to $S^2$, we have the injective homomorphism $\Hil ^1\hookrightarrow B_{2n+2}$ and regard $\Hil ^1$ as a subgroup of $B_{2n+2}$. 
Fadell and Buskirk~\cite{Fadell-Buskirk} showed that the spherical braid group $SB_{2n+2}$ is isomorphic to the quotient of $B_{2n+2}$ by the subgroup generated by 
\[
z=\sigma _1\sigma _2\cdots \sigma _{2n+1}\sigma _{2n+1}\cdots \sigma _2\sigma _1
\]
(see the left-hand side in Figure~\ref{fig_braid_z}). 
Recall that the surjective homomorphism $\Gamma \colon SB_{2n+2}\to \M $ (defined in Section~\ref{section_wicket-group}) has the kernel generated by the full-twist braid. 
The image of $\PH ^1$ by the composition of the quotient map $B_{2n+2}\to SB_{2n+2}$ and $\Gamma $ coincides with $\PH $. 
We have 
\[
z=x_{1,n+1}^{-1}\cdots x_{1,3}^{-1}x_{1,2}^{-1}p_{1,2}p_{1,3}\cdots p_{1,n+1}t_{1}
\]
in $\PH ^1$ (see the right-hand side in Figure~\ref{fig_braid_z}) and the full-twist braid in $B_{2n+2}$ is expressed by 
\[
t_{1}t_2\cdots t_{n+1}p_{1,2}(p_{1,3}p_{2,3})\cdots (p_{1,n}p_{2,n}\cdots p_{n-1,n})(p_{1,n+1}p_{2,n+1}\cdots p_{n,n+1}) 
\]
in $\PH ^1$ (see the right-hand side in Figure~\ref{fig_full-twist}). 
Thus we have the following lemma. 

\begin{figure}[h]
\includegraphics[scale=1.06]{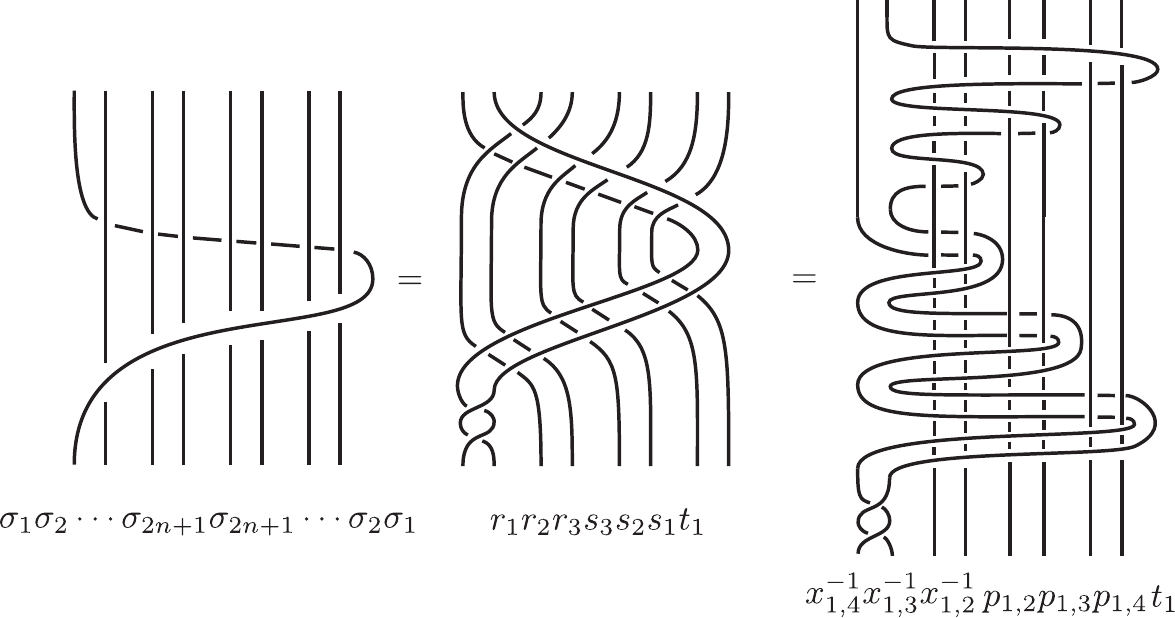}
\caption{Expressions $z=r_1r_2\cdots r_{n}s_n\cdots s_2s_1t_{1}$ in $SW_{2n+2}$ and $z=x_{1,n+1}^{-1}\cdots x_{1,3}^{-1}x_{1,2}^{-1}p_{1,2}p_{1,3}\cdots p_{1,n+1}t_{1}$ in $PSW_{2n+2}$ of the braid $z=\sigma _1\sigma _2\cdots \sigma _{2n+1}\sigma _{2n+1}\cdots \sigma _2\sigma _1$ when $n=3$.}\label{fig_braid_z}
\end{figure}

\begin{figure}[h]
\includegraphics[scale=1.1]{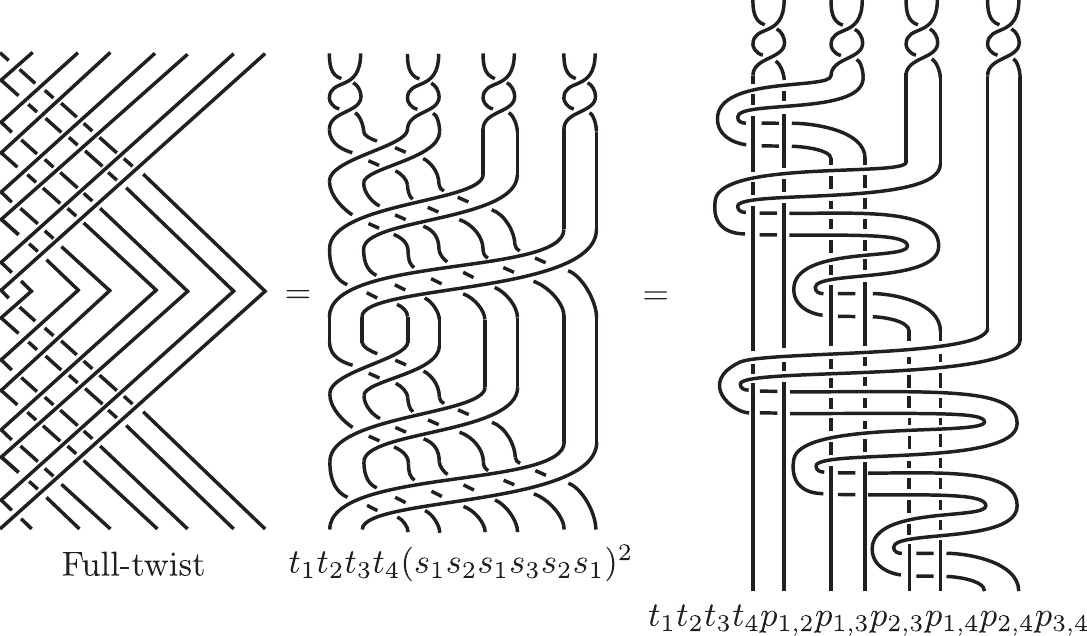}
\caption{Expressions $t_{1}t_2\cdots t_{n+1}\bigl( s_1(s_2s_1)\cdots (s_{n-1}\cdots s_2s_1)(s_n\cdots s_2s_1)\bigr) ^2$ and $t_{1}t_2\cdots t_{n+1}p_{1,2}(p_{1,3}p_{2,3})\cdots (p_{1,n}p_{2,n}\cdots p_{n-1,n})(p_{1,n+1}p_{2,n+1}\cdots p_{n,n+1})$ of the full-twist by generators for $SW_{2n+2}$ and $PSW_{2n+2}$ when $n=3$.}\label{fig_full-twist}
\end{figure}

\begin{lem}\label{lem_pres_PH}
For $n\geq 1$, $\PH $ admits the presentation which is obtained from the finite presentation for $\PH ^1$ in Theorem~\ref{thm_pres_PH^1} by adding the relations 
\begin{enumerate}
\item[(Z)] $x_{1,n+1}^{-1}\cdots x_{1,3}^{-1}x_{1,2}^{-1}p_{1,2}p_{1,3}\cdots p_{1,n+1}t_{1}=1$,
\item[(F)] $t_{1}t_2\cdots t_{n+1}p_{1,2}(p_{1,3}p_{2,3})\cdots (p_{1,n}p_{2,n}\cdots p_{n-1,n})(p_{1,n+1}p_{2,n+1}\cdots p_{n,n+1})=1$.  
\end{enumerate}
\end{lem}

\subsection{Group extensions and presentations for groups}\label{section_presentation_exact}

To prove Theorem~\ref{thm_pres_LH}, we review a relationship between a group extension and group presentations in this section from Section~3 in \cite{Hirose-Omori}. 
Let $G$ be a group and let $H=\bigl< X\mid R\bigr>$ and $Q=\bigl< Y\mid S\bigr>$ be presented groups which have the short exact sequence 
\[
1\longrightarrow H\stackrel{\iota }{\longrightarrow }G\stackrel{\nu }{\longrightarrow }Q\longrightarrow 1.
\]
We take a preimage $\widetilde{y}\in G$ of $y\in Q$ with respect to $\nu $ for each $y\in Q$. 
Then we put $\widetilde{X}=\{ \iota (x) \mid x\in X\} \subset G$ and $\widetilde{Y}=\{ \widetilde{y} \mid y\in Y\} \subset G$. 
Denote by $\widetilde{r}$ the word in $\widetilde{X}$ which is obtained from $r\in R$ by replacing each $x\in X$ by $\iota (x)$ and also denote by $\widetilde{s}$ the word in $\widetilde{Y}$ which is obtained from $s\in S$ by replacing each $y\in Y$ by $\widetilde{y}$. 
We note that $\widetilde{r}=1$ in $G$. 
Since $\widetilde{s}\in G$ is an element in $\ker \nu =\iota (H)$ for each $s\in S$, there exists a word $v_{s}$ in $\widetilde{X}$ such that $\widetilde{s}=v_{s}$ in $G$. 
Since $\iota (H)$ is a normal subgroup of $G$, for each $x\in X$ and $y\in Y$, there exists a word $w_{x,y}$ in $\widetilde{X}$ such that $\widetilde{y}\iota (x)\widetilde{y}^{-1}=w_{x,y}$ in $G$. 
The next lemma follows from an argument of the combinatorial group theory (for instance, see \cite[Proposition~10.2.1, p139]{Johnson}).

\begin{lem}\label{presentation_exact}
Under the situation above, the group $G$ admits the presentation with the generating set $\widetilde{X}\cup \widetilde{Y}$ and following defining relations:
\begin{enumerate}
 \item[(A)] $\widetilde{r}=1$ \quad for $r\in R$,
 \item[(B)] $\widetilde{s}=v_{s}$ \quad for $s\in S$,
 \item[(C)] $\widetilde{y}\iota (x)\widetilde{y}^{-1}=w_{x,y}$ \quad for $x\in X$ and $y\in Y$.
\end{enumerate} 
\end{lem}

\subsection{Proof of presentation for the liftable Hilden group}\label{section_proof_lmod}

Recall that $VW_{2n+2}$ is the subgroup of the symmetric group $S_{2n+2}$ which consists of parity-preserving or parity-reversing elements preserving $\widehat{\B }=\{ \{ p_1, p_2\} , \{ p_3, p_4\} , \dots , \{ p_{2n+1}, p_{2n+2}\} \}$, we define $\bar{\sigma }=\Psi (\sigma )\in S_{2n+2}$ for $\sigma \in \LM $, and $\bar{s}_i=(2i-1\ 2i+1)(2i\ 2i+2)$ for $1\leq i\leq n$ and $\bar{r}=(1\ 2)(3\ 4)\cdots (2n+1\ 2n+2)$ (see Sections~\ref{section_liftable-element} and \ref{section_exact-seq}). 
First we give a finite presentation for $VW_{2n+2}$. 

\begin{lem}\label{lem_pres_VW}
For $n\geq 1$, $VW_{2n+2}$ admits the presentation with generators $\bar{s}_{i}$ for $1\leq i\leq n$ and $\bar{r}$, and the following defining relations: 
\begin{enumerate}
\item $\bar{s}_i^2=1$ \quad for $1\leq i\leq n$, 
\item $\bar{s}_{i} \rightleftarrows \bar{s}_{j}$ \quad for $j-i\geq 2$,
\item $\bar{s}_{i}\bar{s}_{i+1}\bar{s}_{i}=\bar{s}_{i+1}\bar{s}_{i}\bar{s}_{i+1}$ \quad for $1\leq i\leq n-1$, 
\item $\bar{r}^2=1$, 
\item $\bar{s}_{i} \rightleftarrows \bar{r}$ \quad for $1\leq i\leq n$. 
\end{enumerate}
\end{lem} 

\begin{proof}
We apply Lemma~\ref{presentation_exact} to the exact sequence~(\ref{exact_VW}) in Proposition~\ref{prop_exact_VW}. 
By Lemmas~\ref{lem_S^oe} and \ref{lem_S^oe_gen}, the group $S_{n+1}^{oe}$ is isomorphic to the symmetric group $S_{n+1}$ and is generated by the transpositions $\bar{s}_1, \bar{s}_2, \dots , \bar{s}_{n}$.  
Hence $S_{n+1}^{oe}$ has the presentation with generators $\bar{s}_{i}$ for $1\leq i\leq n$ and the following defining relations: 
\begin{enumerate}
\item $\bar{s}_i^2=1$ \quad for $1\leq i\leq n$, 
\item $\bar{s}_{i} \rightleftarrows \bar{s}_{j}$ \quad for $j-i\geq 2$,
\item $\bar{s}_{i}\bar{s}_{i+1}\bar{s}_{i}=\bar{s}_{i+1}\bar{s}_{i}\bar{s}_{i+1}$ \quad for $1\leq i\leq n-1$. 
\end{enumerate}
Since the image $\pi (VW_{2n+2})=\Z _2$ is generated by the image of any one parity-reversing element with respect to $\pi $, $\pi (\bar{r})$ generates $\Z _2$. 
Remark that the relation $\bar{r}^2=1$ holds in $VW_{2n+2}$ and $\bar{r}$ commutes with $\bar{s}_i$ for $1\leq i\leq n$. 
By applying Lemma~\ref{presentation_exact} to the exact sequence~(\ref{exact_VW}), we have the presentation for $VW_{2n+2}$ with generators $\bar{s}_{i}$ for $1\leq i\leq n$ and $\bar{r}$, and the following defining relations: 
\begin{enumerate}
\item[(A)]
\begin{enumerate}
\item $\bar{s}_i^2=1$ \quad for $1\leq i\leq n$, 
\item $\bar{s}_{i} \rightleftarrows \bar{s}_{j}$ \quad for $j-i\geq 2$,
\item $\bar{s}_{i}\bar{s}_{i+1}\bar{s}_{i}=\bar{s}_{i+1}\bar{s}_{i}\bar{s}_{i+1}$ \quad for $1\leq i\leq n-1$, 
\end{enumerate}
\item[(B)] $\bar{r}^2=1$, 
\item[(C)] $\bar{s}_{i} \rightleftarrows \bar{r}$ \quad for $1\leq i\leq n$. 
\end{enumerate}
This presentation coincides with the presentation in Lemma~\ref{lem_pres_VW}. 
Therefore we have completed the proof of Lemma~\ref{lem_pres_VW}.   
\end{proof}

By applying Lemma~\ref{presentation_exact} to the exact sequence~(\ref{exact_LH}) in Proposition~\ref{prop_exact_LH} and presentations for $\PH $ and $VW_{2n+2}$ in Lemmas~\ref{lem_pres_PH} and \ref{lem_pres_VW}, respectively, we have the following lemma. 

\begin{lem}\label{lem_pres_LH}
For $n\geq 1$, $\LH $ admits the presentation with generators $s_i$ for $1\leq i\leq n$, $t_{i}$ for $1\leq i\leq n+1$, $r$, $p_{i,j}$, $x_{i,j}$, and $y_{i,j}$ for $1\leq i,j\leq n+1$ with $i\not =j$ and the following defining relations: 
\begin{enumerate}
\item[(0)] $\alpha _{i,j}=\alpha _{j,i}$  \quad for $1\leq i,j\leq n+1$ with $i\not =j$ and $\alpha \in \{ p, x, y\}$, 
\item[(C-pt)] $p_{i,j} \rightleftarrows t_k$ \quad for $1\leq i<j\leq n+1$ and $1\leq k\leq n+1$,
\item[(C-tt)] $t_{i} \rightleftarrows t_{j}$ \quad for $1\leq i<j\leq n+1$, 
\item[(C-xt)] $x_{i,j} \rightleftarrows t_{k}$ \quad for $1\leq i<j\leq n+1$ and $k\not =i$, 
\item[(C-yt)] $y_{i,j} \rightleftarrows t_{k}$ \quad for $1\leq i<j\leq n+1$ and $k\not =j$, 
\item[(C1)] $\alpha _{i,j} \rightleftarrows \beta _{k,l}$ \quad for cyclically ordered $(i,j,k,l)$ and $\alpha , \beta \in \{ p, x, y\}$, 
\item[(C2)] $\alpha _{i,j} \rightleftarrows \beta _{i,k}\gamma _{j,k}$ \quad for cyclically ordered $(i,j,k)$ and $(\alpha , \beta , \gamma )$ as in Table~\ref{rel_C2_PH^1}, 
\item[(C3)] $\alpha _{i,k} \rightleftarrows p_{j,k}\beta _{j,l}p_{j,k}^{-1}$ \quad for cyclically ordered $(i,j,k,l)$ and $\alpha , \beta \in \{ p, x, y\}$, 
\item[(M-x)] $x_{i,j} \rightleftarrows p_{i,j}t_{i}$ \quad for $i<j$, 
\item[(M-y)] $y_{i,j} \rightleftarrows p_{i,j}t_{j}$ \quad for $i<j$, 
\item[(Z)] $x_{1,n+1}^{-1}\cdots x_{1,3}^{-1}x_{1,2}^{-1}p_{1,2}p_{1,3}\cdots p_{1,n+1}t_{1}=1$,
\item[(F)] $t_{1}t_2\cdots t_{n+1}p_{1,2}(p_{1,3}p_{2,3})\cdots (p_{1,n}p_{2,n}\cdots p_{n-1,n})(p_{1,n+1}p_{2,n+1}\cdots p_{n,n+1})=1$,  
\item $s_i^2=p_{i,i+1}$ \quad for $1\leq i\leq n$, 
\item $s_{i} \rightleftarrows s_{j}$ \quad for $j-i\geq 2$,
\item $s_{i}s_{i+1}s_{i}=s_{i+1}s_{i}s_{i+1}$ \quad for $1\leq i\leq n-1$, 
\item $r^2=t_1t_2\cdots t_{n+1}$, 
\item $s_{i} \rightleftarrows r$ \quad for $1\leq i\leq n$. 
\item[(A1)] 
\begin{enumerate}
\item $s_kt_is_k^{-1}=\left\{
		\begin{array}{ll}
		t_i \quad \text{for }i\not =k, k+1, \\
		t_{i+1} \quad \text{for }i=k, \\
		t_{i-1} \quad \text{for }i=k+1, 
		\end{array}
		\right.$
\item $s_i\alpha _{i,i+1}s_i^{-1}=\left\{
		\begin{array}{ll}
		p_{i,i+1} \quad \text{for }\alpha =p, \\
		p_{i,i+1}y_{i,i+1}p_{i,i+1}^{-1} \quad \text{for }\alpha =x, \\
		x_{i,i+1} \quad \text{for }\alpha =y, 
		\end{array}
		\right.$
\item $s_k\alpha _{i,j}s_k^{-1}=\left\{
		\begin{array}{ll}
		p_{i-1,i}\alpha _{i-1,j}p_{i-1,i}^{-1} \quad \text{for }k=i-1 \text{ and }i<j, \\
		\alpha _{i+1,j} \quad \text{for }k=i \text{ and }j-i\geq 2, \\
		p_{j-1,j}\alpha _{i,j-1}p_{j-1,j}^{-1} \quad \text{for }k=j-1 \text{ and }j-i\geq 2, \\
		\alpha _{i,j+1} \quad \text{for }k=j \text{ and }i<j, \\
		\alpha _{i,j} \quad \text{for }k\leq i-2 \text{ or }i+1\leq k\leq j-2\text{ or }j+1\leq k, 
		\end{array}
		\right.$
\end{enumerate}
\item[(A2)]
\begin{enumerate}
\item $r \rightleftarrows t_i$ \quad for $1\leq i\leq n+1$, 
\item $r \rightleftarrows p_{i,j}$ \quad for $1\leq i<j\leq n+1$, 
\item $r\alpha _{i,j}r^{-1}=\alpha _{i,j}^{-1}p_{i,j}$ \quad for $1\leq i<j\leq n+1$ and $\alpha \in \{ x, y\}$. 
\end{enumerate}
\end{enumerate}
\end{lem}

\begin{proof}
The relations $s_i^2=p_{i,i+1}$ for $1\leq i\leq n$, $s_{i} \rightleftarrows s_{j}$ for $j-i\geq 2$, $s_{i}s_{i+1}s_{i}=s_{i+1}s_{i}s_{i+1}$ for $1\leq i\leq n-1$, $r^2=t_1t_2\cdots t_{n+1}$, and $s_{i} \rightleftarrows r$ for $1\leq i\leq n$ hold in $\LH $. 
Hence, by applying Lemma~\ref{presentation_exact} to the exact sequence~(\ref{exact_LH}) in Proposition~\ref{prop_exact_LH} and presentations for $\PH $ and $VW_{2n+2}$ in Lemmas~\ref{lem_pres_PH} and \ref{lem_pres_VW}, respectively, we have the presentation for $\LH $ with generators $s_i$ for $1\leq i\leq n$, $t_{i}$ for $1\leq i\leq n+1$, $r$, $p_{i,j}$, $x_{i,j}$, and $y_{i,j}$ for $1\leq i,j\leq n+1$ with $i\not =j$ and the following defining relations: 
\begin{enumerate}
\item[(A)]
\begin{enumerate}
\item[(0)] $\alpha _{i,j}=\alpha _{j,i}$  \quad for $1\leq i,j\leq n+1$ with $i\not =j$ and $\alpha \in \{ p, x, y\}$, 
\item[(C-pt)] $p_{i,j} \rightleftarrows t_k$ \quad for $1\leq i<j\leq n+1$ and $1\leq k\leq n+1$,
\item[(C-tt)] $t_{i} \rightleftarrows t_{j}$ \quad for $1\leq i<j\leq n+1$, 
\item[(C-xt)] $x_{i,j} \rightleftarrows t_{k}$ \quad for $1\leq i<j\leq n+1$ and $k\not =i$, 
\item[(C-yt)] $y_{i,j} \rightleftarrows t_{k}$ \quad for $1\leq i<j\leq n+1$ and $k\not =j$, 
\item[(C1)] $\alpha _{i,j} \rightleftarrows \beta _{k,l}$ \quad for cyclically ordered $(i,j,k,l)$ and $\alpha , \beta \in \{ p, x, y\}$, 
\item[(C2)] $\alpha _{i,j} \rightleftarrows \beta _{i,k}\gamma _{j,k}$ \quad for cyclically ordered $(i,j,k)$ and $(\alpha , \beta , \gamma )$ as in Table~\ref{rel_C2_PH^1}, 
\item[(C3)] $\alpha _{i,k} \rightleftarrows p_{j,k}\beta _{j,l}p_{j,k}^{-1}$ \quad for cyclically ordered $(i,j,k,l)$ and $\alpha , \beta \in \{ p, x, y\}$, 
\item[(M-x)] $x_{i,j} \rightleftarrows p_{i,j}t_{i}$ \quad for $i<j$, 
\item[(M-y)] $y_{i,j} \rightleftarrows p_{i,j}t_{j}$ \quad for $i<j$, 
\item[(Z)] $x_{1,n+1}^{-1}\cdots x_{1,3}^{-1}x_{1,2}^{-1}p_{1,2}p_{1,3}\cdots p_{1,n+1}t_{1}=1$,
\item[(F)] $t_{1}t_2\cdots t_{n+1}p_{1,2}(p_{1,3}p_{2,3})\cdots (p_{1,n}p_{2,n}\cdots p_{n-1,n})(p_{1,n+1}p_{2,n+1}\cdots p_{n,n+1})=1$,  
\end{enumerate}
\item[(B)]
\begin{enumerate}
\item $s_i^2=p_{i,i+1}$ \quad for $1\leq i\leq n$, 
\item $s_{i} \rightleftarrows s_{j}$ \quad for $j-i\geq 2$,
\item $s_{i}s_{i+1}s_{i}=s_{i+1}s_{i}s_{i+1}$ \quad for $1\leq i\leq n-1$, 
\item $r^2=t_1t_2\cdots t_{n+1}$, 
\item $s_{i} \rightleftarrows r$ \quad for $1\leq i\leq n$. 
\end{enumerate}
\item[(C)]
\begin{enumerate}
\item $s_kt_is_k^{-1}=w_{t_i, s_k}$ \quad for $1\leq i\leq n+1$ and $1\leq k\leq n$, 
\item $s_k\alpha _{i,j}s_k^{-1}=w_{\alpha _{i,j}, s_k}$ \quad for $1\leq i<j\leq n+1$ and $1\leq k\leq n+1$, 
\item $rt_ir^{-1}=w_{t_i, r}$ \quad for $1\leq i\leq n+1$, 
\item $r\alpha _{i,j}r^{-1}=w_{\alpha _{i,j}, r}$ \quad for $1\leq i<j\leq n+1$, 
\end{enumerate}
\end{enumerate}
where $w_{t_i, s_k}$, $w_{\alpha _{i,j}, s_k}$, $w_{t_i, r}$, and $w_{\alpha _{i,j}, s_i}$ are some products of $t_{i}$ for $1\leq i\leq n+1$, $p_{i,j}$, $x_{i,j}$, and $y_{i,j}$ for $1\leq i<j\leq n+1$.  
These generators and the relations~(A) and (B) of this presentation coincide with the generators and the relations~(0), (C-pt), $\dots $, (5) of the presentation in Lemma~\ref{lem_pres_LH}. 
Hence it is enough for completing the proof of Lemma~\ref{lem_pres_LH} to prove that the relations~(C) in the presentation above coincide with the relations~(A1) and (A2) of the presentation in Lemma~\ref{lem_pres_LH}. 

By Lemmas~\ref{lem_comm_rel} and \ref{lem_conj_rel}, we have
\[
w_{t_k, s_i}=
\left\{
\begin{array}{ll}
t_i \quad \text{for }i\not =k, k+1, \\
t_{i+1} \quad \text{for }i=k, \\
t_{i-1} \quad \text{for }i=k+1 
\end{array}
\right.
\]
and $w_{t_i, r}=t_i$ for $1\leq i\leq n+1$, and the relations~(C) (a) and (c) coincide  with the relations~(A1) (a) and (A2) (a), respectively. 
By Lemma~\ref{lem_p_ij_relations}, we have $w_{p_{i,i+1}, s_i}=p_{i,i+1}$ for $1\leq i\leq n$ and 
\[
w_{\alpha _{i,j}, s_k}=\left\{
		\begin{array}{ll}
		p_{i-1,i}\alpha _{i-1,j}p_{i-1,i}^{-1} \quad \text{for }k=i-1 \text{ and }i<j, \\
		\alpha _{i+1,j} \quad \text{for }k=i \text{ and }j-i\geq 2, \\
		p_{j-1,j}\alpha _{i,j-1}p_{j-1,j}^{-1} \quad \text{for }k=j-1 \text{ and }j-i\geq 2, \\
		\alpha _{i,j+1} \quad \text{for }k=j \text{ and }i<j, \\
		\alpha _{i,j} \quad \text{for }k\leq i-2 \text{ or }i+1\leq k\leq j-2\text{ or }j+1\leq k.  
		\end{array}
		\right.
\]
Finally, The braids $s_i\alpha _{i,i+1}s_i^{-1}$ and $r\alpha _{i,j}r^{-1}$ for $\alpha \in \{ x, y\}$ are braids as in Figure~\ref{fig_proof_lem_LH}.  
By calculations of the braids as in Figure~\ref{fig_proof_lem_LH}, we have 
\[
w_{\alpha _{i,i+1}, s_i}=\left\{
		\begin{array}{ll}
		p_{i,i+1}y_{i,i+1}p_{i,i+1}^{-1} \quad \text{for }\alpha =x, \\
		x_{i,i+1} \quad \text{for }\alpha =y, 
		\end{array}
		\right.
\]
and $w_{\alpha _{i,j}, s_i}=\alpha _{i,j}^{-1}p_{i,j}$ for $1\leq i<j\leq n+1$ and $\alpha \in \{ x, y\}$. 
Thus the relations~(C) (b) and (d) in the presentation above are coincide with the relations~(A1) (b), (c), (A2) (b), and (c) of the presentation in Lemma~\ref{lem_pres_LH}.  
Therefore we have completed the proof of Lemma~\ref{lem_pres_LH}. 
\end{proof}

\begin{figure}[h]
\includegraphics[scale=0.8]{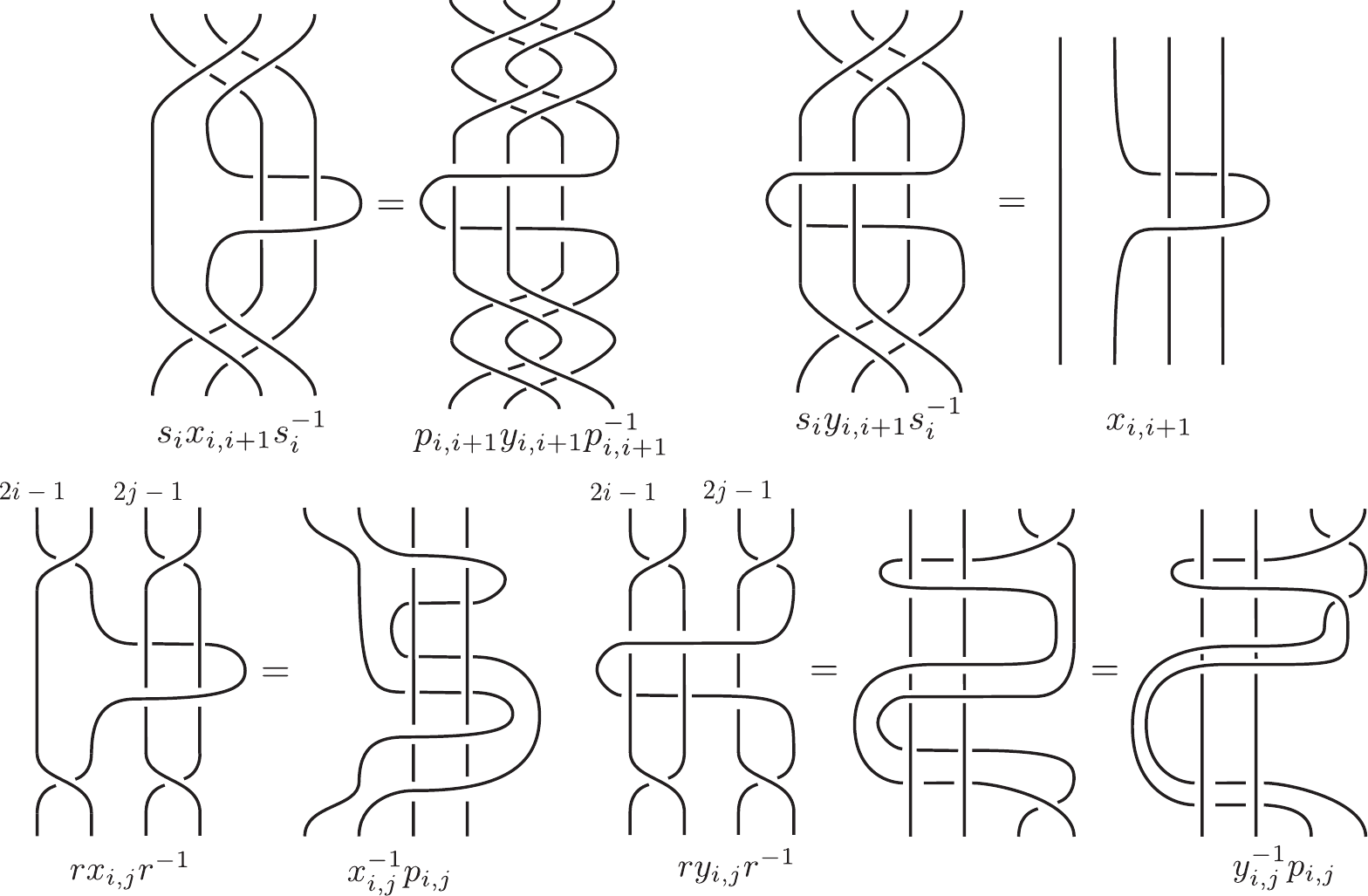}
\caption{The relations $s_{i}x_{i,i+1}s_{i}^{-1}=p_{i,i+1}y_{i,i+1}p_{i,i+1}^{-1}$, $s_{i}y_{i,i+1}s_{i}^{-1}=x_{i,i+1}$ for $1\leq i\leq n$, $rx_{i,j}r^{-1}=x_{i,j}^{-1}p_{i,j}$ and $ry_{i,j}r^{-1}=y_{i,j}^{-1}p_{i,j}$  for $1\leq i<j\leq n+1$.}\label{fig_proof_lem_LH}
\end{figure}

Put $s=s_n\cdots s_2s_1t_1\in SW_{2n+2}$ and $s=\Gamma (s)\in \LH $. 
The braid $s$ is as in Figure~\ref{fig_braid_s}.  
Then, by the conjugation relations, we have the following lemma.  

\begin{figure}[h]
\includegraphics[scale=1.2]{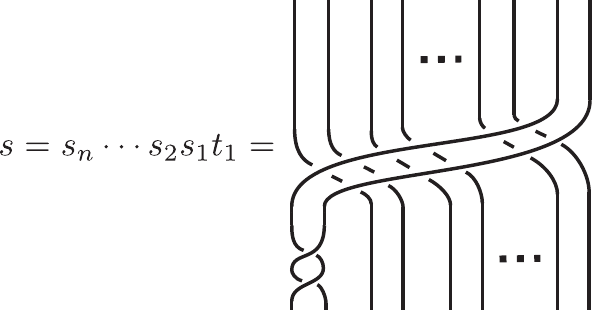}
\caption{The braid $s=s_n\cdots s_2s_1t_1$ in $SW_{2n+2}$.}\label{fig_braid_s}
\end{figure}

\begin{lem}\label{lem_rel_s-p_ij}
The relations
\begin{enumerate}
\item $s\alpha _{1,j}s^{-1}=\beta _{j-1, n+1}$ \quad for $(\alpha ,\beta )\in \{ (p,p), (x,y), (y,x)\}$ and $2\leq j\leq n+1$. 
\item $s\alpha _{i,j}s^{-1}=\alpha _{i-1,j-1}$ \quad for $\alpha \in \{ p, x, y\}$ and $2\leq i<j\leq n+1$
\end{enumerate}
hold in $SW_{2n+2}$ and $\LH $. 
\end{lem}

To prove Theorem~\ref{thm_pres_LH}, we prepare the following proposition that is proved in Section~\ref{section_proof_prop} by showing that the following presentation is equivalent to the presentation in Lemma~\ref{lem_pres_LH} up to Tietze transformations. 

\begin{prop}\label{prop_pres_LH}
For $n\geq 1$, $\LH $ admits the presentation with generators $s_i$ for $1\leq i\leq n$, $r_i$ for $1\leq i\leq n$, $t_{i}$ for $1\leq i\leq n+1$, $r$, $p_{i,j}$, $x_{i,j}$, $y_{i,j}$ for $1\leq i<j\leq n+1$, and $s$, and the following defining relations: 
\begin{enumerate}
\item commutative relations
\begin{enumerate}
\item $\alpha _i \rightleftarrows \beta _j$ \quad for $j-i\geq 2$ and $\alpha , \beta \in \{ s, r\}$,
\item $\alpha _{i} \rightleftarrows t_{j}$ \quad for $j\not =i, i+1$ and $\alpha \in \{ s, r\}$, 
\item $t_i\rightleftarrows t_{j}$ \quad for $1\leq i<j\leq n+1$, 
\item $s_i \rightleftarrows r$ \quad for $1\leq i\leq n$, 
\item $t_i \rightleftarrows r$ \quad for $1\leq i\leq n+1$, 
\end{enumerate}
\item conjugation relations
\begin{enumerate}
\item $\alpha _i\alpha _{i+1}\alpha _i=\alpha _{i+1}\alpha _i\alpha _{i+1}$ \quad for $1\leq i\leq n-1$ and $\alpha \in \{ s, r\}$, 
\item $s_i^{\varepsilon }s_{i+1}^{\varepsilon }r_{i}=r_{i+1}s_i^{\varepsilon }s_{i+1}^{\varepsilon }$ \quad for $1\leq i\leq n-1$ and $\varepsilon \in \{ 1, -1\}$,
\item $r_ir_{i+1}s_{i}=s_{i+1}r_ir_{i+1}$ \quad for $1\leq i\leq n-1$,
\item $r_irs_{i}=rs_{i}r_{i}^{-1}$ \quad for $1\leq i\leq n$,
\item $s_i^{\pm 1}t_{i}=t_{i+1}s_{i}^{\pm 1 }$ \quad for $1\leq i\leq n$,
\item $r_it_{i}=t_{i+1}r_{i}$ \quad for $1\leq i\leq n$,
\item $t_is_{i}^2r_i=r_is_i^2t_{i+1}$ \quad for $1\leq i\leq n$,
\end{enumerate}
\item $r^2=t_1t_2\cdots t_{n+1}$, 
\item $r_1r_2\cdots r_ns_n\cdots s_2s_1t_{1}=1$, 
\item $t_{1}t_2\cdots t_{n+1}\bigl( s_1(s_2s_1)\cdots (s_{n-1}\cdots s_2s_1)(s_n\cdots s_2s_1)\bigr) ^2=1$, 
\item
\begin{enumerate}
\item $p_{i,i+1}=s_{i}^2$, \quad $x_{i,i+1}=s_{i}r_i^{-1}$, \quad $y_{i,i+1}=r_i^{-1}s_{i}$ \quad for $1\leq i\leq n$, 
\item $\alpha _{i,j}=s_{j-1}\cdots s_{i+2}s_{i+1}\alpha _{i,i+1}s_{i+1}^{-1}s_{i+2}^{-1}\cdots s_{j-1}^{-1}$\\
for $\alpha \in \{ p, x, y\}$ and $1\leq i<j-1\leq n$,
\item $s=s_n\cdots s_2s_1t_1$,
\item $s\alpha _{1,j}s^{-1}=\beta _{j-1, n+1}$ \quad for $(\alpha ,\beta )\in \{ (p,p), (x,y), (y,x)\}$ and $2\leq j\leq n+1$. 
\item $s\alpha _{i,j}s^{-1}=\alpha _{i-1,j-1}$ \quad for $\alpha \in \{ p, x, y\}$ and $2\leq i<j\leq n+1$,
\item $r\rightleftarrows p_{i,j}$  \quad for $1\leq i<j\leq n+1$, 
\item $r\alpha _{i,j}r^{-1}=\alpha _{i,j}^{-1}p_{i,j}$ \quad for $1\leq i<j\leq n+1$ and $\alpha \in \{ x, y\}$. 
\end{enumerate}
\end{enumerate}
\end{prop}

We remark that the relations~(1)--(5), (6) (f), and (g) in Proposition~\ref{prop_pres_LH} coincide with the relations~(1)--(5), (A2) (b), and (c) in Theorem~\ref{thm_pres_LH}, respectively. 

From here, a label of a relation in $\LH $ indicates one of relations in Proposition~\ref{prop_pres_LH}. 
For instance, $A\stackrel{\text{(1)(a)}}{=}B$ (resp. $A=B\stackrel{\text{(1)(a)}}{\Longleftrightarrow }A^\prime= B^\prime $) means that $B$ (resp. the relation $A^\prime =B^\prime $) is obtained from $A$ (resp. the relation $A=B$) by the relation~(1) (a) in Proposition~\ref{prop_pres_LH}. 
We denote $A^\prime \underline{A}=B\stackrel{\text{CONJ}}{\Longleftrightarrow }A^\prime =BA^{-1}$ and $\underline{A}A^\prime =B\stackrel{\text{CONJ}}{\Longleftrightarrow }A^\prime =A^{-1}B$, and ``$A_1\underset{\rightarrow }{\underline{A}}A_2$'' (resp. ``$A_1\underset{\leftarrow }{\underline{A}}A_2$'') in an equation means that deformation of the word $A_1AA_2$ by using commutative relations and moving $A$ right (resp. left). 
We have the following lemma. 

\begin{lem}\label{lem_rel_6d}
The relation (6) (d) in Proposition~\ref{prop_pres_LH} is obtained from the relations~(1), (2), and (6) (a)--(c) in Proposition~\ref{prop_pres_LH}. 
\end{lem}

\begin{proof}
In the case of $(\alpha ,\beta )=(p,p)$ for $j=2$, we have
\begin{eqnarray*}
s\alpha _{1,2}s^{-1}&=&sp_{1,2}s^{-1}\\
&\overset{\text{(6)(a),(c)}}{\underset{}{=}}&(s_n\cdots s_2s_1\underset{ }{\underline{t_1)s_{1}s_{1}}}(t_1^{-1}s_1^{-1}s_2^{-1}\cdots s_n^{-1})\\
&\overset{\text{(2)(e)}}{\underset{}{=}}&(s_n\cdots s_2)s_{1}s_{1}(s_2^{-1}\cdots s_n^{-1})\\
&\overset{\text{(6)(a),(b)}}{\underset{}{=}}&p_{1,n+1}=\beta _{1,n+1}.
\end{eqnarray*}

For $j\geq 3$, we have
\begin{eqnarray*}
&&s\alpha _{1,j}s^{-1}=sp_{1,j}s^{-1}\\
&\overset{\text{(6)(a)-(c)}}{\underset{}{=}}&(s_n\cdots s_2s_1\underset{\rightarrow }{\underline{t_1}})s_{j-1}\cdots s_{3}s_{2}\cdot s_{1}s_{1}\cdot s_2^{-1}s_3^{-1}\cdots s_{j-1}^{-1}(t_1^{-1}s_1^{-1}s_2^{-1}\cdots s_n^{-1})\\
&\overset{\text{(1)(b)}}{\underset{\text{(2)(e)}}{=}}&(s_n\cdots s_2\underset{\rightarrow }{\underline{s_1}})s_{j-1}\cdots s_{3}s_{2}\cdot s_{1}s_{1}\cdot s_2^{-1}s_3^{-1}\cdots s_{j-1}^{-1}(\underset{\leftarrow }{\underline{s_1^{-1}}}s_2^{-1}\cdots s_n^{-1})\\
&\overset{\text{(1)(a)}}{\underset{}{=}}&(s_n\cdots s_3s_2)s_{j-1}\cdots s_{3}\cdot \underline{s_1s_{2}\cdot s_{1}}\ \underline{s_{1}\cdot s_2^{-1}s_1^{-1}}\cdot s_3^{-1}\cdots s_{j-1}^{-1}(s_2^{-1}s_3^{-1}\cdots s_n^{-1})\\
&\overset{\text{(2)(a)}}{\underset{}{=}}&(s_n\cdots s_3\underset{\rightarrow }{\underline{s_2}})s_{j-1}\cdots s_{3}\cdot s_2s_2\cdot s_3^{-1}\cdots s_{j-1}^{-1}(\underset{\leftarrow }{\underline{s_2^{-1}}}s_3^{-1}\cdots s_n^{-1})\\
&\overset{\text{(1)(a)}}{\underset{}{=}}&(s_n\cdots s_4s_3)s_{j-1}\cdots s_4\cdot \underline{s_2s_{3}\cdot s_2}\ \underline{s_2\cdot s_3^{-1}s_2^{-1}}\cdot s_4^{-1}\cdots s_{j-1}^{-1}(s_3^{-1}s_4^{-1}\cdots s_n^{-1})\\
&\overset{\text{(2)(a)}}{\underset{}{=}}&(s_n\cdots s_4s_3)s_{j-1}\cdots s_4\cdot s_3s_3\cdot s_4^{-1}\cdots s_{j-1}^{-1}(s_3^{-1}s_4^{-1}\cdots s_n^{-1})\\
&\vdots &\\
&\overset{}{\underset{}{=}}&(s_n\cdots s_{j-1}\underline{s_{j-2})s_{j-1}\cdot s_{j-2}}\ \underline{s_{j-2}\cdot s_{j-1}^{-1}(s_{j-2}^{-1}}s_{j-1}^{-1}\cdots s_n^{-1})\\
&\overset{\text{(2)(a)}}{\underset{}{=}}&(s_n\cdots s_{j})s_{j-1}^2(s_{j}^{-1}\cdots s_n^{-1})\\
&\overset{\text{(6)(a),(b)}}{\underset{}{=}}&p_{j-1,n+1}=\beta _{j-1,n+1}.
\end{eqnarray*}
Thus the relation (6) (d) in Proposition~\ref{prop_pres_LH} for $(\alpha ,\beta )=(p,p)$ is obtained from the relations~(1), (2), and (6) (a)--(c) in Proposition~\ref{prop_pres_LH}. 

In the case of $(\alpha ,\beta )=(x,y)$ for $j=2$, first, we have
\begin{eqnarray}
\text{the rel. (2) (g)}&\Longleftrightarrow & \underline{t_is_{i}^2}r_i=r_is_i\underline{s_it_{i+1}}\overset{\text{(2)(e)}}{\underset{}{\Longleftrightarrow }}s_{i}\underline{s_{i}t_ir_i}=\underline{r_i}s_it_{i}s_i\notag \\
&\overset{\text{CONJ}}{\underset{}{\Longleftrightarrow }}& r_i^{-1}s_{i}=s_it_{i}s_ir_i^{-1}t_i^{-1}s_{i}^{-1},\label{rel_2g} 
\end{eqnarray}
Hence we have  
\begin{eqnarray*}
s\alpha _{1,2}s^{-1}&=&sx_{1,2}s^{-1}\\
&\overset{\text{(6)(a),(c)}}{\underset{}{=}}&(s_n\cdots s_2\underline{s_1t_1)s_{1}r_{1}^{-1}(t_1^{-1}s_1^{-1}}s_2^{-1}\cdots s_n^{-1})\\
&\overset{\text{(2)(e),(g)}}{\underset{}{=}}&(s_n\cdots s_2)r_{1}^{-1}s_{1}(s_2^{-1}\cdots s_n^{-1})\\
&\overset{\text{(6)(a),(b)}}{\underset{}{=}}&y_{1,n+1}=\beta _{1,n+1}.
\end{eqnarray*}
For $j\geq 3$, first, we have
\begin{eqnarray*}
\underline{s_is_{i+1}t_i}s_{i}r_{i}^{-1}\underline{t_i^{-1}s_{i+1}^{-1}s_i^{-1}}&\overset{\text{(1)(b)}}{\underset{\text{(2)(e)}}{=}}&t_{i+1}\underline{s_is_{i+1}s_{i}}r_{i}^{-1}s_{i+1}^{-1}s_i^{-1}t_{i+1}^{-1}\\
&\overset{\text{(2)(a)}}{\underset{}{=}}&t_{i+1}s_{i+1}\underline{s_is_{i+1}r_{i}^{-1}}s_{i+1}^{-1}s_i^{-1}t_{i+1}^{-1}\\
&\overset{\text{(2)(b)}}{\underset{}{=}}&t_{i+1}s_{i+1}r_{i+1}^{-1}t_{i+1}^{-1}
\end{eqnarray*}
for $1\leq i\leq n-1$. 
Hence, by using this relation inductively, we have
\begin{eqnarray*}
&&s\alpha _{1,j}s^{-1}=sx_{1,j}s^{-1}\\
&\overset{\text{(6)(a)-(c)}}{\underset{}{=}}&(s_n\cdots s_2\underset{\rightarrow }{\underline{s_1t_1}})s_{j-1}\cdots s_{3}s_{2}\cdot s_{1}r_{1}^{-1}\cdot s_2^{-1}s_3^{-1}\cdots s_{j-1}^{-1}(\underset{\leftarrow }{\underline{t_1^{-1}s_1^{-1}}}s_2^{-1}\cdots s_n^{-1})\\
&\overset{\text{(1)(a)}}{\underset{\text{(1)(b)}}{=}}&(s_n\cdots s_2)s_{j-1}\cdots s_{3}\cdot \underline{s_1s_{2}t_1s_{1}r_{1}^{-1}t_1^{-1}s_2^{-1}s_1^{-1}}\cdot s_3^{-1}\cdots s_{j-1}^{-1}(s_2^{-1}\cdots s_n^{-1})\\
&\overset{}{\underset{}{=}}&(s_n\cdots \underset{\rightarrow }{\underline{s_2}})s_{j-1}\cdots s_{3}\cdot t_2s_{2}r_{2}^{-1}t_2^{-1}\cdot s_3^{-1}\cdots s_{j-1}^{-1}(\underset{\leftarrow }{\underline{s_2^{-1}}}\cdots s_n^{-1})\\
&\overset{\text{(1)(a)}}{\underset{}{=}}&(s_n\cdots s_3)s_{j-1}\cdots s_{4}\cdot \underline{s_2s_3t_2s_{2}r_{2}^{-1}t_2^{-1}s_3^{-1}s_2^{-1}}\cdot s_4^{-1}\cdots s_{j-1}^{-1}(s_3^{-1}\cdots s_n^{-1})\\
&\overset{}{\underset{}{=}}&(s_n\cdots s_3)s_{j-1}\cdots s_{4}\cdot t_3s_{3}r_{3}^{-1}t_3^{-1}\cdot s_4^{-1}\cdots s_{j-1}^{-1}(s_3^{-1}\cdots s_n^{-1})\\
&\vdots &\\
&\overset{}{\underset{}{=}}&(s_n\cdots s_{j-1}\underline{s_{j-2})s_{j-1}\cdot t_{j-2}s_{j-2}r_{j-2}^{-1}t_{j-2}^{-1}\cdot s_{j-1}^{-1}(s_{j-2}^{-1}}s_{j-1}^{-1}\cdots s_n^{-1})\\
&\overset{}{\underset{}{=}}&(s_n\cdots s_{j}s_{j-1})t_{j-1}s_{j-1}r_{j-1}^{-1}t_{j-1}^{-1}(s_{j-1}^{-1}s_{j}^{-1}\cdots s_n^{-1})\\
&\overset{}{\underset{}{=}}&(s_n\cdots s_{j}\underline{s_{j-1})t_{j-1}s_{j-1}r_{j-1}^{-1}t_{j-1}^{-1}(s_{j-1}^{-1}}s_{j}^{-1}\cdots s_n^{-1})\\
&\overset{\text{(\ref{rel_2g})}}{\underset{}{=}}&(s_n\cdots s_{j+1}s_{j})r_{j-1}^{-1}s_{j-1}(s_{j}^{-1}s_{j+1}^{-1}\cdots s_n^{-1})\\
&\overset{\text{(6)(a),(b)}}{\underset{}{=}}&y_{j-1,n+1}=\beta _{j-1,n+1}.
\end{eqnarray*}
Thus the relation (6) (d) in Proposition~\ref{prop_pres_LH} for $(\alpha ,\beta )=(x,y)$ is obtained from the relations~(1), (2), and (6) (a)--(c) in Proposition~\ref{prop_pres_LH}. 

In the case of $(\alpha ,\beta )=(y,x)$ for $j=2$, we have
\begin{eqnarray*}
s\alpha _{1,2}s^{-1}&=&sy_{1,2}s^{-1}\\
&\overset{\text{(6)(a),(c)}}{\underset{}{=}}&(s_n\cdots s_2s_1\underline{t_1)r_{1}^{-1}}\ \underline{s_{1}(t_1^{-1}}s_1^{-1}s_2^{-1}\cdots s_n^{-1})\\
&\overset{\text{(2)(e),(f)}}{\underset{}{=}}&(s_n\cdots s_2)s_1r_{1}^{-1}(s_2^{-1}\cdots s_n^{-1})\\
&\overset{\text{(6)(a),(b)}}{\underset{}{=}}&x_{1,n+1}=\beta _{1,n+1}.
\end{eqnarray*}
For $j\geq 3$, we have
\begin{eqnarray*}
&&s\alpha _{1,j}s^{-1}=sy_{1,j}s^{-1}\\
&\overset{\text{(6)(a)-(c)}}{\underset{}{=}}&(s_n\cdots s_2s_1\underset{\rightarrow }{\underline{t_1}})s_{j-1}\cdots s_{3}s_{2}\cdot r_{1}^{-1}s_{1}\cdot s_2^{-1}s_3^{-1}\cdots s_{j-1}^{-1}(t_1^{-1}s_1^{-1}s_2^{-1}\cdots s_n^{-1})\\
&\overset{\text{(1)(b)}}{\underset{\text{(2)(e),(f)}}{=}}&(s_n\cdots s_2\underset{\rightarrow }{\underline{s_1}})s_{j-1}\cdots s_{3}s_{2}\cdot r_{1}^{-1}s_{1}\cdot s_2^{-1}s_3^{-1}\cdots s_{j-1}^{-1}(\underset{\leftarrow }{\underline{s_1^{-1}}}s_2^{-1}\cdots s_n^{-1})\\
&\overset{\text{(1)(a)}}{\underset{}{=}}&(s_n\cdots s_3s_2)s_{j-1}\cdots s_{3}\cdot \underline{s_1s_{2}\cdot r_{1}^{-1}}\ \underline{s_{1}\cdot s_2^{-1}s_1^{-1}}\cdot s_3^{-1}\cdots s_{j-1}^{-1}(s_2^{-1}s_3^{-1}\cdots s_n^{-1})\\
&\overset{\text{(2)(a),(b)}}{\underset{}{=}}&(s_n\cdots s_3\underset{\rightarrow }{\underline{s_2}})s_{j-1}\cdots s_{3}\cdot r_{2}^{-1}s_2\cdot s_3^{-1}\cdots s_{j-1}^{-1}(\underset{\leftarrow }{\underline{s_2^{-1}}}s_3^{-1}\cdots s_n^{-1})\\
&\overset{\text{(1)(b)}}{\underset{}{=}}&(s_n\cdots s_4s_3)s_{j-1}\cdots s_4\cdot \underline{s_2s_{3}\cdot r_{2}^{-1}}\ \underline{s_2\cdot s_3^{-1}s_2^{-1}}\cdot s_4^{-1}\cdots s_{j-1}^{-1}(s_3^{-1}s_4^{-1}\cdots s_n^{-1})\\
&\overset{\text{(2)(a),(b)}}{\underset{}{=}}&(s_n\cdots s_4s_3)s_{j-1}\cdots s_4\cdot r_{3}^{-1}s_3\cdot s_4^{-1}\cdots s_{j-1}^{-1}(s_3^{-1}s_4^{-1}\cdots s_n^{-1})\\
&\vdots &\\
&\overset{}{\underset{}{=}}&(s_n\cdots s_js_{j-1}\underline{s_{j-2})s_{j-1}\cdot r_{j-2}^{-1}}\ \underline{s_{j-2}\cdot s_{j-1}^{-1}(s_{j-2}^{-1}}s_{j-1}^{-1}s_j^{-1}\cdots s_n^{-1})\\
&\overset{\text{(2)(a),(b)}}{\underset{}{=}}&(s_n\cdots s_{j})s_{j-1}r_{j-1}^{-1}(s_{j}^{-1}\cdots s_n^{-1})\\
&\overset{\text{(6)(a),(b)}}{\underset{}{=}}&x_{j-1,n+1}=\beta _{j-1,n+1}.
\end{eqnarray*}
Therefore, the relation (6) (d) in Proposition~\ref{prop_pres_LH} for $(\alpha ,\beta )=(y,x)$ is obtained from the relations~(1), (2), and (6) (a)--(c) in Proposition~\ref{prop_pres_LH} and we have completed the proof of Lemma~\ref{lem_rel_6d}. 
\end{proof}

To prove that the relation (6) (e) in Proposition~\ref{prop_pres_LH} is obtained from the relations~(1), (2), and (6) (a)--(c) in Proposition~\ref{prop_pres_LH}, we prepaire the following lemma.

\begin{lem}\label{lem_rel_6e_pre}
For $2\leq i\leq n$, $\varepsilon \in \{ 1,-1\}$, and $\alpha \in \{ p, x, y\}$, the relation 
\[
s_{i-1}^\varepsilon \alpha _{i,i+1}s_{i-1}^{-\varepsilon }=s_i^{-\varepsilon }\alpha _{i-1,i}s_i^\varepsilon 
\]
is obtained from the relations~(1), (2), and (6) (a)--(c) in Proposition~\ref{prop_pres_LH}.  
\end{lem}

\begin{proof}
For $2\leq i\leq n$ and $\varepsilon \in \{ 1,-1\}$, since we have
\begin{enumerate}
\item[(p)] $s_{i-1}^\varepsilon s_i^2s_{i-1}^{-\varepsilon }=\underline{s_{i-1}^\varepsilon s_is_{i-1}^{-\varepsilon }}\ \underline{s_{i-1}^\varepsilon s_is_{i-1}^{-\varepsilon }}\overset{\text{(2)(a)}}{\underset{}{=}}s_i^{-\varepsilon }s_{i-1}^2s_i^\varepsilon $,
\item[(x)] $s_{i-1}^\varepsilon s_ir_i^{-1}s_{i-1}^{-\varepsilon }=s_{i-1}^\varepsilon s_i\underline{r_i^{-1}s_{i-1}^{-\varepsilon }s_{i}^{-\varepsilon }}s_i^\varepsilon \overset{\text{(2)(b)}}{\underset{}{=}}\underline{s_{i-1}^\varepsilon s_is_{i-1}^{-\varepsilon }}s_{i}^{-\varepsilon }r_{i-1}^{-1}s_i^\varepsilon $\\
\hspace{1.9cm} $\overset{\text{(2)(a)}}{\underset{}{=}}s_i^{-\varepsilon }s_{i-1}r_{i-1}^{-1}s_i^\varepsilon $,
\item[(y)] $s_{i-1}^\varepsilon r_i^{-1}s_is_{i-1}^{-\varepsilon }=s_{i-1}^\varepsilon r_i^{-1}s_{i-1}^{-\varepsilon }\underline{s_{i-1}^\varepsilon s_is_{i-1}^{-\varepsilon }}\overset{\text{(2)(b)}}{\underset{}{=}}s_{i-1}^\varepsilon \underline{r_i^{-1}s_{i-1}^{-\varepsilon }s_{i}^{-\varepsilon }}s_{i-1}s_i^\varepsilon $\\
\hspace{1.9cm} $\overset{\text{(2)(a)}}{\underset{}{=}}s_i^{-\varepsilon }r_{i-1}^{-1}s_{i-1}s_i^\varepsilon $,
\end{enumerate}
the relation $s_{i-1}^\varepsilon \alpha _{i,i+1}s_{i-1}^{-\varepsilon }=s_i^{-\varepsilon }\alpha _{i-1,i}s_i^\varepsilon $ for $\alpha \in \{ p, x, y\}$ is obtained from the relations~(1), (2), and (6) (a)--(c) in Proposition~\ref{prop_pres_LH}.  
Therefore, we have completed the proof of Lemma~\ref{lem_rel_6e_pre}. 
\end{proof}

\begin{lem}\label{lem_rel_6e}
The relation (6) (e) in Proposition~\ref{prop_pres_LH} is obtained from the relations~(1), (2), and (6) (a)--(c) in Proposition~\ref{prop_pres_LH}. 
\end{lem}

\begin{proof}
For $2\leq i=j-1\leq n$ and $\alpha \in \{ p, x, y\}$, we have
\begin{eqnarray*}
s\alpha _{i,i+1}s^{-1}
&\overset{\text{(6)(c)}}{\underset{}{=}}&(s_n\cdots s_{i-1}\underset{\rightarrow }{\underline{s_{i-2}\cdots s_1t_1}})\alpha _{i,i+1}(t_1^{-1}s_1^{-1}s_2^{-1}\cdots s_n^{-1})\\
&\overset{\text{(1)}}{\underset{\text{(6)(a)}}{=}}&(s_n\cdots s_{i}\underline{s_{i-1})\alpha _{i,i+1}(s_{i-1}^{-1}}s_{i}^{-1}\cdots s_n^{-1})\\
&\overset{\text{Lem.\ref{lem_rel_6e_pre}}}{\underset{}{=}}&(s_n\cdots s_{i})s_i^{-1}\alpha _{i-1,i}s_i(s_{i}^{-1}\cdots s_n^{-1})\\
&\overset{}{\underset{}{=}}&\underset{\rightarrow }{\underline{(s_n\cdots s_{i+1})}}\alpha _{i-1,i}(s_{i+1}^{-1}\cdots s_n^{-1})\\
&\overset{\text{(1)}}{\underset{\text{(6)(a)}}{=}}&\alpha _{i-1,i}. 
\end{eqnarray*}
Thus, the relation (6) (e) in Proposition~\ref{prop_pres_LH} for $2\leq i=j-1\leq n$ is obtained from the relations~(1), (2), and (6) (a)--(c) in Proposition~\ref{prop_pres_LH}.  

For $2\leq i<j-1\leq n$ and $\alpha \in \{ p, x, y\}$, we have
\begin{eqnarray*}
&&s\alpha _{i,j}s^{-1}\\
&\overset{\text{(6)(b)}}{\underset{\text{(6)(c)}}{=}}&(s_n\cdots s_{i-1}\underset{\rightarrow }{\underline{s_{i-2}\cdots s_1t_1}})s_{j-1}\cdots s_{i+2}s_{i+1}\alpha _{i,i+1}s_{i+1}^{-1}s_{i+2}^{-1}\cdots s_{j-1}^{-1}(t_1^{-1}s_1^{-1}s_2^{-1}\cdots s_n^{-1})\\
&\overset{\text{(1)}}{\underset{\text{(6)(a)}}{=}}&(s_n\cdots s_{i}\underset{\rightarrow }{\underline{s_{i-1}}})s_{j-1}\cdots s_{i+2}s_{i+1}\alpha _{i,i+1}s_{i+1}^{-1}s_{i+2}^{-1}\cdots s_{j-1}^{-1}(\underset{\leftarrow }{\underline{s_{i-1}^{-1}}}s_{i}^{-1}\cdots s_n^{-1})\\
&\overset{\text{(1)(a)}}{\underset{}{=}}&(s_n\cdots s_{i+1}s_{i})s_{j-1}\cdots s_{i+2}s_{i+1}\cdot \underline{s_{i-1}\alpha _{i,i+1}s_{i-1}^{-1}}\cdot s_{i+1}^{-1}s_{i+2}^{-1}\cdots s_{j-1}^{-1}(s_{i}^{-1}s_{i+1}^{-1}\cdots s_n^{-1})\\
&\overset{\text{Lem.\ref{lem_rel_6e_pre}}}{\underset{}{=}}&(s_n\cdots s_{i+1}\underset{\rightarrow }{\underline{s_{i}}})s_{j-1}\cdots s_{i+2}s_{i+1}\cdot s_i^{-1}\alpha _{i-1,i}s_i\cdot s_{i+1}^{-1}s_{i+2}^{-1}\cdots s_{j-1}^{-1}(\underset{\leftarrow }{\underline{s_{i}^{-1}}}s_{i+1}^{-1}\cdots s_n^{-1})\\
&\overset{\text{(1)(a)}}{\underset{}{=}}&(s_n\cdots s_{i+2}s_{i+1})s_{j-1}\cdots s_{i+2}\cdot \underline{s_is_{i+1}s_i^{-1}}\alpha _{i-1,i}\underline{s_is_{i+1}^{-1}s_i^{-1}}\\
&&\cdot s_{i+2}^{-1}\cdots s_{j-1}^{-1}(s_{i+1}^{-1}s_{i+2}^{-1}\cdots s_n^{-1})\\
&\overset{\text{(2)(a)}}{\underset{\text{(1)}}{=}}&(s_n\cdots s_{i+2}\underset{\rightarrow }{\underline{s_{i+1}}})s_{j-1}\cdots s_{i+2}\cdot s_{i+1}^{-1}\cdot s_i\alpha _{i-1,i}s_i^{-1}\cdot s_{i+1}\\
&&\cdot s_{i+2}^{-1}\cdots s_{j-1}^{-1}(\underset{\leftarrow }{\underline{s_{i+1}^{-1}}}s_{i+2}^{-1}\cdots s_n^{-1})\\
&\overset{\text{(1)(a)}}{\underset{}{=}}&(s_n\cdots s_{i+3}s_{i+2})s_{j-1}\cdots s_{i+3}\cdot \underline{s_{i+1}s_{i+2}s_{i+1}^{-1}}s_i\alpha _{i-1,i}s_i^{-1}\underline{s_{i+1}s_{i+2}^{-1}s_{i+1}^{-1}}\\
&&\cdot s_{i+3}^{-1}\cdots s_{j-1}^{-1}(s_{i+2}^{-1}s_{i+3}^{-1}\cdots s_n^{-1})\\
&\overset{\text{(2)(a)}}{\underset{\text{(1)}}{=}}&(s_n\cdots s_{i+3}s_{i+2})s_{j-1}\cdots s_{i+3}\cdot s_{i+2}^{-1}\cdot s_{i+1}s_i\alpha _{i-1,i}s_i^{-1}s_{i+1}^{-1}\cdot s_{i+2}\\
&&\cdot s_{i+3}^{-1}\cdots s_{j-1}^{-1}(s_{i+2}^{-1}s_{i+3}^{-1}\cdots s_n^{-1})\\
&\vdots &\\
&\overset{}{\underset{}{=}}&(s_n\cdots s_{j-1}\underline{s_{j-2})s_{j-1}\cdot s_{j-2}^{-1}}\cdot s_{j-3}\cdots s_{i+1}s_i\alpha _{i-1,i}s_i^{-1}s_{i+1}^{-1}\cdots s_{j-3}^{-1}\\
&&\cdot \underline{s_{j-2}^{-1}\cdot s_{j-1}^{-1}(s_{j-2}^{-1}}s_{j-1}^{-1}\cdots s_n^{-1})\\
&\overset{\text{(2)(a)}}{\underset{\text{(1)}}{=}}&\underset{\rightarrow }{\underline{(s_n\cdots s_{j})}}s_{j-2}\cdots s_{i+1}s_i\alpha _{i-1,i}s_i^{-1}s_{i+1}^{-1}\cdots s_{j-2}^{-1}(s_{j}^{-1}\cdots s_n^{-1})\\
&\overset{\text{(1)}}{\underset{\text{(6)(a)}}{=}}&s_{j-2}\cdots s_{i+1}s_i\alpha _{i-1,i}s_i^{-1}s_{i+1}^{-1}\cdots s_{j-2}^{-1}\\
&\overset{\text{(6)(b)}}{\underset{}{=}}&\alpha _{i-1,j-1}. 
\end{eqnarray*}
Therefore, we have completed the proof of Lemma~\ref{lem_rel_6e}. 
\end{proof}

By the relations~(1) (d), (6) (a), and (b), the next lemma holds immediately. 

\begin{lem}\label{lem_rel_6f}
The relation (6) (f) in Proposition~\ref{prop_pres_LH} is obtained from the relations~(1), (2), and (6) (a)--(c) in Proposition~\ref{prop_pres_LH}. 
\end{lem}

\begin{lem}\label{lem_rel_6g}
The relation (6) (g) in Proposition~\ref{prop_pres_LH} is obtained from the relations~(1), (2), and (6) (a)--(c) in Proposition~\ref{prop_pres_LH}. 
\end{lem}

\begin{proof}
First, up to the relations~(1) (d), the relation~(2) (d) is equivalent to the following relation. 
\begin{eqnarray*}
r_irs_{i}=rs_{i}r_{i}^{-1}\Longleftrightarrow r_is_{i}=rs_{i}r_{i}^{-1}r^{-1}\Longleftrightarrow s_{i}^{-1}r_is_{i}=rr_{i}^{-1}r^{-1}. 
\end{eqnarray*} 
Hence, since we have
\begin{enumerate}
\item[(x)] $rs_ir_i^{-1}r^{-1}=r_is_{i}=r_is_{i}^{-1}\cdot s_{i}^2$,
\item[(y)] $rr_i^{-1}s_i\underset{\leftarrow }{\underline{r^{-1}}}\overset{\text{(1)(d)}}{\underset{}{=}}\underline{rr_i^{-1}r^{-1}}s_i=s_{i}^{-1}r_i\cdot s_i^2$
\end{enumerate}
for $1\leq i\leq n$, the relation $r\alpha _{i,i+1}r^{-1}=\alpha _{i,i+1}^{-1}p_{i,i+1}$ for $\alpha \in \{ x, y\}$ is obtained from the relations~(1), (2), and (6) (a)--(c) in Proposition~\ref{prop_pres_LH}.  
Hence, for $\alpha \in \{ x, y\}$ and $1\leq i<j-1\leq n$, we have
\begin{eqnarray*}
&&r\alpha _{i,j}r^{-1}\\
&\overset{\text{(6)(b)}}{\underset{}{=}}&\underset{\rightarrow }{\underline{r}}s_{j-1}\cdots s_{i+2}s_{i+1}\alpha _{i,i+1}s_{i+1}^{-1}s_{i+2}^{-1}\cdots s_{j-1}^{-1}\underset{\leftarrow }{\underline{r^{-1}}}\\
&\overset{\text{(1)(d)}}{\underset{}{=}}&s_{j-1}\cdots s_{i+2}s_{i+1}\underline{r\alpha _{i,i+1}r^{-1}}s_{i+1}^{-1}s_{i+2}^{-1}\cdots s_{j-1}^{-1}\\
&\overset{}{\underset{}{=}}&s_{j-1}\cdots s_{i+2}s_{i+1}\alpha _{i,i+1}^{-1}p_{i,i+1}s_{i+1}^{-1}s_{i+2}^{-1}\cdots s_{j-1}^{-1}\\
&\overset{}{\underset{}{=}}&s_{j-1}\cdots s_{i+2}s_{i+1}\alpha _{i,i+1}^{-1}s_{i+1}^{-1}s_{i+2}^{-1}\cdots s_{j-1}^{-1}\cdot s_{j-1}\cdots s_{i+2}s_{i+1}p_{i,i+1}s_{i+1}^{-1}s_{i+2}^{-1}\cdots s_{j-1}^{-1}\\
&\overset{\text{(6)(b)}}{\underset{}{=}}&\alpha _{i,j}^{-1}p_{i,j}. 
\end{eqnarray*}
Therefore, we have completed the proof of Lemma~\ref{lem_rel_6g}. 
\end{proof}

By using Proposition~\ref{prop_pres_LH} and lemmas above, we will prove Theorem~\ref{thm_pres_LH} as follows. 

\begin{proof}[Proof of Theorem~\ref{thm_pres_LH}]
The finite presentation for $\LH $ in Proposition~\ref{prop_pres_LH} is obtained from the presentation for $\LH $ in Theorem~\ref{thm_pres_LH} by adding the generators $p_{i,j}$, $x_{i,j}$, $y_{i,j}$ for $1\leq i<j\leq n+1$, and $s$ and the relations~(6) (a)--(g) in Proposition~\ref{prop_pres_LH}. 
By Lemmas~\ref{lem_rel_6d}, \ref{lem_rel_6e}, \ref{lem_rel_6f}, and \ref{lem_rel_6d}, $\LH $ admits the presentation which is obtained from the presentation in Proposition~\ref{prop_pres_LH} by removing the relations~(6) (d)--(g). 
Since the generators $p_{i,j}$, $x_{i,j}$, $y_{i,j}$ for $1\leq i<j\leq n+1$, and $s$ do not appear in the relations~(1)--(5), by Tietze transformations, we can remove the generators $p_{i,j}$, $x_{i,j}$, $y_{i,j}$ for $1\leq i<j\leq n+1$, and $s$ and the relations~(6) (a)--(c) from this presentation for $\LH $.  
Therefore, we have completed the proof of Theorem~\ref{thm_pres_LH}. 
\end{proof}

\subsection{The first homology groups of the liftable Hilden group}\label{section_abel-lmod}

Recall the integral first homology group $H_1(G)$ of a group $G$ is isomorphic to the abelianization of $G$. 
In this section, we prove Theorem~\ref{thm_abel_lmod} by using the finite presentation in Theorem~\ref{thm_pres_LH}. 
For conveniences, we denote the equivalence class in $H_1(\LH )$ of an element $h$ in $\LH $ by $h$. 

\begin{proof}[Proof of Theorem~\ref{thm_abel_lmod}]
The relations~(2) (a), (b), and (c) in Theorem~\ref{thm_pres_LH} are equivalent to the relations $s_i=s_{i+1}$ and $r_i=r_{i+1}$ for $1\leq i\leq n-1$ in $H_1(\LH )$. 
The relation~(2) (d) in Theorem~\ref{thm_pres_LH} is equivalent to the relation $r_i^2=1$ for $1\leq i\leq n$ in $H_1(\LH )$. 
The relations~(2) (e), (f), and (g) in Theorem~\ref{thm_pres_LH} are equivalent to the relation $t_{i}=t_{i+1}$ for $1\leq i\leq n$ in $H_1(\LH )$. 
Up to these relations, the relations~(3), (4), and (5) in Theorem~\ref{thm_pres_LH} are equivalent to the relations $r^2t_1^{-n-1}=1$, $t_1=(r_1s_1)^{-n}$, and $t_1^{n+1}s_1^{n(n+1)}=1$ in $H_1(\LH )$, respectively. 
Hence, as a presentation for an abelian group, we have
\begin{eqnarray*}
&&H_1(\LH )\\
&\cong &\left< s_1, r_1, t_1, r \middle| r_1^2=1, r^2t_1^{-n-1}=1, t_1=(r_1s_1)^{-n}, t_1^{n+1}s_1^{n(n+1)}=1\right> \\
&\cong &\left< s_1, r_1, r \middle| r_1^2=1, r^2(r_1s_1)^{n(n+1)}=1, r_1^{-n(n+1)}=1\right> \\
&\cong &\left< s_1, r_1, r \middle| r_1^2=1, r^2(r_1s_1)^{n(n+1)}=1\right> \\
&\cong &\left< s_1, r_1, r, X \middle| r_1^2=1, X^2=1, X=r(r_1s_1)^{\frac{n(n+1)}{2}}\right> \\
&\cong &\left< s_1, r_1, X \middle| r_1^2=1, X^2=1\right> \\
&\cong &\Z [s_1]\oplus \Z _2[r_1]\oplus \Z _2[X],
\end{eqnarray*}
where $X=r(r_1s_1)^{\frac{n(n+1)}{2}}$. 
Therefore,  $H_1(\LH )$ is isomorphic to $\Z \oplus \Z _2\oplus \Z _2$ and we have completed the proof of Theorem~\ref{thm_abel_lmod}. 
\end{proof}

\subsection{The proof of Proposition~\ref{prop_pres_LH}}\label{section_proof_prop}

In this section, we prove Proposition~\ref{prop_pres_LH} by showing that the two presentations in Lemma~\ref{lem_pres_LH} and Proposition~\ref{prop_pres_LH} are equivalent up to Tietze transformations. 

First, the following lemma immediately holds by using the relations~(1) (b) and (2) (e) in Proposition~\ref{prop_pres_LH}. 

\begin{lem}\label{lem_Cst}
For $1\leq i<j\leq n+1$ and $\varepsilon \in \{ 1, -1\}$, the following relations are obtained from the relations (1) (b) and (2) (e) in Proposition~\ref{prop_pres_LH}. 
\begin{enumerate}
\item $(s_{j-1}^\varepsilon \cdots s_{i+1}^\varepsilon s_{i}^\varepsilon )t_{i}=t_{j}(s_{j-1}^\varepsilon \cdots s_{i+1}^\varepsilon s_{i}^\varepsilon )$, \\ 
$(s_{i}^\varepsilon s_{i+1}^\varepsilon \cdots s_{j-1}^\varepsilon )t_{j}=t_{i}(s_{i}^\varepsilon s_{i+1}^\varepsilon \cdots s_{j-1}^\varepsilon )$,
\item $(s_{j-1}^\varepsilon \cdots s_{i+1}^\varepsilon s_{i}^\varepsilon )t_{k}=t_{k-1}(s_{j-1}^\varepsilon \cdots s_{i+1}^\varepsilon s_{i}^\varepsilon )$, \\ 
$(s_{i}^\varepsilon s_{i+1}^\varepsilon \cdots s_{j-1}^\varepsilon )t_{k-1}=t_{k}(s_{i}^\varepsilon s_{i+1}^\varepsilon \cdots s_{j-1}^\varepsilon )$\quad for $i+1\leq k\leq j$. 
\end{enumerate}
\end{lem}

\begin{lem}\label{lem_Ct}
The relations~(C-pt), (C-tt), (C-xt), and (C-yt) in Lemma~\ref{lem_pres_LH} are obtained from the relations (1), (2), (6) (a), and (b) in Proposition~\ref{prop_pres_LH}. 
\end{lem}

\begin{proof}
The relations~(C-pt), (C-tt), and (C-xt) and (C-yt) for $k\not \in \{ i, j\}$ in Lemma~\ref{lem_pres_LH} are clearly obtained from the relations (1) (b), (c), (2) (e), (6) (a), and (b) in Proposition~\ref{prop_pres_LH} and Lemma~\ref{lem_Cst}. 
Since we have
\begin{enumerate}
\item[(x)] $\underline{x_{i,i+1}}t_{i+1}\overset{\text{(6)(a)}}{\underset{}{=}}s_i\underline{r_i^{-1}\cdot t_{i+1}}\overset{\text{(2)(f)}}{\underset{}{=}}\underline{s_it_{i}}r_i^{-1}\overset{\text{(2)(e)}}{\underset{}{=}}t_{i+1}\cdot \underline{s_ir_i^{-1}}\overset{\text{(6)(a)}}{\underset{}{=}}t_{i+1}x_{i,i+1}$,
\item[(y)] $\underline{y_{i,i+1}}t_{i}\overset{\text{(6)(a)}}{\underset{}{=}}r_i^{-1}\underline{s_i\cdot t_{i}}\overset{\text{(2)(e)}}{\underset{}{=}}\underline{r_i^{-1}t_{i+1}}s_i\overset{\text{(2)(f)}}{\underset{}{=}}t_{i}\cdot \underline{r_i^{-1}s_i}\overset{\text{(6)(a)}}{\underset{}{=}}t_{i}y_{i,i+1}$
\end{enumerate}
for $1\leq i\leq n$, the relations (C-xt) for $k=j=i+1$ and (C-yt) for $k=i=j-1$ in Lemma~\ref{lem_pres_LH} are obtained from the relations (1), (2), (6) (a), and (b) in Proposition~\ref{prop_pres_LH}. 
For $j-i\geq 2$, we have
\begin{eqnarray*}
\underline{x_{i,j}}t_{j}&\overset{\text{(6)(b)}}{\underset{}{=}}&(s_{j-1}\cdots s_{i+2}s_{i+1})x_{i,i+1}\underline{(s_{i+1}^{-1}s_{i+2}^{-1}\cdots s_{j-1}^{-1})t_{j}}\\
&\overset{\text{Lem.\ref{lem_Cst}}}{\underset{}{=}}&(s_{j-1}\cdots s_{i+2}s_{i+1})\underline{x_{i,i+1}t_{i+1}}(s_{i+1}^{-1}s_{i+2}^{-1}\cdots s_{j-1}^{-1})\\
&\overset{}{\underset{}{=}}&\underline{(s_{j-1}\cdots s_{i+2}s_{i+1})t_{i+1}}x_{i,i+1}(s_{i+1}^{-1}s_{i+2}^{-1}\cdots s_{j-1}^{-1})\\
&\overset{\text{Lem.\ref{lem_Cst}}}{\underset{}{=}}&t_{j}\underline{(s_{j-1}\cdots s_{i+2}s_{i+1})x_{i,i+1}(s_{i+1}^{-1}s_{i+2}^{-1}\cdots s_{j-1}^{-1})}\\
&\overset{\text{(6)(b)}}{\underset{}{=}}&t_{j}x_{i,j}.
\end{eqnarray*}
Thus, the relation~(C-xt) for $j-i\geq 2$ is obtained from the relations (1), (2), (6) (a), and (b) in Proposition~\ref{prop_pres_LH}. 
By a similar argument above,  the relation (C-yt) for $k=i$ and $1\leq i<j\leq n+1$ in Lemma~\ref{lem_pres_LH} is also obtained from the relations (1), (2), (6) (a), and (b) in Proposition~\ref{prop_pres_LH}.  
Therefore, we have completed the proof of Lemma~\ref{lem_Ct}. 
\end{proof}

\begin{lem}\label{lem_Mxy}
The relations~(M-x) and (M-y) in Lemma~\ref{lem_pres_LH} are obtained from the relations (1), (2), (6) (a), and (b) in Proposition~\ref{prop_pres_LH}. 
\end{lem}

\begin{proof}
First, we remark that the relation~(2) (g) in Proposition~\ref{prop_pres_LH} is equivalent to the following relation up to the relation~(2) (e) in Proposition~\ref{prop_pres_LH}. 
\begin{eqnarray*}
t_is_i^2\underline{r_i}=\underline{r_i}s_i^2t_{i+1}\overset{\text{CONJ}}{\underset{}{\Longleftrightarrow }}r_i^{-1}\underline{t_is_i^2}=\underline{s_i^2t_{i+1}}r_i^{-1}
\overset{\text{(2)(e)}}{\underset{}{\Longleftrightarrow }}r_i^{-1}s_i^2t_i=t_{i+1}s_i^2r_i^{-1}.
\end{eqnarray*}
In the case of $j=i+1$, we have 
\begin{enumerate}
\item[(x)] $\underline{x_{i,i+1}}\cdot \underline{p_{i,i+1}}t_{i}\overset{\text{(6)(a)}}{\underset{}{=}}s_i\underline{r_i^{-1}\cdot s_i^2t_{i}}\overset{}{\underset{}{=}}s_i\underline{t_{i+1}s_i}s_ir_i^{-1}\overset{\text{(2)(e)}}{\underset{}{=}}s_i^2t_{i}\cdot s_ir_i^{-1}$\\
\hspace{2.3cm}$\overset{\text{(6)(a)}}{\underset{}{=}}p_{i,i+1}t_{i}\cdot x_{i,i+1}$,
\item[(y)] $\underline{y_{i,i+1}}\cdot \underline{p_{i,i+1}}t_{i+1}\overset{\text{(6)(a)}}{\underset{}{=}}r_i^{-1}s_i\cdot s_i\underline{s_it_{i+1}}\overset{\text{(2)(e)}}{\underset{}{=}}\underline{r_i^{-1}s_i^2t_{i}}s_i\overset{}{\underset{}{=}}s_i^2t_{i+1}r_i^{-1}s_i$\\
\hspace{2.3cm}$\overset{\text{(6)(a)}}{\underset{}{=}}p_{i,i+1}t_{i+1}\cdot y_{i,i+1}$.
\end{enumerate}
Thus, we have 
\begin{eqnarray*}
\underline{y_{i,j}}\cdot \underline{p_{i,j}}t_{j}&\overset{\text{(6)(b)}}{\underset{}{=}}&(s_{j-1}\cdots s_{i+2}s_{i+1})y_{i,i+1}p_{i,i+1}\underline{(s_{i+1}^{-1}s_{i+2}^{-1}\cdots s_{j-1}^{-1})t_{j}}\\
&\overset{\text{Lem.\ref{lem_Cst}}}{\underset{}{=}}&(s_{j-1}\cdots s_{i+2}s_{i+1})\underline{y_{i,i+1}p_{i,i+1}t_{i+1}}(s_{i+1}^{-1}s_{i+2}^{-1}\cdots s_{j-1}^{-1})\\
&\overset{}{\underset{}{=}}&(s_{j-1}\cdots s_{i+2}s_{i+1})p_{i,i+1}t_{i+1}\cdot y_{i,i+1}(s_{i+1}^{-1}s_{i+2}^{-1}\cdots s_{j-1}^{-1})\\
&\overset{}{\underset{}{=}}&(s_{j-1}\cdots s_{i+2}s_{i+1})p_{i,i+1}\underline{t_{i+1}(s_{i+1}^{-1}s_{i+2}^{-1}\cdots s_{j-1}^{-1})}\\
&&\cdot (s_{j-1}\cdots s_{i+2}s_{i+1})y_{i,i+1}(s_{i+1}^{-1}s_{i+2}^{-1}\cdots s_{j-1}^{-1})\\
&\overset{}{\underset{}{=}}&\underline{(s_{j-1}\cdots s_{i+2}s_{i+1})p_{i,i+1}(s_{i+1}^{-1}s_{i+2}^{-1}\cdots s_{j-1}^{-1})}t_{j}\\
&&\cdot \underline{(s_{j-1}\cdots s_{i+2}s_{i+1})y_{i,i+1}(s_{i+1}^{-1}s_{i+2}^{-1}\cdots s_{j-1}^{-1})}\\
&\overset{\text{(6)(b)}}{\underset{}{=}}&p_{i,j}t_{j}\cdot y_{i,j}.
\end{eqnarray*}
By a similar argument and the commutativity of $t_i$ and $s_{j-1}\cdots s_{i+2}s_{i+1}$, we can also show that the relation~(M-x) is obtained from the relations (1), (2), (6) (a), and (b) in Proposition~\ref{prop_pres_LH}. 
Therefore, we have completed the proof of Lemma~\ref{lem_Mxy}. 
\end{proof}

Recall that the relations~(6) (d) and (e) in Proposition~\ref{prop_pres_LH} are obtained from the relations (1), (2), and (6) (a)--(c) in Proposition~\ref{prop_pres_LH} by Lemmas~\ref{lem_rel_6d} and \ref{lem_rel_6e}. 
By this fact, we have the following lemma. 

\begin{lem}\label{lem_C123_cyclic}
The relations~(C1), (C2), and (C3) in Lemma~\ref{lem_pres_LH} are equivalent to the relations~(C1), (C2), and (C3) for $i<j<k<l$ in Lemma~\ref{lem_pres_LH} up to the relations (1), (2), and (6) (a)--(c) in Proposition~\ref{prop_pres_LH}, respectively. 
\end{lem}

By Lemma~\ref{lem_Cst} and \ref{lem_C123_cyclic}, the next lemma immediately follows. 

\begin{lem}\label{lem_C1}
The relation~(C1) in Lemma~\ref{lem_pres_LH} is obtained from the relations (1), (2), and (6) (a)--(c) in Proposition~\ref{prop_pres_LH}.  
\end{lem}

To prove that the relations~(C2) and (C3) in Lemma~\ref{lem_pres_LH} are obtained from the relations (1), (2), and (6) (a)--(c) in Proposition~\ref{prop_pres_LH}, we prepaire the following lemma. 

\begin{lem}\label{lem_alpha_ij_conj}
The relation
\[
\alpha _{i,j}=s_i^{-1}s_{i+1}^{-1}\cdots s_{j-2}^{-1}\alpha _{j-1,j}s_{j-2}\cdots s_{i+1}s_i
\]
for $1\leq i<j-1\leq n$ is obtained from the relations (1), (2), and (6) (a)--(c) in Proposition~\ref{prop_pres_LH}.  
\end{lem}

\begin{proof}
We have 
\begin{eqnarray*}
&&(s_i^{-1}\cdots s_{j-3}^{-1}\underline{s_{j-2}^{-1})\alpha _{j-1,j}(s_{j-2}}s_{j-3}\cdots s_i)\\
&\overset{\text{Lem.\ref{lem_rel_6e_pre}}}{\underset{}{=}}&(s_i^{-1}\cdots s_{j-4}^{-1}s_{j-3}^{-1})\underset{\leftarrow }{\underline{s_{j-1}}}\alpha _{j-2,j-1}\underset{\rightarrow }{\underline{s_{j-1}^{-1}}}(s_{j-3}s_{j-4}\cdots s_i)\\
&\overset{\text{(1)(a)}}{\underset{}{=}}&s_{j-1}(s_i^{-1}\cdots s_{j-4}^{-1}\underline{s_{j-3}^{-1})\alpha _{j-2,j-1}(s_{j-3}}s_{j-4}\cdots s_i)s_{j-1}^{-1}\\
&\overset{\text{Lem.\ref{lem_rel_6e_pre}}}{\underset{}{=}}&s_{j-1}(s_i^{-1}\cdots s_{j-5}^{-1}s_{j-4}^{-1})\underset{\leftarrow }{\underline{s_{j-2}}}\alpha _{j-3,j-2}\underset{\rightarrow }{\underline{s_{j-2}^{-1}}}(s_{j-4}s_{j-4}\cdots s_i)s_{j-1}^{-1}\\
&\overset{\text{(1)(a)}}{\underset{}{=}}&s_{j-1}s_{j-2}(s_i^{-1}\cdots s_{j-5}^{-1}s_{j-4}^{-1})\alpha _{j-3,j-2}(s_{j-4}s_{j-4}\cdots s_i)s_{j-2}^{-1}s_{j-1}^{-1}\\
&\vdots &\\
&\overset{}{\underset{}{=}}&s_{j-1}s_{j-2}\cdots s_{i+2}\cdot \underline{s_i^{-1}\alpha _{i+1,i+2}s_i}\cdot s_{i+2}^{-1}\cdots s_{j-2}^{-1}s_{j-1}^{-1}\\
&\overset{\text{Lem.\ref{lem_rel_6e_pre}}}{\underset{}{=}}&s_{j-1}s_{j-2}\cdots s_{i+2}\cdot s_{i+1}\alpha _{i,i+1}s_{i+1}^{-1}\cdot s_{i+2}^{-1}\cdots s_{j-2}^{-1}s_{j-1}^{-1}\\
&\overset{\text{(6)(b)}}{\underset{}{=}}&\alpha _{i,j}.
\end{eqnarray*}
Therefore, we have completed the proof of Lemma~\ref{lem_alpha_ij_conj}. 
\end{proof}

\begin{lem}\label{lem_C2}
The relation~(C2) in Lemma~\ref{lem_pres_LH} is obtained from the relations (1), (2), and (6) (a)--(c) in Proposition~\ref{prop_pres_LH}.  
\end{lem}

\begin{proof}
By Lemma~\ref{lem_C123_cyclic}, it is enough for completing Lemma~\ref{lem_C2} to prove that the relation~(C2) for $1\leq i<j<k\leq n+1$ in Lemma~\ref{lem_pres_LH} is obtained from the relations (1), (2), (6) (a)--(c) in Proposition~\ref{prop_pres_LH}. 
By Table~\ref{rel_C2_PH^1}, we will consider the cases $(\alpha , \beta , \gamma )\in \{ (p, p, p), (p, y, y), (x, p, p), (x, x, p), (x, y, y), (y, p, p), (y, p, x), (y, y, y)\} $. 
First, we prove Lemma~\ref{lem_C2} for $(i,j,k)=(i,i+1,i+2)$. 

In the case that $(\beta , \gamma )=(p, p)$, we have 
\begin{eqnarray*}
\beta _{i,i+2}\gamma _{i+1,i+2}=p_{i,i+2}p_{i+1,i+2}\overset{\text{(6)(a)}}{\underset{\text{(6)(b)}}{=}}s_{i+1}s_i^2s_{i+1}^{-1}\cdot s_{i+1}^2=s_{i+1}s_i^2s_{i+1}.
\end{eqnarray*}
For $\delta \in \{ s, r\}$ and $1\leq i\leq n$, the relation $\delta _i\rightleftarrows s_{i+1}s_i^2s_{i+1}$ is obtained from the relations (1), (2), and (6) (a)--(c) in Proposition~\ref{prop_pres_LH} as follows:  
\begin{eqnarray*}
\underline{\delta _i\cdot s_{i+1}s_i}s_is_{i+1}\overset{\text{(2)(a)}}{\underset{\text{(2)(b)}}{=}}s_{i+1}s_i\underline{\delta _{i+1}s_is_{i+1}}\overset{\text{(2)(a)}}{\underset{\text{(2)(b)}}{=}}s_{i+1}s_i^2s_{i+1}\cdot \delta _i. 
\end{eqnarray*}
Since $p_{i,i+1}=s_i^2$, $x_{i,i+1}=s_ir_i^{-1}$, and $y_{i,i+1}=r_i^{-1}s_i$ for $1\leq i\leq n$ by the relation~(6)~(a) in Proposition~\ref{prop_pres_LH}, the relation $\alpha _{i,i+1}\rightleftarrows p_{i,i+2}p_{i+1,i+2}$ for $1\leq i\leq n$ and $\alpha \in \{ p, x, y\}$, that is the relation~(C2) for $(i,j,k)=(i,i+1,i+2)$ and $(\beta , \gamma )=(p, p)$, is obtained from the relations (1), (2), and (6) (a)--(c) in Proposition~\ref{prop_pres_LH}.  

In the case that $(\beta , \gamma )=(y, y)$, we have 
\begin{eqnarray*}
\beta _{i,i+2}\gamma _{i+1,i+2}&=&y_{i,i+2}y_{i+1,i+2}\overset{\text{(6)(a)}}{\underset{\text{(6)(b)}}{=}}s_{i+1}r_i^{-1}s_is_{i+1}^{-1}\cdot r_{i+1}^{-1}s_{i+1}\\
&\overset{}{\underset{}{=}}&s_{i+1}r_i^{-1}\underline{s_is_{i+1}^{-1}s_{i}^{-1}}s_{i}r_{i+1}^{-1}s_{i+1}\\
&\overset{\text{(2)(a)}}{\underset{}{=}}&s_{i+1}\underline{r_i^{-1}s_{i+1}^{-1}s_{i}^{-1}}\ \underline{s_{i+1}s_{i}r_{i+1}^{-1}}s_{i+1}\overset{\text{(2)(b)}}{\underset{}{=}}\underline{s_{i}^{-1}r_{i+1}^{-1}r_{i}^{-1}}s_{i+1}s_{i}s_{i+1}\\
&\overset{\text{(2)(c)}}{\underset{}{=}}&r_{i+1}^{-1}r_i^{-1}s_is_{i+1}.
\end{eqnarray*}
For $\delta \in \{ s, r\}$ and $1\leq i\leq n$, the relation $\delta _i\rightleftarrows r_{i+1}^{-1}r_i^{-1}s_is_{i+1}$ is obtained from the relations (1), (2), and (6) (a)--(c) in Proposition~\ref{prop_pres_LH} as follows:  
\begin{eqnarray*}
\underline{\delta _i\cdot r_{i+1}^{-1}r_i^{-1}}s_is_{i+1}\overset{\text{(2)(a)}}{\underset{\text{(2)(c)}}{=}}r_{i+1}^{-1}r_i^{-1}\underline{\delta _{i+1}s_is_{i+1}}\overset{\text{(2)(a)}}{\underset{\text{(2)(b)}}{=}}r_{i+1}^{-1}r_i^{-1}s_is_{i+1}\cdot \delta _i. 
\end{eqnarray*}
Thus, the relation $\alpha _{i,i+1}\rightleftarrows y_{i,i+2}y_{i+1,i+2}$ for $1\leq i\leq n$ and $\alpha \in \{ p, x, y\}$, that is the relation~(C2) for $(i,j,k)=(i,i+1,i+2)$ and $(\beta , \gamma )=(y, y)$, is obtained from the relations (1), (2), and (6) (a)--(c) in Proposition~\ref{prop_pres_LH}.  

In the case that $(\alpha , \beta , \gamma )=(x, x, p)$, by a similar argument above, we also have $\beta _{i,i+2}\gamma _{i+1,i+2}=s_{i+1}s_ir_i^{-1}s_{i+1}$ and the relation $\delta _i\rightleftarrows s_{i+1}s_ir_i^{-1}s_{i+1}$ for $\delta \in \{ s, r\}$ and $1\leq i\leq n$ is obtained from the relations (1), (2), (6) (a)--(c) in Proposition~\ref{prop_pres_LH}.  
Hence the relation~(C2) for $(i,j,k)=(i,i+1,i+2)$ and $(\alpha , \beta , \gamma )=(x, x, p)$ is obtained from the relations (1), (2), and (6) (a)--(c) in Proposition~\ref{prop_pres_LH}.  

In the case that $(\alpha , \beta , \gamma )=(y, p, x)$, we have
\begin{eqnarray*}
\alpha _{i,i+1}\beta _{i,i+2}\gamma _{i+1,i+2}&=&y_{i,i+1}p_{i,i+2}x_{i+1,i+2}\overset{\text{(6)(a)}}{\underset{\text{(6)(b)}}{=}}r_i^{-1}\underline{s_i\cdot s_{i+1}s_i}s_is_{i+1}^{-1}\cdot s_{i+1}r_{i+1}^{-1}\\
&\overset{\text{(2)(a)}}{\underset{}{=}}&\underline{r_i^{-1}s_{i+1}s_i}\ \underline{s_{i+1}s_ir_{i+1}^{-1}}\overset{\text{(2)(b)}}{\underset{}{=}}s_{i+1}s_i\underline{r_{i+1}^{-1}r_{i}^{-1}s_{i+1}}s_i\\
&\overset{\text{(2)(c)}}{\underset{}{=}}&s_{i+1}s_is_{i}r_{i+1}^{-1}r_{i}^{-1}s_i=s_{i+1}s_i^2s_{i+1}^{-1}\cdot s_{i+1}r_{i+1}^{-1}\cdot r_{i}^{-1}s_i\\
&\overset{\text{(6)(a)}}{\underset{\text{(6)(b)}}{=}}&p_{i,i+2}x_{i+1,i+2}y_{i,i+1}\\
&=&\beta _{i,i+2}\gamma _{i+1,i+2}\alpha _{i,i+1}.
\end{eqnarray*}
Thus, the relation~(C2) for $(i,j,k)=(i,i+1,i+2)$ and $(\alpha , \beta , \gamma )=(y, p, x)$ is obtained from the relations (1), (2), and (6) (a)--(c) in Proposition~\ref{prop_pres_LH} and we have proved Lemma~\ref{lem_C2} for $(i,j,k)=(i,i+1,i+2)$. 

Assume that $(i,j,k)\not =(i,i+1,i+2)$ and $1\leq i<j<k\leq n+1$. 
For $\gamma \in \{ p, x, y\}$, we remark that 
\begin{eqnarray*}
&&s_{i+1}^{-1}s_{i+2}^{-1}\cdots s_{j-1}^{-1}\gamma _{j,k}s_{j-1}\cdots s_{i+2}s_{i+1}\overset{\text{Lem.\ref{lem_alpha_ij_conj}}}{\underset{}{=}}\gamma _{i+1,k}\\
&\overset{\text{CONJ}}{\underset{\text{(6)(b)}}{\Longleftrightarrow }}&\gamma _{j,k}=s_{j-1}\cdots s_{i+2}s_{i+1}\cdot s_{k-1}\cdots s_{i+3}s_{i+2}\gamma _{i+1,i+2}s_{i+2}^{-1}s_{i+3}^{-1}\cdots s_{k-1}^{-1}\cdot s_{i+1}^{-1}s_{i+2}^{-1}\cdots s_{j-1}^{-1}.
\end{eqnarray*}
Then we have 
\begin{eqnarray*}
&&\alpha _{i,j}\beta _{i,k}\gamma _{j,k}\\
&\overset{\text{(6)(b)}}{\underset{\text{Lem.\ref{lem_alpha_ij_conj}}}{=}}&s_{j-1}\cdots s_{i+2}s_{i+1}\alpha _{i,i+1}\underline{s_{i+1}^{-1}s_{i+2}^{-1}\cdots s_{j-1}^{-1}\cdot s_{k-1}\cdots s_{i+2}s_{i+1}}\beta _{i,i+1}\underline{s_{i+1}^{-1}s_{i+2}^{-1}\cdots s_{k-1}^{-1}}\\
&&\underline{\cdot s_{j-1}\cdots s_{i+2}s_{i+1}}\cdot s_{k-1}\cdots s_{i+3}s_{i+2}\gamma _{i+1,i+2}s_{i+2}^{-1}s_{i+3}^{-1}\cdots s_{k-1}^{-1}\cdot s_{i+1}^{-1}s_{i+2}^{-1}\cdots s_{j-1}^{-1}\\
&\overset{\text{(1)(a)}}{\underset{\text{(2)(a)}}{=}}&s_{j-1}\cdots s_{i+2}s_{i+1}\alpha _{i,i+1}\underset{\leftarrow }{\underline{s_{k-1}\cdots s_{i+2}}}s_{i+1}\cdot s_{i+2}^{-1}s_{i+3}^{-1}\cdots s_{j}^{-1}\beta _{i,i+1}\underset{\leftarrow }{\underline{s_{j}\cdots s_{i+3}s_{i+2}}}\\
&&\cdot s_{i+1}^{-1}\gamma _{i+1,i+2}s_{i+2}^{-1}s_{i+3}^{-1}\cdots s_{k-1}^{-1}\cdot s_{i+1}^{-1}s_{i+2}^{-1}\cdots s_{j-1}^{-1}\\
&\overset{\text{(1)}}{\underset{}{=}}&s_{j-1}\cdots s_{i+2}s_{i+1}\cdot s_{k-1}\cdots s_{i+3}s_{i+2}\alpha _{i,i+1}\underline{s_{i+1}\beta _{i,i+1}s_{i+1}^{-1}}\gamma _{i+1,i+2}\\
&&\cdot s_{i+2}^{-1}s_{i+3}^{-1}\cdots s_{k-1}^{-1}\cdot s_{i+1}^{-1}s_{i+2}^{-1}\cdots s_{j-1}^{-1}\\
&\overset{\text{(6)(b)}}{\underset{}{=}}&s_{j-1}\cdots s_{i+2}s_{i+1}\cdot s_{k-1}\cdots s_{i+3}s_{i+2}\underline{\alpha _{i,i+1}\beta _{i,i+2}\gamma _{i+1,i+2}}\\
&&\cdot s_{i+2}^{-1}s_{i+3}^{-1}\cdots s_{k-1}^{-1}\cdot s_{i+1}^{-1}s_{i+2}^{-1}\cdots s_{j-1}^{-1}\\
&\overset{}{\underset{}{=}}&s_{j-1}\cdots s_{i+2}s_{i+1}\cdot s_{k-1}\cdots s_{i+3}s_{i+2}\underline{\beta _{i,i+2}}\gamma _{i+1,i+2}\alpha _{i,i+1}\\
&&\cdot \underset{\leftarrow }{\underline{s_{i+2}^{-1}s_{i+3}^{-1}\cdots s_{k-1}^{-1}}}\cdot s_{i+1}^{-1}s_{i+2}^{-1}\cdots s_{j-1}^{-1}\\
&\overset{\text{(6)(b)}}{\underset{\text{(1)}}{=}}&\underline{s_{j-1}\cdots s_{i+2}s_{i+1}\cdot s_{k-1}\cdots s_{i+2}s_{i+1}}\beta _{i,i+1}s_{i+1}^{-1}\gamma _{i+1,i+2}s_{i+2}^{-1}s_{i+3}^{-1}\cdots s_{k-1}^{-1}\\
&&\cdot \alpha _{i,i+1}s_{i+1}^{-1}s_{i+2}^{-1}\cdots s_{j-1}^{-1}\\
&\overset{\text{(1)(a)}}{\underset{\text{(2)(a)}}{=}}&s_{k-1}\cdots s_{i+2}s_{i+1}\cdot \underset{\rightarrow }{\underline{s_{j}\cdots s_{i+3}s_{i+2}}}\beta _{i,i+1}s_{i+1}^{-1}\gamma _{i+1,i+2}s_{i+2}^{-1}s_{i+3}^{-1}\cdots s_{k-1}^{-1}\\
&&\cdot \alpha _{i,i+1}s_{i+1}^{-1}s_{i+2}^{-1}\cdots s_{j-1}^{-1}\\
&\overset{\text{(1)}}{\underset{}{=}}&s_{k-1}\cdots s_{i+2}s_{i+1}\beta _{i,i+1}s_{j}\cdots s_{i+3}s_{i+2}s_{i+1}^{-1}\gamma _{i+1,i+2}s_{i+2}^{-1}s_{i+3}^{-1}\cdots s_{k-1}^{-1}\\
&&\cdot \alpha _{i,i+1}s_{i+1}^{-1}s_{i+2}^{-1}\cdots s_{j-1}^{-1}\\
&\overset{}{\underset{}{=}}&s_{k-1}\cdots s_{i+2}s_{i+1}\beta _{i,i+1}\underline{s_{j}\cdots s_{i+3}s_{i+2}s_{i+1}^{-1}s_{i+2}^{-1}\cdots s_{k-1}^{-1}}\\
&&\cdot s_{k-1}\cdots s_{i+3}s_{i+2}\gamma _{i+1,i+2}s_{i+2}^{-1}s_{i+3}^{-1}\cdots s_{k-1}^{-1}\cdot s_{i+1}^{-1}s_{i+2}^{-1}\cdots s_{j-1}^{-1}\\
&&\cdot s_{j-1}\cdots s_{i+2}s_{i+1}\alpha _{i,i+1}s_{i+1}^{-1}s_{i+2}^{-1}\cdots s_{j-1}^{-1}\\
&\overset{\text{(1)(a)}}{\underset{\text{(2)(a)}}{=}}&s_{k-1}\cdots s_{i+2}s_{i+1}\beta _{i,i+1}s_{i+1}^{-1}s_{i+2}^{-1}\cdots s_{k-1}^{-1}\\
&&\cdot s_{j-1}\cdots s_{i+2}s_{i+1}\cdot s_{k-1}\cdots s_{i+3}s_{i+2}\gamma _{i+1,i+2}s_{i+2}^{-1}s_{i+3}^{-1}\cdots s_{k-1}^{-1}\cdot s_{i+1}^{-1}s_{i+2}^{-1}\cdots s_{j-1}^{-1}\\
&&\cdot s_{j-1}\cdots s_{i+2}s_{i+1}\alpha _{i,i+1}s_{i+1}^{-1}s_{i+2}^{-1}\cdots s_{j-1}^{-1}\\
&\overset{\text{(6)(b)}}{\underset{\text{Lem.\ref{lem_alpha_ij_conj}}}{=}}&\beta _{i,k}\gamma _{j,k}\alpha _{i,j}.
\end{eqnarray*}
Therefore, the relation~(C2) for $(i,j,k)\not =(i,i+1,i+2)$ is obtained from the relations (1), (2), and (6) (a)--(c) in Proposition~\ref{prop_pres_LH} and we have completed the proof of Lemma~\ref{lem_C2}. 
\end{proof}

\begin{lem}\label{lem_Cxs}
For $\alpha \in \{ p, x, y\}$, $\beta \in \{ s, r\}$, and $1\leq i<k<j-1\leq n+1$, the relation $\alpha _{i,j}\rightleftarrows \beta _k$ is obtained from the relations~(1), (2), and (6) (a)--(c) in Proposition~\ref{prop_pres_LH}. 
\end{lem}

\begin{proof}
For $1\leq i<k<j-1\leq n+1$, we have 
\begin{eqnarray*}
\alpha _{i,j}\beta _k&\overset{\text{(6)(b)}}{\underset{}{=}}&s_{j-1}\cdots s_{i+2}s_{i+1}\alpha _{i,i+1}s_{i+1}^{-1}s_{i+2}^{-1}\cdots s_{j-1}^{-1}\underset{\leftarrow }{\underline{\beta _k}}\\
&\overset{\text{(1)(a)}}{\underset{}{=}}&s_{j-1}\cdots s_{i+2}s_{i+1}\alpha _{i,i+1}s_{i+1}^{-1}s_{i+2}^{-1}\cdots s_{k-1}^{-1}\underline{s_{k}^{-1}s_{k+1}^{-1}\cdot \beta _k}\cdot s_{k+2}^{-1}\cdots s_{j-1}^{-1}\\
&\overset{\text{(2)(a)}}{\underset{\text{(2)(b)}}{=}}&s_{j-1}\cdots s_{i+2}s_{i+1}\alpha _{i,i+1}s_{i+1}^{-1}s_{i+2}^{-1}\cdots s_{k-1}^{-1}\cdot \underset{\leftarrow }{\underline{\beta _{k+1}}}\cdot s_{k}^{-1}s_{k+1}^{-1}s_{k+2}^{-1}\cdots s_{j-1}^{-1}\\
&\overset{\text{(1)(a)}}{\underset{\text{(6)(a)}}{=}}&s_{j-1}\cdots s_{k+2}\underline{s_{k+1}s_{k}\cdot \beta _{k+1}}\cdot s_{k-1}\cdots s_{i+2}s_{i+1}\alpha _{i,i+1}s_{i+1}^{-1}s_{i+2}^{-1}\cdots s_{j-1}^{-1}\\
&\overset{\text{(2)(a)}}{\underset{\text{(2)(b)}}{=}}&s_{j-1}\cdots s_{k+2}\cdot \underset{\leftarrow }{\underline{\beta _{k}}}\cdot s_{k+1}s_{k}s_{k-1}\cdots s_{i+2}s_{i+1}\alpha _{i,i+1}s_{i+1}^{-1}s_{i+2}^{-1}\cdots s_{j-1}^{-1}\\
&\overset{\text{(1)(a)}}{\underset{\text{(6)(b)}}{=}}&\beta _k\alpha _{i,j}.
\end{eqnarray*}
Therefore, we have completed the proof of Lemma~\ref{lem_Cxs}. 
\end{proof}

\begin{lem}\label{lem_C3}
The relation~(C3) in Lemma~\ref{lem_pres_LH} is obtained from the relations (1), (2), and (6) (a)--(c) in Proposition~\ref{prop_pres_LH}.  
\end{lem}

\begin{proof}
By Lemma~\ref{lem_C123_cyclic}, it is enough for completing Lemma~\ref{lem_C3} to prove that the relation~(C3) for $1\leq i<j<k<l\leq n+1$ in Lemma~\ref{lem_pres_LH} is obtained from the relations (1), (2), (6) (a)--(c) in Proposition~\ref{prop_pres_LH}. 
In the case that $(i,j,k,l)=(i,i+1,i+2,i+3)$, for $\alpha , \beta \in \{ p, x, y\}$, we have 
\begin{eqnarray*}
\alpha _{i,i+2}\cdot p_{i+1,i+2}\beta _{i+1,i+3}p_{i+1,i+2}^{-1}&\overset{\text{(6)}}{\underset{\text{Lem.\ref{lem_alpha_ij_conj}}}{=}}&s_{i+1}\alpha _{i,i+1}s_{i+1}^{-1}\cdot s_{i+1}^2\cdot s_{i+1}^{-1}\beta _{i+2,i+3}s_{i+1}\cdot s_{i+1}^{-2}\\
&=&s_{i+1}\underline{\alpha _{i,i+1}\beta _{i+2,i+3}}s_{i+1}^{-1}\\
&\overset{\text{(6)(a)}}{\underset{\text{(1)}}{=}}&s_{i+1}\beta _{i+2,i+3}\alpha _{i,i+1}s_{i+1}^{-1}\\
&\overset{}{\underset{}{=}}&s_{i+1}^2\cdot s_{i+1}^{-1}\beta _{i+2,i+3}s_{i+1}\cdot s_{i+1}^{-2}\cdot s_{i+1}\alpha _{i,i+1}s_{i+1}^{-1}\\
&\overset{\text{(6)}}{\underset{\text{Lem.\ref{lem_alpha_ij_conj}}}{=}}&p_{i+1,i+2}\beta _{i+1,i+3}p_{i+1,i+2}^{-1}\cdot \alpha _{i,i+2}.
\end{eqnarray*}
Thus, the relation~(C3) for $(i,j,k,l)=(i,i+1,i+2,i+3)$ is obtained from the relations (1), (2), and (6) (a)--(c) in Proposition~\ref{prop_pres_LH}. 

In the case that $(i,j,k,l)\not =(i,i+1,i+2,i+3)$, by an argument similar to the proof of Lemma~\ref{lem_C2} and using Lemma~\ref{lem_alpha_ij_conj} and the relation~(6) (b), we have 
\begin{align*}
p_{j,k}&=s_{j-1}\cdots s_{i+2}s_{i+1}\cdot s_{k-1}\cdots s_{i+3}s_{i+2}p_{i+1,i+2}s_{i+2}^{-1}s_{i+3}^{-1}\cdots s_{k-1}^{-1}\cdot s_{i+1}^{-1}s_{i+2}^{-1}\cdots s_{j-1}^{-1}, \\
\beta _{j,l}&=s_{j-1}\cdots s_{i+2}s_{i+1}\cdot s_{l-1}\cdots s_{i+3}s_{i+2}\beta _{i+1,i+2}s_{i+2}^{-1}s_{i+3}^{-1}\cdots s_{l-1}^{-1}\cdot s_{i+1}^{-1}s_{i+2}^{-1}\cdots s_{j-1}^{-1}.
\end{align*}
Then, we have
\begin{eqnarray*}
&&\alpha _{i,k}\cdot \underline{p_{j,k}\beta _{j,l}p_{j,k}^{-1}}\\
&\overset{\text{(6)(b)}}{\underset{\text{Lem.\ref{lem_alpha_ij_conj}}}{=}}&\underline{\alpha _{i,k}}\cdot \underset{\leftarrow }{\underline{s_{j-1}\cdots s_{i+2}s_{i+1}}}\cdot s_{k-1}\cdots s_{i+3}s_{i+2}p_{i+1,i+2}s_{i+2}^{-1}s_{i+3}^{-1}\cdots s_{k-1}^{-1}\\
&&\cdot s_{l-1}\cdots s_{i+3}s_{i+2}\beta _{i+1,i+2}s_{i+2}^{-1}s_{i+3}^{-1}\cdots s_{l-1}^{-1}\\
&&\cdot s_{k-1}\cdots s_{i+3}s_{i+2}p_{i+1,i+2}^{-1}s_{i+2}^{-1}s_{i+3}^{-1}\cdots s_{k-1}^{-1}\cdot s_{i+1}^{-1}s_{i+2}^{-1}\cdots s_{j-1}^{-1}\\
&\overset{\text{Lem.\ref{lem_Cxs}}}{\underset{\text{(6)(b)}}{=}}&s_{j-1}\cdots s_{i+2}s_{i+1}\cdot s_{k-1}\cdots s_{i+3}s_{i+2}\alpha _{i,i+2}p_{i+1,i+2}s_{i+2}^{-1}\underline{s_{i+3}^{-1}\cdots s_{k-1}^{-1}}\\
&&\underline{\cdot s_{l-1}\cdots s_{i+3}}s_{i+2}\beta _{i+1,i+2}s_{i+2}^{-1}\underline{s_{i+3}^{-1}\cdots s_{l-1}^{-1}}\\
&&\underline{\cdot s_{k-1}\cdots s_{i+3}}s_{i+2}p_{i+1,i+2}^{-1}s_{i+2}^{-1}s_{i+3}^{-1}\cdots s_{k-1}^{-1}\cdot s_{i+1}^{-1}s_{i+2}^{-1}\cdots s_{j-1}^{-1}\\
&\overset{\text{(1)(a)}}{\underset{\text{(2)(a)}}{=}}&s_{j-1}\cdots s_{i+2}s_{i+1}\cdot s_{k-1}\cdots s_{i+3}s_{i+2}\alpha _{i,i+2}p_{i+1,i+2}s_{i+2}^{-1}\cdot \underset{\leftarrow }{\underline{s_{l-1}\cdots s_{i+4}}}s_{i+3}\\
&&\cdot s_{i+4}^{-1}\cdots s_{k}^{-1}\cdot s_{i+2}\beta _{i+1,i+2}s_{i+2}^{-1}\cdot \underset{\leftarrow }{\underline{s_{k}\cdots s_{i+4}}}\\
&&\cdot s_{i+3}^{-1}\underset{\rightarrow }{\underline{s_{i+4}^{-1}\cdots s_{l-1}^{-1}}}\cdot s_{i+2}p_{i+1,i+2}^{-1}s_{i+2}^{-1}s_{i+3}^{-1}\cdots s_{k-1}^{-1}\cdot s_{i+1}^{-1}s_{i+2}^{-1}\cdots s_{j-1}^{-1}\\
&\overset{\text{(1)(a)}}{\underset{}{=}}&s_{j-1}\cdots s_{i+2}s_{i+1}\cdot s_{k-1}\cdots s_{i+3}s_{i+2}\cdot s_{l-1}\cdots s_{i+4}\\
&&\cdot \alpha _{i,i+2}p_{i+1,i+2}\underline{s_{i+2}^{-1}s_{i+3}s_{i+2}}\beta _{i+1,i+2}\underline{s_{i+2}^{-1}s_{i+3}^{-1}s_{i+2}}p_{i+1,i+2}^{-1}\\
&&\cdot s_{i+4}^{-1}\cdots s_{l-1}^{-1}\cdot s_{i+2}^{-1}s_{i+3}^{-1}\cdots s_{k-1}^{-1}\cdot s_{i+1}^{-1}s_{i+2}^{-1}\cdots s_{j-1}^{-1}\\
&\overset{\text{(2)(a)}}{\underset{}{=}}&s_{j-1}\cdots s_{i+2}s_{i+1}\cdot s_{k-1}\cdots s_{i+3}s_{i+2}\cdot s_{l-1}\cdots s_{i+4}\\
&&\cdot \alpha _{i,i+2}p_{i+1,i+2}\underset{\leftarrow }{\underline{s_{i+3}}}s_{i+2}s_{i+3}^{-1}\beta _{i+1,i+2}\underset{\leftarrow }{\underline{s_{i+3}}}s_{i+2}^{-1}\underset{\rightarrow }{\underline{s_{i+3}^{-1}}}p_{i+1,i+2}^{-1}\\
&&\cdot s_{i+4}^{-1}\cdots s_{l-1}^{-1}\cdot s_{i+2}^{-1}s_{i+3}^{-1}\cdots s_{k-1}^{-1}\cdot s_{i+1}^{-1}s_{i+2}^{-1}\cdots s_{j-1}^{-1}\\
&\overset{\text{(1)(a)}}{\underset{}{=}}&s_{j-1}\cdots s_{i+2}s_{i+1}\cdot s_{k-1}\cdots s_{i+3}s_{i+2}\cdot s_{l-1}\cdots s_{i+4}s_{i+3}\\
&&\cdot \alpha _{i,i+2}p_{i+1,i+2}\underline{s_{i+2}\beta _{i+1,i+2}s_{i+2}^{-1}}p_{i+1,i+2}^{-1}\\
&&\cdot s_{i+3}^{-1}s_{i+4}^{-1}\cdots s_{l-1}^{-1}\cdot s_{i+2}^{-1}s_{i+3}^{-1}\cdots s_{k-1}^{-1}\cdot s_{i+1}^{-1}s_{i+2}^{-1}\cdots s_{j-1}^{-1}\\
&\overset{\text{(6)(b)}}{\underset{}{=}}&s_{j-1}\cdots s_{i+2}s_{i+1}\cdot s_{k-1}\cdots s_{i+3}s_{i+2}\cdot s_{l-1}\cdots s_{i+4}s_{i+3}\\
&&\cdot \underline{\alpha _{i,i+2}p_{i+1,i+2}\beta _{i+1,i+3}p_{i+1,i+2}^{-1}}\\
&&\cdot s_{i+3}^{-1}s_{i+4}^{-1}\cdots s_{l-1}^{-1}\cdot s_{i+2}^{-1}s_{i+3}^{-1}\cdots s_{k-1}^{-1}\cdot s_{i+1}^{-1}s_{i+2}^{-1}\cdots s_{j-1}^{-1}\\
&\overset{}{\underset{}{=}}&s_{j-1}\cdots s_{i+2}s_{i+1}\cdot s_{k-1}\cdots s_{i+3}s_{i+2}\cdot s_{l-1}\cdots s_{i+4}\underset{\rightarrow }{\underline{s_{i+3}}}\\
&&\cdot p_{i+1,i+2}\underline{\beta _{i+1,i+3}}p_{i+1,i+2}^{-1}\alpha _{i,i+2}\\
&&\cdot \underset{\leftarrow }{\underline{s_{i+3}^{-1}}}s_{i+4}^{-1}\cdots s_{l-1}^{-1}\cdot s_{i+2}^{-1}s_{i+3}^{-1}\cdots s_{k-1}^{-1}\cdot s_{i+1}^{-1}s_{i+2}^{-1}\cdots s_{j-1}^{-1}\\
&\overset{\text{(6)(b)}}{\underset{\text{(1)(a)}}{=}}&s_{j-1}\cdots s_{i+2}s_{i+1}\cdot s_{k-1}\cdots s_{i+3}s_{i+2}\cdot \underset{\rightarrow }{\underline{s_{l-1}\cdots s_{i+4}}}\\
&&\cdot p_{i+1,i+2}\underline{s_{i+3}s_{i+2}s_{i+3}^{-1}}\beta _{i+1,i+2}\underline{s_{i+3}s_{i+2}^{-1}s_{i+3}^{-1}}p_{i+1,i+2}^{-1}\alpha _{i,i+2}\\
&&\cdot \underset{\leftarrow }{\underline{s_{i+4}^{-1}\cdots s_{l-1}^{-1}}}\cdot s_{i+2}^{-1}s_{i+3}^{-1}\cdots s_{k-1}^{-1}\cdot s_{i+1}^{-1}s_{i+2}^{-1}\cdots s_{j-1}^{-1}\\
&\overset{\text{(2)(a)}}{\underset{\text{(1)(a)}}{=}}&s_{j-1}\cdots s_{i+2}s_{i+1}\cdot s_{k-1}\cdots s_{i+3}s_{i+2}\cdot p_{i+1,i+2}s_{i+2}^{-1}\\
&&\cdot s_{l-1}\cdots s_{i+4}s_{i+3}s_{i+2}\beta _{i+1,i+2}s_{i+2}^{-1}s_{i+3}^{-1}s_{i+4}^{-1}\cdots s_{l-1}^{-1}\\
&&\cdot s_{i+2}p_{i+1,i+2}^{-1}\alpha _{i,i+2}\cdot s_{i+2}^{-1}s_{i+3}^{-1}\cdots s_{k-1}^{-1}\cdot s_{i+1}^{-1}s_{i+2}^{-1}\cdots s_{j-1}^{-1}\\
&\overset{}{\underset{}{=}}&\underline{s_{j-1}\cdots s_{i+2}s_{i+1}\cdot s_{k-1}\cdots s_{i+3}s_{i+2}\cdot p_{i+1,i+2}s_{i+2}^{-1}(s_{i+3}^{-1}\cdots s_{k-1}^{-1}\cdot s_{i+1}^{-1}s_{i+2}^{-1}\cdots s_{j-1}^{-1}}\\
&&\cdot s_{j-1}\cdots s_{i+2}s_{i+1}\cdot s_{k-1}\cdots s_{i+3})s_{l-1}\cdots s_{i+3}s_{i+2}\beta _{i+1,i+2}s_{i+2}^{-1}s_{i+3}^{-1}\cdots s_{l-1}^{-1}\\
&&\cdot s_{i+2}p_{i+1,i+2}^{-1}(s_{i+2}^{-1}s_{i+3}^{-1}\cdots s_{k-1}^{-1}\\
&&\cdot \underline{s_{k-1}\cdots s_{i+3}s_{i+2})\alpha _{i,i+2}\cdot s_{i+2}^{-1}s_{i+3}^{-1}\cdots s_{k-1}^{-1}}\cdot s_{i+1}^{-1}s_{i+2}^{-1}\cdots s_{j-1}^{-1}\\
&\overset{\text{(6)(b)}}{\underset{\text{Lem.\ref{lem_alpha_ij_conj}}}{=}}&p_{j,k}s_{j-1}\cdots s_{i+2}s_{i+1}\cdot \underline{s_{k-1}\cdots s_{i+3}\cdot s_{l-1}\cdots s_{i+3}s_{i+2}}\beta _{i+1,i+2}\underline{s_{i+2}^{-1}s_{i+3}^{-1}\cdots s_{l-1}^{-1}}\\
&&\underline{\cdot (s_{i+3}^{-1}\cdots s_{k-1}^{-1}}\cdot s_{i+1}^{-1}s_{i+2}^{-1}\cdots s_{j-1}^{-1}s_{j-1}\cdots s_{i+2}s_{i+1}\cdot s_{k-1}\cdots s_{i+3})\\
&&\cdot s_{i+2}p_{i+1,i+2}^{-1}s_{i+2}^{-1}s_{i+3}^{-1}\cdots s_{k-1}^{-1}\alpha _{i,k}\underset{\leftarrow }{\underline{s_{i+1}^{-1}s_{i+2}^{-1}\cdots s_{j-1}^{-1}}}\\
&\overset{\text{Lem.\ref{lem_Cxs}}}{\underset{\text{(1),(2)}}{=}}&p_{j,k}\underline{s_{j-1}\cdots s_{i+2}s_{i+1}\cdot s_{l-1}\cdots s_{i+3}s_{i+2}\beta _{i+1,i+2}s_{i+2}^{-1}s_{i+3}^{-1}\cdots s_{l-1}^{-1}\cdot s_{i+1}^{-1}s_{i+2}^{-1}\cdots s_{j-1}^{-1}}\\
&&\cdot \underline{s_{j-1}\cdots s_{i+2}s_{i+1}\cdot s_{k-1}\cdots s_{i+3}s_{i+2}p_{i+1,i+2}^{-1}s_{i+2}^{-1}s_{i+3}^{-1}\cdots s_{k-1}^{-1}\cdot s_{i+1}^{-1}s_{i+2}^{-1}\cdots s_{j-1}^{-1}}\alpha _{i,k}\\
&\overset{\text{(6)(b)}}{\underset{\text{Lem.\ref{lem_alpha_ij_conj}}}{=}}&p_{j,k}\beta _{j,l}p_{j,k}^{-1}\cdot \alpha _{i,k}. 
\end{eqnarray*}
Therefore, we have completed the proof of Lemma~\ref{lem_C3}. 
\end{proof}

\begin{lem}\label{lem_Z}
The relation~(Z) in Lemma~\ref{lem_pres_LH} is equivalent to the relation~(4) in Proposition~\ref{prop_pres_LH} up to the relations (1), (2), and (6) (a)--(c) in Proposition~\ref{prop_pres_LH}.  
\end{lem}

\begin{proof}
We have 
\begin{eqnarray*}
&&x_{1,n+1}^{-1}\cdots x_{1,3}^{-1}x_{1,2}^{-1}p_{1,2}p_{1,3}\cdots p_{1,n+1}\\
&\overset{\text{(6)(a)}}{\underset{\text{(6)(b)}}{=}}&(s_{n}\cdots s_3s_2r_1s_1^{-1}s_2^{-1}s_3^{-1}\cdots s_{n}^{-1})\cdots (s_3s_2r_1s_1^{-1}s_2^{-1}s_3^{-1})(s_2r_1\underline{s_1^{-1}s_2^{-1})r_1}s_1^{-1}\\
&&\cdot s_1^2(\underline{s_2s_1^{2}s_2^{-1}})(s_3\underline{s_2s_1^{2}s_2^{-1}}s_3^{-1})\cdots (s_{n}\cdots s_3\underline{s_2s_1^{2}s_2^{-1}}s_3^{-1}\cdots s_{n}^{-1})\\
&\overset{\text{(2)(a)}}{\underset{\text{(2)(b)}}{=}}&(s_{n}\cdots s_3s_2r_1s_1^{-1}s_2^{-1}s_3^{-1}\cdots s_{n}^{-1})\cdots (s_3s_2r_1s_1^{-1}s_2^{-1}s_3^{-1})\underline{s_2r_1r_2}s_1^{-1}s_2^{-1}\\
&&\cdot s_1(s_1^{-1}s_2^2s_1)(s_3\underset{\leftarrow }{\underline{s_1^{-1}}}s_2^2s_1s_3^{-1})\cdots (s_{n}\cdots s_3\underset{\leftarrow }{\underline{s_1^{-1}}}s_2^2\underset{\rightarrow }{\underline{s_1}}s_3^{-1}\cdots s_{n}^{-1})\\
&\overset{\text{(2)(c)}}{\underset{\text{(1)(a)}}{=}}&(s_{n}\cdots s_3s_2r_1s_1^{-1}s_2^{-1}s_3^{-1}\cdots s_{n}^{-1})\cdots (s_3s_2r_1\underline{s_1^{-1}s_2^{-1}s_3^{-1})r_1r_2}s_2^{-1}\\
&&\cdot s_2^2(\underline{s_3s_2^2s_3^{-1}})(s_4\underline{s_3s_2^2s_3^{-1}}s_4^{-1})\cdots (s_{n}\cdots s_4\underline{s_3s_2^2s_3^{-1}}s_4^{-1}\cdots s_{n}^{-1})s_1\\
&\overset{\text{(2)(a),(b)}}{\underset{\text{(1)(a)}}{=}}&(s_{n}\cdots s_3s_2r_1s_1^{-1}s_2^{-1}s_3^{-1}\cdots s_{n}^{-1})\cdots (s_4s_3s_2r_1s_1^{-1}s_2^{-1}s_3^{-1}s_4^{-1})\underline{s_3s_2r_1r_2r_3}s_1^{-1}s_2^{-1}s_3^{-1}\\
&&\cdot s_2(s_2^{-1}s_3^2s_2)(s_4\underset{\leftarrow }{\underline{s_2^{-1}}}s_3^2s_2s_4^{-1})\cdots (s_{n}\cdots s_4\underset{\leftarrow }{\underline{s_2^{-1}}}s_3^2\underset{\rightarrow }{\underline{s_2}}s_4^{-1}\cdots s_{n}^{-1})s_1\\
&\overset{\text{(2)(c)}}{\underset{\text{(1)(a)}}{=}}&(s_{n}\cdots s_3s_2r_1s_1^{-1}s_2^{-1}s_3^{-1}\cdots s_{n}^{-1})\cdots (s_4s_3s_2r_1s_1^{-1}s_2^{-1}s_3^{-1}s_4^{-1})r_1r_2r_3s_3^{-1}\\
&&\cdot s_3^2(s_4s_3^2s_4^{-1})\cdots (s_{n}\cdots s_4s_3^2s_4^{-1}\cdots s_{n}^{-1})s_2s_1\\
&\vdots &\\
&=&(s_{n}\cdots s_3s_2r_1\underline{s_1^{-1}s_2^{-1}s_3^{-1}\cdots s_{n}^{-1})r_1r_2\cdots r_{n-1}}s_{n-1}^{-1}\\
&&\cdot s_{n-1}^2(\underline{s_{n}s_{n-1}^2s_{n}^{-1}})s_{n-2}\cdots s_2s_1\\
&\overset{\text{(2)(a),(b)}}{\underset{\text{(1)}}{=}}&\underline{s_{n}\cdots s_3s_2r_1r_2r_3\cdots r_{n}}s_1^{-1}s_2^{-1}s_3^{-1}\cdots s_{n}^{-1}\\
&&\cdot s_{n-1}(s_{n-1}^{-1}s_{n}^2s_{n-1})s_{n-2}\cdots s_2s_1\\
&\overset{\text{(2)(c)}}{\underset{\text{(1)(a)}}{=}}&r_1r_2\cdots r_ns_n\cdots s_2s_1.
\end{eqnarray*}
Therefore, the relation $x_{1,n+1}^{-1}\cdots x_{1,3}^{-1}x_{1,2}^{-1}p_{1,2}p_{1,3}\cdots p_{1,n+1}t_1=1$ is equivalent to the relation $r_1r_2\cdots r_ns_n\cdots s_2s_1t_1=1$ up to the relations (1), (2), and (6) (a)--(c) in Proposition~\ref{prop_pres_LH}, and we have completed the proof of Lemma~\ref{lem_Z}. 
\end{proof}

\begin{lem}\label{lem_F}
The relation~(F) in Lemma~\ref{lem_pres_LH} is equivalent to the relation~(5) in Proposition~\ref{prop_pres_LH} up to the relations (1), (2), and (6) (a)--(c) in Proposition~\ref{prop_pres_LH}.  
\end{lem}

\begin{proof}
For $3\leq i\leq n+1$, we have 
\begin{eqnarray*}
&&p_{1,i}p_{2,i}\cdots p_{i-1,i}\\
&=&(s_{i-1}\cdots s_{3}s_{2}s_{1}^2s_{2}^{-1}s_{3}^{-1}\cdots s_{i-1}^{-1})(s_{i-1}\cdots s_{4}s_{3}s_{2}^2s_{3}^{-1}s_{4}^{-1}\cdots s_{i-1}^{-1})\cdots (s_{i-1}s_{i-2}s_{i-1}^{-1})s_{i-1}^{2}\\
&=&s_{i-1}\cdots s_{3}s_{2}s_{1}^2s_{2}s_{3}\cdots s_{i-1}.
\end{eqnarray*}
Thus, we have 
\begin{eqnarray*}
&&p_{1,2}(p_{1,3}p_{2,3})\cdots (p_{1,n}p_{2,n}\cdots p_{n-1,n})(p_{1,n+1}p_{2,n+1}\cdots p_{n,n+1})\\
&=&s_1\underline{s_1(s_2s_1}\ \underline{s_1s_2)(s_3s_2s_1}\ \underline{s_1s_2s_3)}\cdots \underline{(s_{n-1}\cdots s_{2}s_1}\ \underline{s_1s_2\cdots s_{n-1})(s_{n}\cdots s_{2}s_1}s_1s_2\cdots s_{n})\\
&\overset{\text{(2)(a)}}{\underset{\text{(1)(a)}}{=}}&\left( s_1(s_2s_1)\right) \\
&&\cdot \underline{s_2(s_3s_2s_1)}\ \underline{s_2s_3(s_4s_3s_2s_1)}\cdots \underline{s_2s_3\cdots s_{n-2}(s_{n-1}\cdots s_{2}s_1)}\ \underline{s_2s_3\cdots s_{n-1}(s_{n}\cdots s_{2}s_1)}\\
&&\cdot \left( s_2s_3\cdots s_{n}\cdot s_1s_2\cdots s_{n}\right) \\
&\overset{\text{(2)(a)}}{\underset{\text{(1)(a)}}{=}}&\left( s_1(s_2s_1)(s_3s_2s_1)\right) \\
&&\cdot s_3(s_4s_3s_2s_1)s_3s_4(s_5s_4s_3s_2s_1)\cdots s_3s_4\cdots s_{n-1}(s_{n}\cdots s_{2}s_1)\\
&&\cdot \left( s_3s_4\cdots s_{n}\cdot s_2s_3\cdots s_{n}\cdot s_1s_2\cdots s_{n}\right) \\
&\vdots &\\
&=&\left( s_1(s_2s_1)\cdots (s_{n-1}\cdots s_2s_1)(s_n\cdots s_2s_1)\right) \cdot \left( s_n(s_{n-1}s_{n})(s_{n-2}s_{n-1}s_n)\cdots (s_1s_2\cdots s_{n})\right) .
\end{eqnarray*}
Since the relation
\[
s_i(s_{i-1}s_i\cdots s_n)=(s_{i-1}s_i\cdots s_n)s_{i-1}
\]
for $2\leq i\leq n$ holds by the relations~(1) (a) and (2) (a) in Proposition~\ref{prop_pres_LH}, we have
\begin{eqnarray*}
&&\underline{s_n}(s_{n-1}s_{n})(s_{n-2}s_{n-1}s_n)\cdots (s_1s_2\cdots s_{n})\\
&=&(s_{n-1}\underline{s_{n}})(s_{n-2}s_{n-1}s_n)\cdots (s_1s_2\cdots s_{n})\cdot s_1\\
&=&s_{n-1}(s_{n-2}s_{n-1}\underline{s_{n}})\cdots (s_1s_2\cdots s_{n})\cdot s_2s_1\\
&=&s_{n-1}(s_{n-2}s_{n-1})(s_{n-3}s_{n-2}s_{n-1}s_n)\cdots (s_1s_2\cdots s_{n})\cdot s_3s_2s_1\\
&\vdots &\\
&=&\underline{s_{n-1}}(s_{n-2}s_{n-1})\cdots (s_1s_2\cdots s_{n-1})\cdot s_{n}\cdots s_2s_1\\
&=&(s_{n-2}\underline{s_{n-1}})(s_{n-3}s_{n-2}s_{n-1})\cdots (s_1s_2\cdots s_{n-1})\cdot s_{1}\cdot s_{n}\cdots s_2s_1\\
&=&s_{n-2}(s_{n-3}s_{n-2}s_{n-1})\cdots (s_1s_2\cdots s_{n-1})\cdot s_2s_{1}\cdot s_{n}\cdots s_2s_1\\
&\vdots &\\
&=&s_{n-2}(s_{n-3}s_{n-2})\cdots (s_1s_2\cdots s_{n-2})\cdot s_{n-1}\cdots s_2s_{1}\cdot s_{n}\cdots s_2s_1\\
&\vdots &\\
&=&s_1(s_2s_1)\cdots (s_{n-1}\cdots s_2s_1)(s_n\cdots s_2s_1). 
\end{eqnarray*}
Therefore, the relation $t_{1}t_2\cdots t_{n+1}p_{1,2}(p_{1,3}p_{2,3})\cdots (p_{1,n+1}p_{2,n+1}\cdots p_{n,n+1})=1$ is equivalent to the relation $t_{1}t_2\cdots t_{n+1}\bigl( s_1(s_2s_1)\cdots (s_{n-1}\cdots s_2s_1)(s_n\cdots s_2s_1)\bigr) ^2=1$ up to the relations (1), (2), and (6) (a)--(c) in Proposition~\ref{prop_pres_LH}, and we have completed the proof of Lemma~\ref{lem_F}. 
\end{proof}

\begin{proof}[Proof of Proposition~\ref{prop_pres_LH}]
We consider the finite presentation which is obtained from the finite presentation for $\LH $ in Lemma~\ref{lem_pres_LH} by adding generators $r_i$ for $1\leq i\leq n$ and $s$ and the relations~(1)--(6) in Proposition~\ref{prop_pres_LH}. 
Since $\LH $ admits the relations~(1)--(6) in Proposition~\ref{prop_pres_LH} and the relations~(6) (a) and (c) in Proposition~\ref{prop_pres_LH} imply that $r_i$ for $1\leq i\leq n$ and $s$ are products of generators for the presentation in Lemma~\ref{lem_pres_LH}, the group which is obtained from this finite presentation is also isomorphic to $\LH $. 

By Lemmas~\ref{lem_Ct}, \ref{lem_Mxy}, \ref{lem_C1}, \ref{lem_C2}, \ref{lem_C3}, \ref{lem_Z}, and \ref{lem_F}, the relations~(C-pt), (C-tt), (C-xt), (C-yt), (M-x), (M-y), (C1), (C2), (C3), (Z), and (F) in Lemma~\ref{lem_pres_LH} are obtained from the relations~(1), (2), and (6) (a)--(c) in Proposition~\ref{prop_pres_LH}. 
The relations~(1), (2), (3), (4), (5), (A2) (a), (b), and (c) in Lemma~\ref{lem_pres_LH} coincide with the relations~(6) (a), (1) (a), (2) (a), (3), (1) (d), (1) (e), (6) (f), and (6) (g) in Proposition~\ref{prop_pres_LH}, respectively. 
The relation~(A1) (a) in Lemma~\ref{lem_pres_LH} obtained from the relations~(1) (b) and (2) (e) in Proposition~\ref{prop_pres_LH}. 
The relation~(A1) (c) in Lemma~\ref{lem_pres_LH} obtained from the relations~(6) (a) and (b) in Proposition~\ref{prop_pres_LH} and Lemma~\ref{lem_alpha_ij_conj}. 
Finally, the relation~(A1) (b) in Lemma~\ref{lem_pres_LH} obtained from the relations~(6) (a) in Proposition~\ref{prop_pres_LH} as follows: in the case that $\alpha =p$, the relation~(A1) is clearly obtained from the relation~(6) (a) in Proposition~\ref{prop_pres_LH}. 
In the case that $\alpha \in \{ x, y\}$, we have
\begin{enumerate}
\item[(x)] $s_ix_{i,i+1}s_i^{-1}\overset{\text{(6)(a)}}{\underset{}{=}}s_is_ir_i^{-1}s_i^{-1}=s_i^2(r_i^{-1}s_i)s_i^{-2}\overset{\text{(6)(a)}}{\underset{}{=}}p_{i,i+1}y_{i,i+1}p_{i,i+1}^{-1}$,
\item[(y)] $s_iy_{i,i+1}s_i^{-1}\overset{\text{(6)(a)}}{\underset{}{=}}s_ir_i^{-1}s_is_i^{-1}=s_ir_i^{-1}\overset{\text{(6)(a)}}{\underset{}{=}}x_{i,i+1}$.
\end{enumerate}

By an argument above and Tietze transformations, $\LH $ admits the presentation which is obtained from the presentation in Proposition~\ref{prop_pres_LH} by adding generators $p_{i,j}$, $x_{i,j}$, and $y_{i,j}$ for $1\leq j<i\leq n+1$ and the relation~(0) in Lemma~\ref{lem_pres_LH}. 
Since generators $p_{i,j}$, $x_{i,j}$, and $y_{i,j}$ for $1\leq j<i\leq n+1$ do not appear in the relations~(1)--(6) in Proposition~\ref{prop_pres_LH}, we can remove these generators and the relation~(0) from this presentation. 
Therefore, we have completed the proof of Proposition~\ref{prop_pres_LH}. 
\end{proof}

\section{Presentations for the balanced superelliptic handlebody groups}\label{section_smod}

Throughout this section, we assume that $g=n(k-1)$ for $n\geq 1$ and $k\geq 3$, and $p=p_{g,k}\colon H_g\to B^3$ is the balanced superelliptic covering map with the covering transformation group $\left< \zeta =\zeta _{g,k}\right> $.

\subsection{Explicit lifts of generators for the liftable Hilden groups}\label{section_lifts}

In this section, we give explicit lifts of generators for $\LH $ in Theorem~\ref{thm_pres_LH} with respect to $p$ (namely, that are preimages with respect to $\theta \colon \SH \to \LH$). 
First, we will review the generators for $\SM $ in Corollary~6.13 of \cite{Hirose-Omori}. 

Recall that $\gamma _{i}^l$ for $1\leq i\leq 2n+1$ and $1\leq l\leq k$ is a simple closed curve on $\Sigma _g$ as in Figure~\ref{fig_scc_c_il} and $\gamma _{2i-1}^l$ for $1\leq i\leq n+1$ bounds a disk in $H_g$. 
By Lemma~6.4 in \cite{Hirose-Omori}, we have the following lemma.

\begin{lem}\label{lift-t_{i,i+1}}
For $1\leq i\leq 2n+1$ and some lift $\widetilde{t}_{\gamma _{i,i+1}}$ of $t_{\gamma _{i,i+1}}$ with respect to $p$, we have
\[
\widetilde{t}_{\gamma _{i,i+1}}=t_{\gamma _i^1}t_{\gamma _i^2}\cdots t_{\gamma _i^{k}}.
\] 
\end{lem}

Denote $\widetilde{t}_{\gamma _{i,i+1}}=t_{\gamma _i^1}t_{\gamma _i^2}\cdots t_{\gamma _i^{k}}$ for $1\leq i\leq 2n+1$ and $\widetilde{t}_{i}=\widetilde{t}_{\gamma _{2i-1,2i}}$ for $1\leq i\leq n+1$. 
Remark that $\widetilde{t}_{i}$ is a lift of $t_i$ with respect to $p$. 
Since $\gamma _{2i-1}^l$ for $1\leq i\leq n+1$ and $1\leq l\leq k$ bounds a disk in $H_g$, $\widetilde{t}_{i}$ for $1\leq i\leq n+1$ lies in $\SH $.  


Let $h_i$ for $1\leq i\leq 2n$ be a self-homeomorphism on $S^2$ which is supported on a regular neighborhood of $l_i\cup l_{i+1}$ in $(S^2-\B )\cup \{ p_{i},\ p_{i+1},\ p_{i+2}\}$ and described as the result of anticlockwise half-rotation of $l_i\cup l_{i+1}$ as in Figure~\ref{fig_h_i}. 
Note that $h_i=\sigma _i\sigma _{i+1}\sigma _i$ for $1\leq i\leq 2n$. 
Since $\Psi (h_i)=(i\ i+2)$ for $1\leq i\leq 2n$, the mapping class $h_i$ is parity-preserving. 
By Lemma~6.6 in \cite{Hirose-Omori}, we have the following lemma. 

\begin{figure}[h]
\includegraphics[scale=0.95]{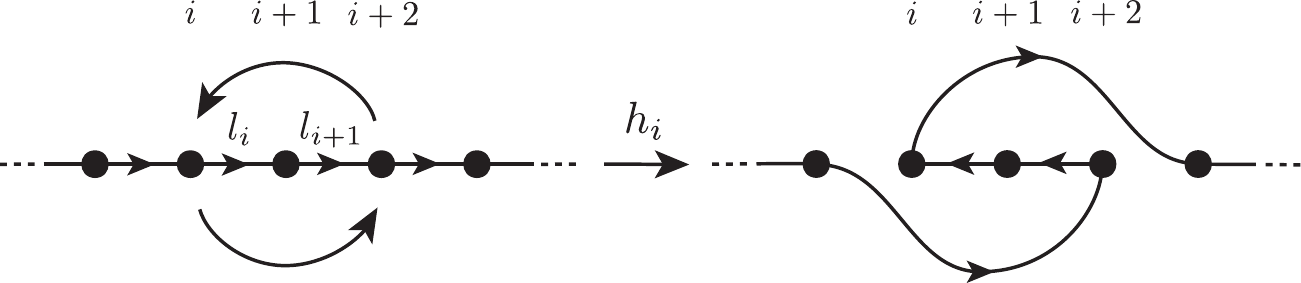}
\caption{The mapping class $h_i$ $(1\leq i\leq 2n)$ on $S^2$.}\label{fig_h_i}
\end{figure}

\begin{lem}\label{lift-h_i}
For some lift $\widetilde{h}_i$ of $h_i$ with respect to $p$, we have 
\begin{enumerate}
\item $\widetilde{h}_i=t_{\gamma _i^1}t_{\gamma _{i+1}^1}t_{\gamma _i^2}t_{\gamma _{i+1}^2}\cdots t_{\gamma _i^{k-1}}t_{\gamma _{i+1}^{k-1}}t_{\gamma _i^{k}}$\quad for odd $1\leq i\leq 2n-1$,
\item $\widetilde{h}_i=t_{\gamma _i^{k}}t_{\gamma _{i+1}^{k}}t_{\gamma _i^{k-1}}t_{\gamma _{i+1}^{k-1}}\cdots t_{\gamma _i^{2}}t_{\gamma _{i+1}^{2}}t_{\gamma _i^{1}}$\quad for even $2\leq i\leq 2n$. 
\end{enumerate}
\end{lem}

Denote $\widetilde{h}_i=t_{\gamma _i^1}t_{\gamma _{i+1}^1}t_{\gamma _i^2}t_{\gamma _{i+1}^2}\cdots t_{\gamma _i^{k-1}}t_{\gamma _{i+1}^{k-1}}t_{\gamma _i^{k}}$\quad for odd $1\leq i\leq 2n-1$, $\widetilde{h}_i=t_{\gamma _i^{k}}t_{\gamma _{i+1}^{k}}t_{\gamma _i^{k-1}}t_{\gamma _{i+1}^{k-1}}\cdots t_{\gamma _i^{2}}t_{\gamma _{i+1}^{2}}t_{\gamma _i^{1}}$\quad for even $2\leq i\leq 2n$. 
Since we can show that $s_i=h_{2i}h_{2i-1}t_i^{-1}$ and $r_i=h_{2i}^{-1}h_{2i}$ for $1\leq i\leq n$ as in Figure~\ref{fig_braid-s_i-r_i-product}, we have the following lemma. 

\begin{figure}[h]
\includegraphics[scale=1.1]{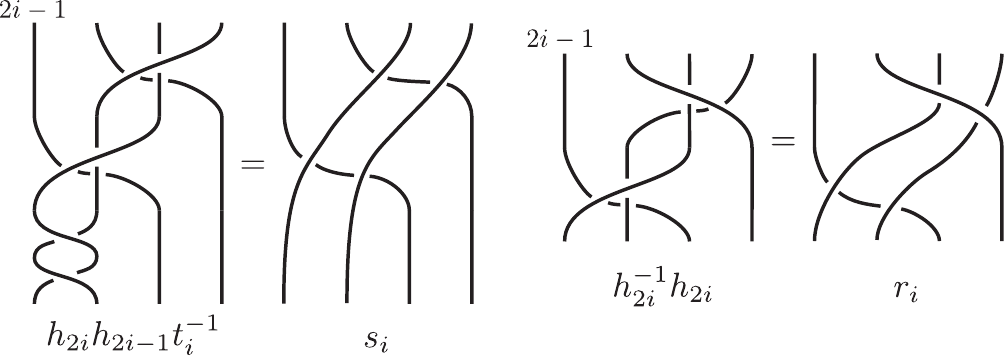}
\caption{Braids $s_i=h_{2i}h_{2i-1}t_i^{-1}$ and $r_i=h_{2i}^{-1}h_{2i}$ for $1\leq i\leq n$.}\label{fig_braid-s_i-r_i-product}
\end{figure}

\begin{lem}\label{lift-s_i-r_i}
For $1\leq i\leq n$ and some lifts $\widetilde{s}_i$ and $\widetilde{r}_i$ of $h_i$ and $r_i$ with respect to $p$, respectively, we have
\begin{enumerate}
\item $\widetilde{s}_i=\widetilde{h}_{2i}\widetilde{h}_{2i-1}\widetilde{t}_i^{-1}$,
\item $\widetilde{r}_i=\widetilde{h}_{2i}^{-1}\widetilde{h}_{2i-1}$.
\end{enumerate}
\end{lem}

Denote $\widetilde{s}_i=\widetilde{h}_{2i}\widetilde{h}_{2i-1}\widetilde{t}_i^{-1}$ and $\widetilde{r}_i=\widetilde{h}_{2i}^{-1}\widetilde{h}_{2i-1}$ for $1\leq i \leq n$. 
Since we can check that $\widetilde{h}_{2i}\widetilde{h}_{2i-1}$ switches $\gamma_{2i-1}^l$ and $\gamma_{2i+1}^l$, the element $\widetilde{s}_i$ also switches $\gamma_{2i-1}^l$ and $\gamma_{2i+1}^l$. 
We denote $\delta _i^{l}=t_{\gamma _{2i}^{l}}^{-1}(\gamma _{2i-1}^{l})$ and $\eta _i^{l}=t_{\gamma _{2i}^{l}}^{-1}(\gamma _{2i+1}^{l+1})$ for $1\leq i \leq n$ and $1\leq l\leq k$, where we cosider the index $l+1$ mod $k$. 
We remark that $\delta _i^{l}$ and $\eta _i^{l}$ are simple closed curves on $\Sigma _{g}$ as in Figure~\ref{fig_scc_delta-eta}. 
Then we have the following lemma. 

\begin{figure}[h]
\includegraphics[scale=0.8]{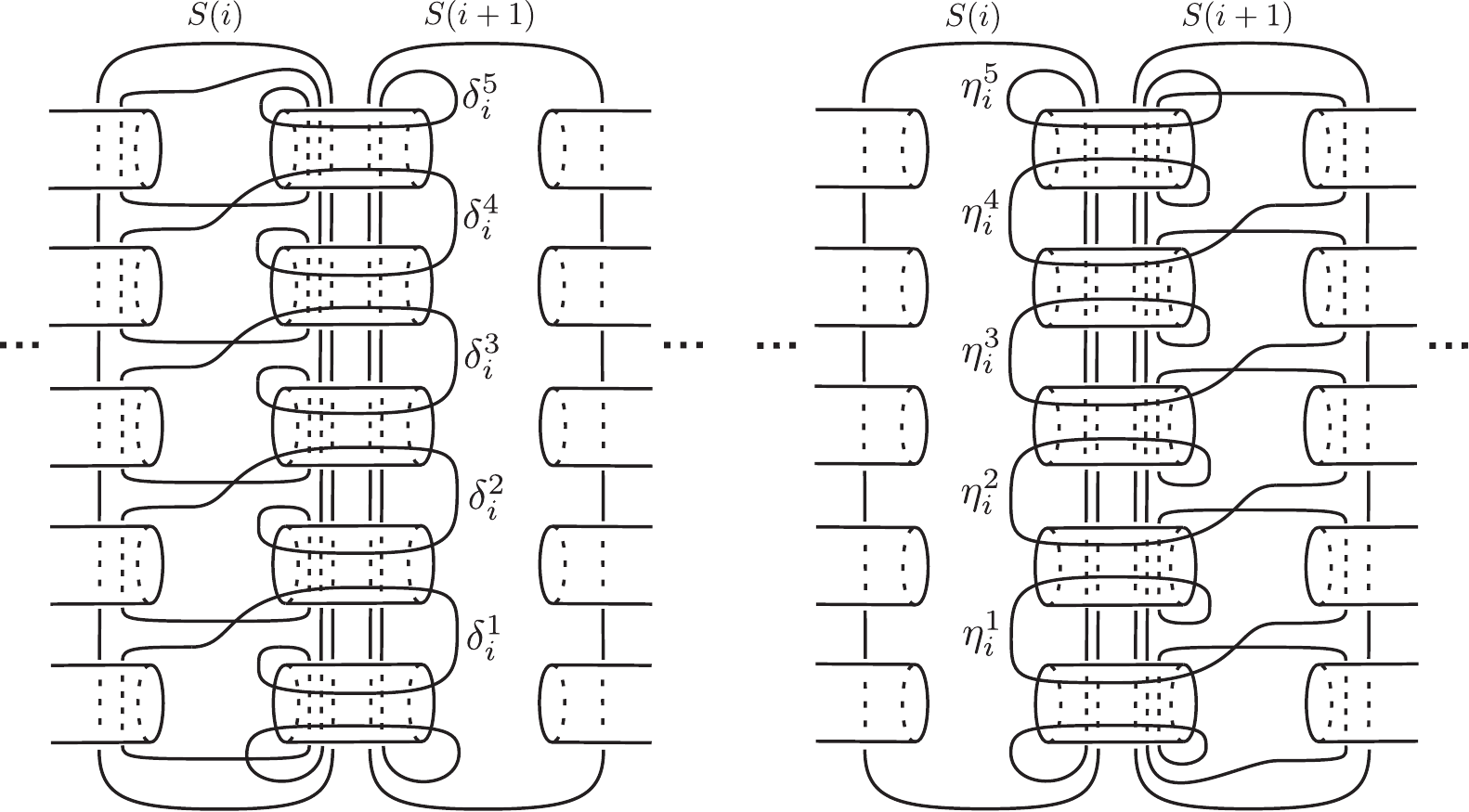}
\caption{Simple closed curves $\delta _i^{l}$ and $\eta _i^{l}$ for $1\leq i\leq n$ and $1\leq l\leq k$ when $k=5$.}\label{fig_scc_delta-eta}
\end{figure}

\begin{lem}\label{lift-r_i}
For $1\leq i\leq n$, we have
\[
\widetilde{r}_i=t_{\delta _{i}^1}t_{\eta _{i}^1}^{-1}t_{\delta _{i}^2}t_{\eta _{i}^2}^{-1}\cdots t_{\delta _{i}^{k-1}}t_{\eta _{i}^{k-1}}^{-1}.
\]
\end{lem}

\begin{proof}
By Lemma~\ref{lift-s_i-r_i} and the definition above, we have $\widetilde{r}_i=\widetilde{h}_{2i}^{-1}\widetilde{h}_{2i-1}$ for $1\leq i\leq n$. 
We can show that $t_{\gamma _{2i-1}^1}t_{\gamma _{2i}^1}t_{\gamma _{2i-1}^2}t_{\gamma _{2i}^2}\cdots t_{\gamma _{2i-1}^{k-1}}t_{\gamma _{2i}^{k-1}}(\gamma _{2i-1}^{k})=\gamma _{2i}^{k}$ for $1\leq i\leq 2n$ as in Figure~\ref{fig_proof_lift_r_i}. 
Hence, we have
\begin{eqnarray*}
\widetilde{r}_i&=&\widetilde{h}_{2i}^{-1}\widetilde{h}_{2i-1}\\
&=&(t_{\gamma _{2i}^{1}}^{-1}t_{\gamma _{2i+1}^{2}}^{-1}t_{\gamma _{2i}^{2}}^{-1}\cdots t_{\gamma _{2i+1}^{k-1}}^{-1}t_{\gamma _{2i}^{k-1}}^{-1}t_{\gamma _{2i+1}^{k}}^{-1}t_{\gamma _{2i}^{k}}^{-1})
\underline{(t_{\gamma _{2i-1}^1}t_{\gamma _{2i}^1}t_{\gamma _{2i-1}^2}t_{\gamma _{2i}^2}\cdots t_{\gamma _{2i-1}^{k-1}}t_{\gamma _{2i}^{k-1}}t_{\gamma _{2i-1}^{k}})}\\
&&\underline{\cdot (t_{\gamma _{2i-1}^1}t_{\gamma _{2i}^1}t_{\gamma _{2i-1}^2}t_{\gamma _{2i}^2}\cdots t_{\gamma _{2i-1}^{k-1}}t_{\gamma _{2i}^{k-1}})^{-1}}t_{\gamma _{2i-1}^1}t_{\gamma _{2i}^1}t_{\gamma _{2i-1}^2}t_{\gamma _{2i}^2}\cdots t_{\gamma _{2i-1}^{k-1}}t_{\gamma _{2i}^{k-1}}\\
&=&(t_{\gamma _{2i}^{1}}^{-1}t_{\gamma _{2i+1}^{2}}^{-1}t_{\gamma _{2i}^{2}}^{-1}\cdots t_{\gamma _{2i+1}^{k-1}}^{-1}t_{\gamma _{2i}^{k-1}}^{-1}t_{\gamma _{2i+1}^{k}}^{-1})\underset{\leftarrow }{\underline{t_{\gamma _{2i-1}^1}t_{\gamma _{2i}^1}t_{\gamma _{2i-1}^2}t_{\gamma _{2i}^2}\cdots t_{\gamma _{2i-1}^{k-2}}t_{\gamma _{2i}^{k-2}}}}t_{\gamma _{2i-1}^{k-1}}t_{\gamma _{2i}^{k-1}}\\
&=&(t_{\gamma _{2i}^{1}}^{-1}t_{\gamma _{2i+1}^{2}}^{-1}t_{\gamma _{2i-1}^1}t_{\gamma _{2i}^1})(t_{\gamma _{2i}^{2}}^{-1}t_{\gamma _{2i+1}^{3}}^{-1}t_{\gamma _{2i-1}^2}t_{\gamma _{2i}^2})\cdots (t_{\gamma _{2i}^{k-1}}^{-1}t_{\gamma _{2i+1}^{k}}^{-1}t_{\gamma _{2i-1}^{k-1}}t_{\gamma _{2i}^{k-1}})\\
&=&(t_{\eta _{i}^1}^{-1}\underset{\leftarrow }{\underline{t_{\delta _{i}^1}}})(t_{\eta _{i}^2}^{-1}\underset{\leftarrow }{\underline{t_{\delta _{i}^2}}})\cdots (t_{\eta _{i}^{k-1}}^{-1}\underset{\leftarrow }{\underline{t_{\delta _{i}^{k-1}}}})\\
&=&t_{\delta _{i}^1}t_{\eta _{i}^1}^{-1}t_{\delta _{i}^2}t_{\eta _{i}^2}^{-1}\cdots t_{\delta _{i}^{k-1}}t_{\eta _{i}^{k-1}}^{-1}
\end{eqnarray*}
Therefore, we have completed the proof of Lemma~\ref{lift-r_i}. 
\end{proof}

\begin{figure}[h]
\includegraphics[scale=0.52]{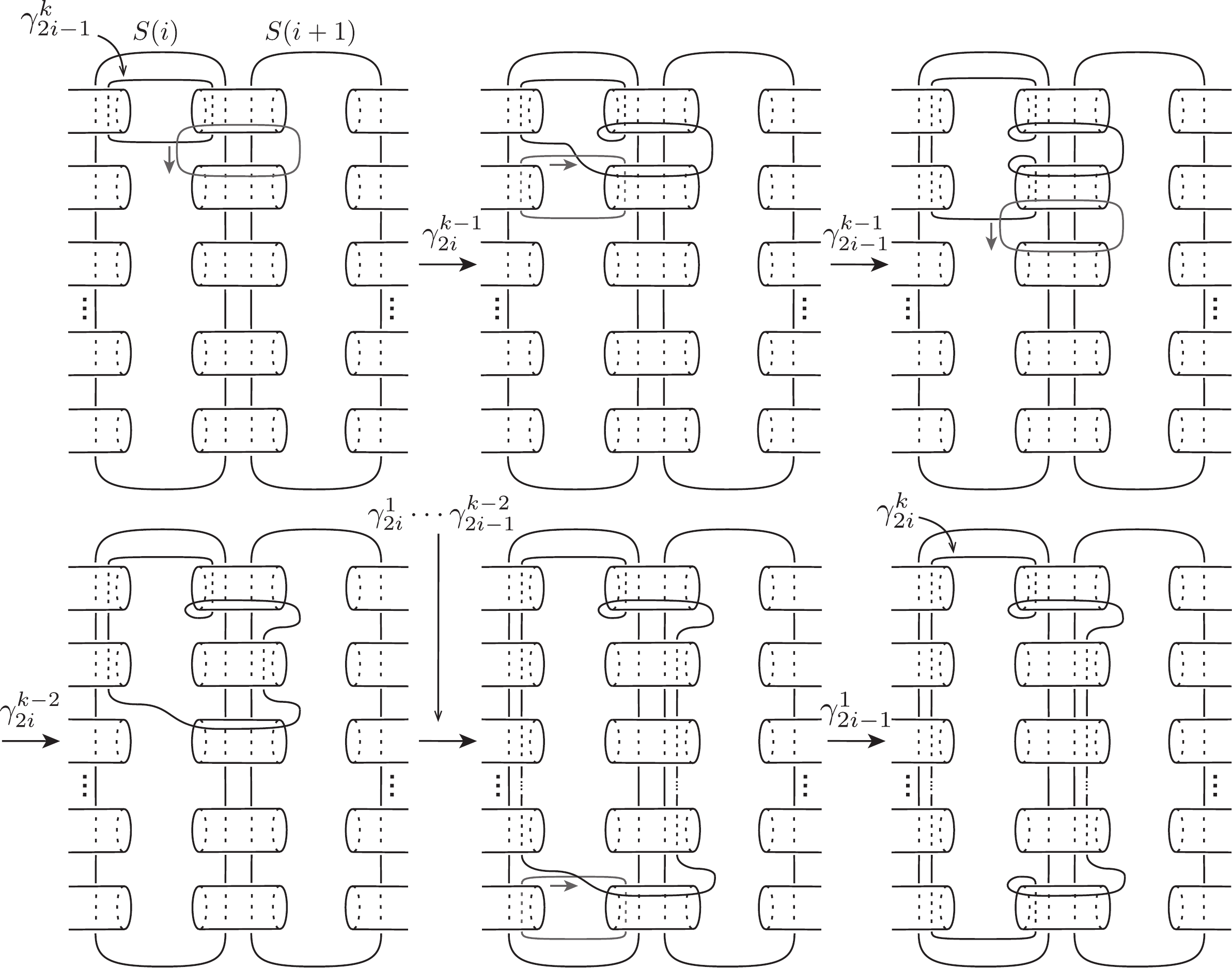}
\caption{Proof of $t_{\gamma _{2i-1}^1}t_{\gamma _{2i}^1}t_{\gamma _{2i-1}^2}t_{\gamma _{2i}^2}\cdots t_{\gamma _{2i-1}^{k-1}}t_{\gamma _{2i}^{k-1}}(\gamma _{2i-1}^{k})=\gamma _{2i}^{k}$ for $1\leq i\leq 2n$, where, in this figure, we express a right-handed Dehn twist $t_{\gamma }$ along $\gamma $ by $\gamma $.}\label{fig_proof_lift_r_i}
\end{figure}

For $1\leq i\leq n$ and $1\leq l\leq k$, the pair of $\delta _i^{l}$ and $\eta _i^{l}$ bounds a proper annulus in $H_g$ as in Figure~\ref{fig_bounded-annulus}. 
Hence the product $t_{\delta _{i}^l}t_{\eta _{i}^l}^{-1}$ is an annulus twist on $H_g$, and by Lemma~\ref{lift-r_i}, $\widetilde{r}_i$ is a product of $k-1$ annulus twists.  

\begin{figure}[h]
\includegraphics[scale=1.0]{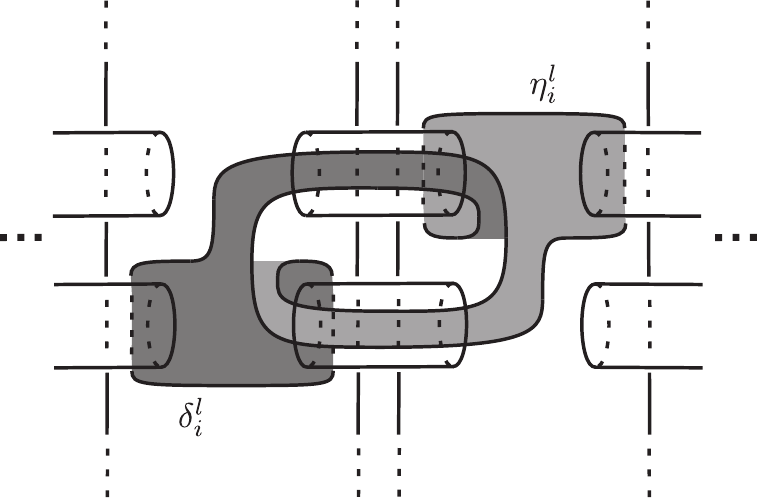}
\caption{A proper annulus in $H_g$ which is bounded by $\delta _i^{l}\sqcup \eta _i^{l}$.}\label{fig_bounded-annulus}
\end{figure}

We identify $H_g$ with the 3-manifold with boundary as on the lower side in Figure~\ref{fig_handlebody-homeo2} by a natural homeomorphism. 
Let $\mu _{i}^{l}$ for $1\leq i\leq n$ and $1\leq l\leq k$ be a simple closed curves on $\Sigma _g$ as in Figure~\ref{fig_handlebody-homeo2}, and $\nu _{i}^{l}$ and $\xi _i^{l^\prime}$ for $1\leq i\leq n-1$, $1\leq l\leq k-1$, and $1\leq l^\prime \leq k-2$ simple closed curves on $\Sigma _g$ as in Figure~\ref{fig_scc_nu}. 
Remark that these simple close curves are mutually disjoint and we have $\mu _{1}^l=\gamma _1^l$, $\mu _{n}^l=(\gamma _{2n+1}^l)^{-1}$, $\nu _i^1=\gamma _{2i+1}^1$, and $\nu _i^{k-1}=\gamma _{2i+1}^k$. 
The simple closed curves $\nu _{i}^{l}$ and $\xi _i^{l^\prime}$ are defined for $n\geq 2$. 
Denote by $R$ a homeomorphism on $H_g$ which is described as the result of a $\pi $-rotation of $H_g$ as in Figure~\ref{fig_rotation1}. 
Then we have the following lemma.  

\begin{figure}[h]
\includegraphics[scale=0.85]{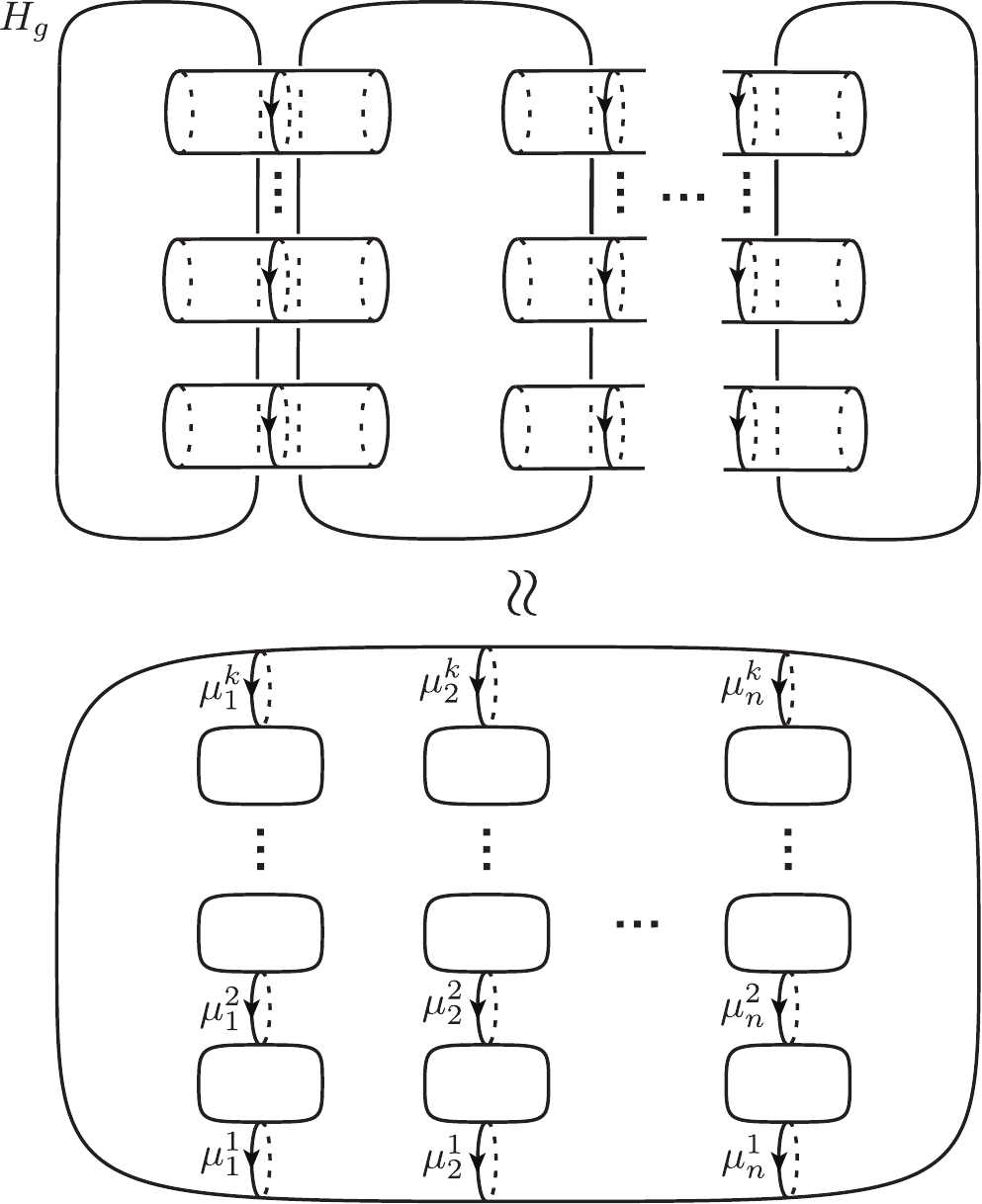}
\caption{A natural homeomorphism of $H_g$ and a simple closed curve $\mu _i^{l}$ for $1\leq i\leq n$ and $1\leq l\leq k$.}\label{fig_handlebody-homeo2}
\end{figure}

\begin{figure}[h]
\includegraphics[scale=0.90]{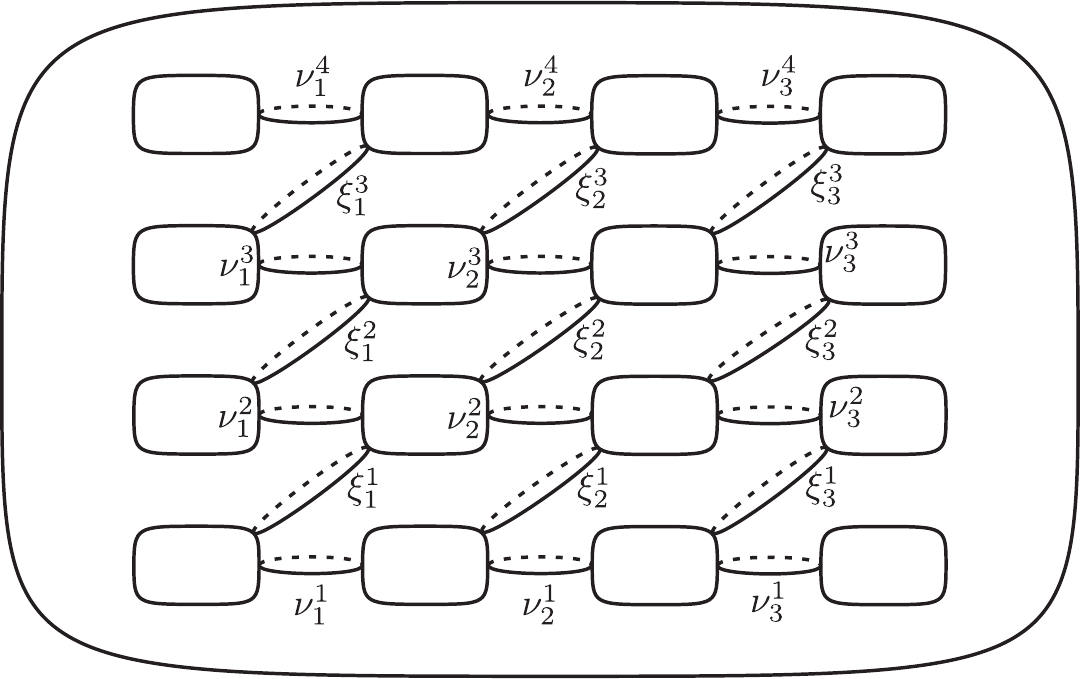}
\caption{Simple closed curves $\nu _i^{l}$ and $\xi _i^{l^\prime}$ for $1\leq i\leq n-1$, $1\leq l\leq k-1$, and $1\leq l^\prime \leq k-2$ when $n=5$ and $k=5$.}\label{fig_scc_nu}
\end{figure}

\begin{figure}[h]
\includegraphics[scale=0.80]{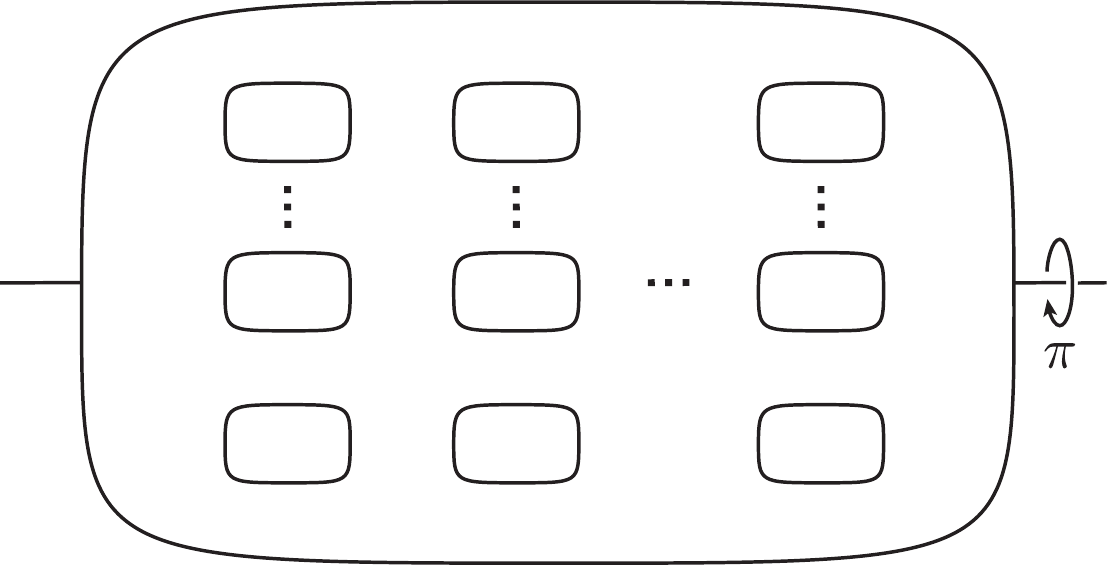}
\caption{A $\pi $-rotation $R$ of $H_g$.}\label{fig_rotation1}
\end{figure}

\begin{lem}\label{lift-r}
For some lift $\widetilde{r}$ of $r$ with respect to $p$, we have
\[
\widetilde{r}=t_{\mu _1^1}t_{\mu _n^k}\left( \prod _{\substack{1\leq i\leq n\\ 2\leq l\leq k-1}}t_{\mu _i^l}\right) \left( \prod _{\substack{1\leq i\leq n-1\\ 1\leq l\leq k-1}}t_{\nu _i^l}\right) \left( \prod _{\substack{1\leq i\leq n-1\\ 1\leq l\leq k-2}}t_{\xi _i^l}^{-1}\right) R.
\] 
\end{lem}

\begin{proof}
Recall that $L=l_1\cup \cdots \cup l_{n+1}\subset S ^2$ and $\displaystyle \widetilde{L}=p^{-1}(L)=\bigcup _{\substack{1\leq i\leq 2n+1\\ 1\leq l\leq k}}\widetilde{l}_{i}^l\subset \Sigma _g$ (see Figures~\ref{fig_path_l} and \ref{fig_isotopy_surface_3-handles}). 
The isotopy class of a self-homeomorphism $\varphi $ on $H_g$ (resp. on $B^3$ relative to $\A $) is determined by the isotopy class of $\varphi (\widetilde{L})$ (resp. $\varphi (L)$) in $\Sigma _g$ (resp. in $S^2$ relative to $\B $). 
Since $r(L)$ is the union of arcs as on the lower side in Figure~\ref{fig_proof_lift_r-1}, for some lift $\widetilde{r}$ of $r$, the image $\widetilde{r}(\widetilde{L})$ is the union of arcs as on the upper side in Figure~\ref{fig_proof_lift_r-1}. 

Since the surface which is obtained from $\Sigma _g$ by cutting along the union $\displaystyle \Bigl( \bigcup _{\substack{1\leq i\leq n\\ 2\leq l\leq k}}\mu _i^l\Bigr) \cup \Bigl( \bigcup _{\substack{1\leq i\leq n\\ 1\leq l\leq k-1}}\gamma _{2i}^l\Bigr) \cup \Bigl( \bigcup _{2\leq i\leq n}\gamma _{2i-1}^1\Bigr) $ is the disjoint union of disks, the isotopy classes of the images of these curves by $\widetilde{r}$ determine $\widetilde{r}$. 
The simple closed curve $\mu _i^l$ transversely intersects with $\widetilde{l}_{2i}^l$ at one point from the left-hand side of $\widetilde{l}_{2i}^l$, $\gamma _{2i}^l$ transversely intersects with $\widetilde{l}_{2i-1}^{l+1}$ at one point from the left-hand side of $\widetilde{l}_{2i-1}^{l+1}$ and with $\widetilde{l}_{2i+1}^{l+1}$ at one point from the right-hand side of $\widetilde{l}_{2i+1}^{l+1}$, and $\gamma _{2i-1}^1$ for $2\leq i\leq n$ transversely intersects with $\widetilde{l}_{2i-2}^{1}$ at one point from the right-hand side of $\widetilde{l}_{2i-2}^{1}$ and with $\widetilde{l}_{2i}^{1}$ at one point from the right-hand side of $\widetilde{l}_{2i}^{1}$. 
Hence we have $\widetilde{r}(\mu _i^l)=\mu _i^{k-l+1}$ and $\widetilde{r}(\gamma _{2i-1}^1)=\gamma _{2i-1}^k$ with orientations, and $\widetilde{r}(\gamma _{2i}^l)$ is a simple closed curve as in Figure~\ref{fig_proof_lift_r-2} (remark that $\widetilde{r}(\gamma _{2i}^{l-1})=\zeta (\widetilde{r}(\gamma _{2i}^{l}))$). 

By acting $R$, we have $R(\mu _i^l)=\mu _i^{k-l+1}$, $R(\gamma _{2i-1}^1)=\gamma _{2i-1}^k$, and $R(\gamma _{2i}^l)=(\gamma _{2i}^{k-l})^{-1}$ for $1\leq l\leq k-1$. 
The product $\displaystyle t_{\mu _1^1}t_{\mu _n^k}\Bigl( \prod _{\substack{1\leq i\leq n\\ 2\leq l\leq k-1}}t_{\mu _i^l}\Bigr) \Bigl( \prod _{\substack{1\leq i\leq n-1\\ 1\leq l\leq k-1}}t_{\nu _i^l}\Bigr) \Bigl( \prod _{\substack{1\leq i\leq n-1\\ 1\leq l\leq k-2}}t_{\xi _i^l}^{-1}\Bigr) $ acts on $\mu _i^{k-l+1}$ and $\gamma _{2i-1}^k$ trivially, and we can check that 
\begin{align*}
&\biggl( t_{\mu _1^1}t_{\mu _n^k}\Bigl( \prod _{\substack{1\leq i\leq n\\ 2\leq l\leq k-1}}t_{\mu _i^l}\Bigr) \Bigl( \prod _{\substack{1\leq i\leq n-1\\ 1\leq l\leq k-1}}t_{\nu _i^l}\Bigr) \Bigl( \prod _{\substack{1\leq i\leq n-1\\ 1\leq l\leq k-2}}t_{\xi _i^l}^{-1}\Bigr) \biggr) ((\gamma _{2i}^{k-l})^{-1})\\
=&t_{\mu _{i}^{k-l}}t_{\mu _{i}^{k-l+1}}t_{\nu _{i-1}^{k-l}}t_{\xi _{i-1}^{k-l-1}}^{-1}t_{\nu _{i}^{k-l}}t_{\xi _{i}^{k-l}}^{-1}((\gamma _{2i}^{k-l})^{-1})\\
=&\widetilde{r}(\gamma _{2i}^l),
\end{align*}
where we regard the homeomorphisms in the equation above which are not defined as the identity map. 
Therefore, we have $\displaystyle \widetilde{r}=t_{\mu _1^1}t_{\mu _n^k}\Bigl( \prod _{\substack{1\leq i\leq n\\ 2\leq l\leq k-1}}t_{\mu _i^l}\Bigr) \Bigl( \prod _{\substack{1\leq i\leq n-1\\ 1\leq l\leq k-1}}t_{\nu _i^l}\Bigr) \Bigl( \prod _{\substack{1\leq i\leq n-1\\ 1\leq l\leq k-2}}t_{\xi _i^l}^{-1}\Bigr) R$ and have completed the proof of Lemma~\ref{lift-r}. 
\end{proof}

\begin{figure}[h]
\includegraphics[scale=1.0]{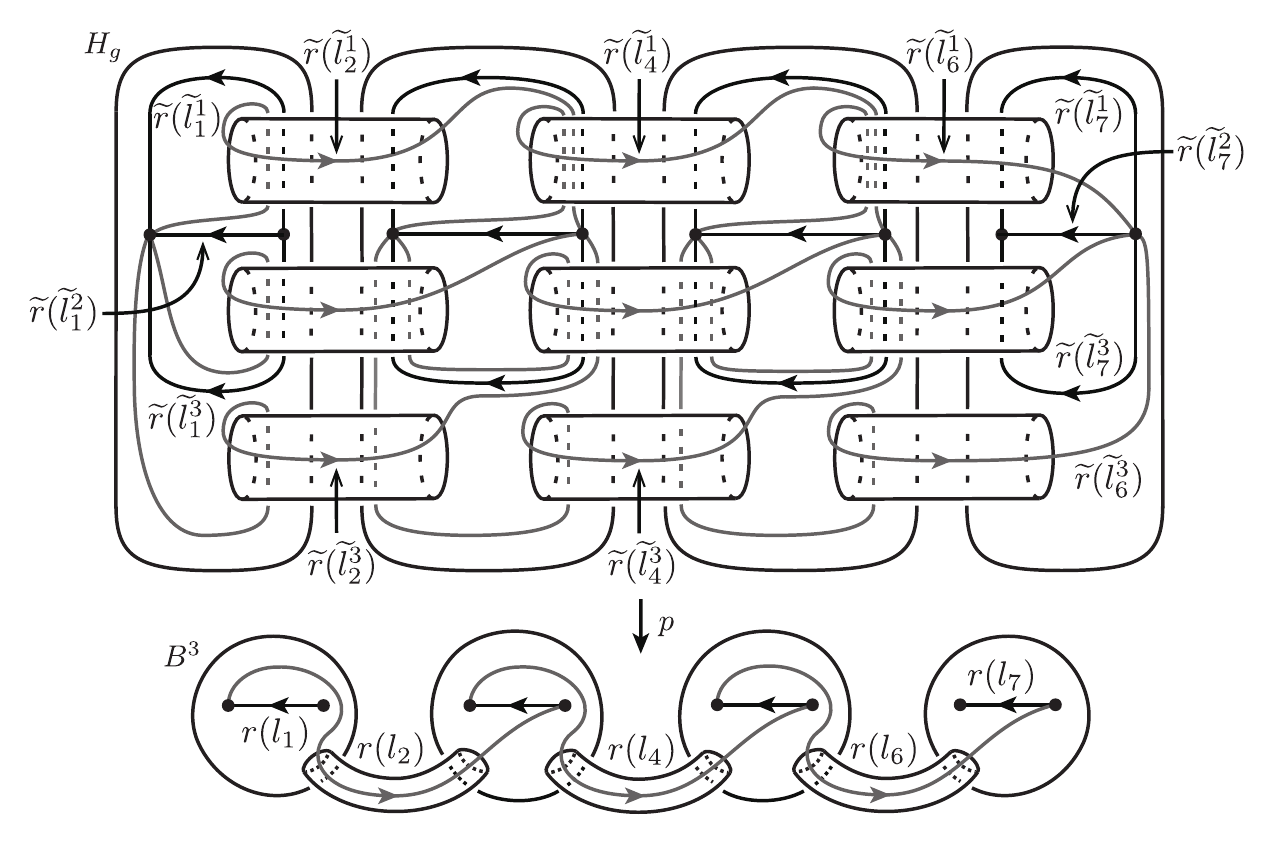}
\caption{The images $r(L)$ and $\widetilde{r}(\widetilde{L})$ when $n=3$ and $k=3$.}\label{fig_proof_lift_r-1}
\end{figure}

\begin{figure}[h]
\includegraphics[scale=1.14]{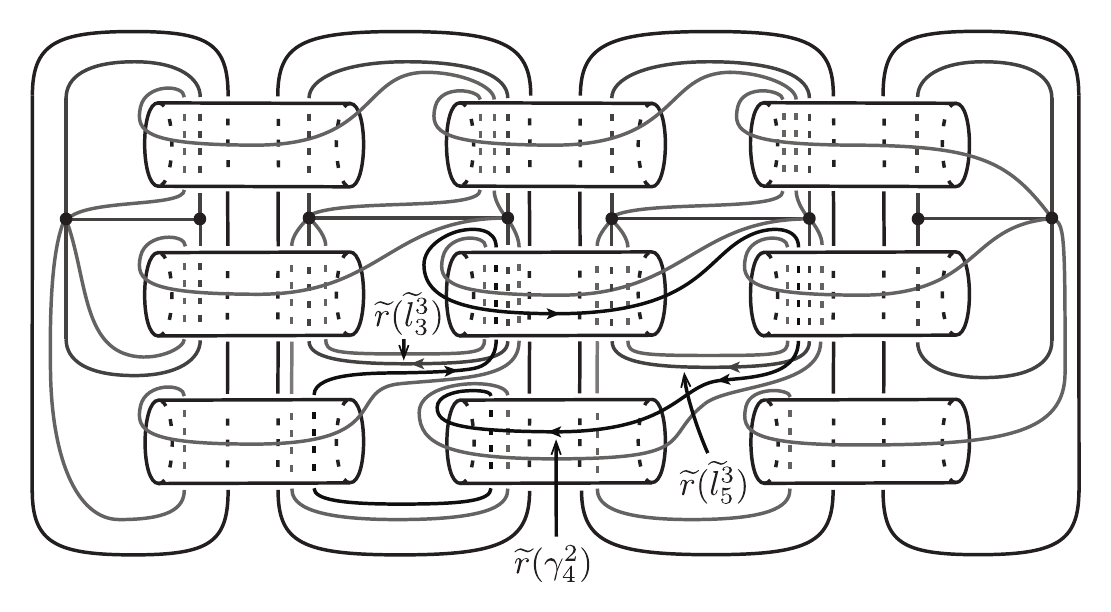}
\caption{The image $\widetilde{r}(\gamma _4^2)$ when $n=3$ and $k=3$.}\label{fig_proof_lift_r-2}
\end{figure}

Denote 
\[
\widetilde{r}=t_{\mu _1^1}t_{\mu _n^k}\Bigl( \prod _{\substack{1\leq i\leq n\\ 2\leq l\leq k-1}}t_{\mu _i^l}\Bigr) \Bigl( \prod _{\substack{1\leq i\leq n-1\\ 1\leq l\leq k-1}}t_{\nu _i^l}\Bigr) \Bigl( \prod _{\substack{1\leq i\leq n-1\\ 1\leq l\leq k-2}}t_{\xi _i^l}^{-1}\Bigr) R.
\] 
We remark that $\widetilde{r}=t_{\mu _1^1}t_{\mu _1^2}\cdots t_{\mu _1^k}R$ when $n=1$. 

Recall that the braid $z\in B_{2n+2}$ lies in the left-hand side of Figure~\ref{fig_braid_z} and the image of $z$ with respect to the natural homomorphism $B_{2n+2}\to SB_{2n+2}$ is trivial (hence its image with respect to $\Gamma \colon SW_{2n+2}\to \Hil $ is also trivial). 
As in Figure~\ref{fig_braid_z}, we have $z=r_1r_2\cdots r_ns_n\cdots s_2s_1t_{1}$ in $B_{2n+2}$. 
The next lemma give a factorization of $\zeta $ by lifts of generators for $\LH $ with respect to $p$.  

\begin{lem}\label{lem_zeta_prod}
We have $\zeta =\widetilde{r}_1\widetilde{r}_2\cdots \widetilde{r}_{n}\widetilde{s}_n\cdots \widetilde{s}_2\widetilde{s}_1\widetilde{t}_{1}$ in $\SH $. 
\end{lem}

\begin{proof}
Since $r_1r_2\cdots r_ns_n\cdots s_2s_1t_{1}=1$ in $\LH $, there exists $0\leq l\leq k-1$ such that $\widetilde{r}_1\widetilde{r}_2\cdots \widetilde{r}_{n}\widetilde{s}_n\cdots \widetilde{s}_2\widetilde{s}_1\widetilde{t}_{1}=\zeta ^l$. 
Thus, it is enough for completing the proof of Lemma~\ref{lem_zeta_prod} to prove that $\widetilde{r}_1\widetilde{r}_2\cdots \widetilde{r}_{n}\widetilde{s}_n\cdots \widetilde{s}_2\widetilde{s}_1\widetilde{t}_{1}(\gamma _{2n+1}^1)=\gamma _{2n+1}^2$. 
Since the supports of $\widetilde{s}_i$ for $1\leq i\leq n-1$ and $\widetilde{t}_j$ are disjoint from $\gamma _{2n+1}^1$, we have $\widetilde{r}_1\widetilde{r}_2\cdots \widetilde{r}_{n}\widetilde{s}_n\cdots \widetilde{s}_2\widetilde{s}_1\widetilde{t}_{1}(\gamma _{2n+1}^1)=\widetilde{r}_1\widetilde{r}_2\cdots \widetilde{r}_{n}\widetilde{s}_n(\gamma _{2n+1}^1)=\widetilde{r}_1\widetilde{r}_2\cdots \widetilde{r}_{n}\widetilde{h}_{2n}t_{\gamma _{2n-1}^1}t_{\gamma _{2n}^1}(\gamma _{2n+1}^1)$. 
By an argument as in Figure~\ref{fig_proof_zeta_prod}, we have $\widetilde{r}_{n}\widetilde{s}_n(\gamma _{2n+1}^1)=\gamma _{2n+1}^2$. 
Since the support of $\widetilde{r}_i$ for $1\leq i\leq n-1$ are disjoint from $\gamma _{2n+1}^2$, we have $\widetilde{r}_1\widetilde{r}_2\cdots \widetilde{r}_{n}\widetilde{s}_n\cdots \widetilde{s}_2\widetilde{s}_1\widetilde{t}_{1}(\gamma _{2n+1}^1)=\widetilde{r}_1\widetilde{r}_2\cdots \widetilde{r}_{n-1}(\gamma _{2n+1}^2)=\gamma _{2n+1}^2$. 
Therefore, we have completed the proof of Lemma~\ref{lem_zeta_prod}. 
\end{proof}

\begin{figure}[p]
\includegraphics[scale=0.52]{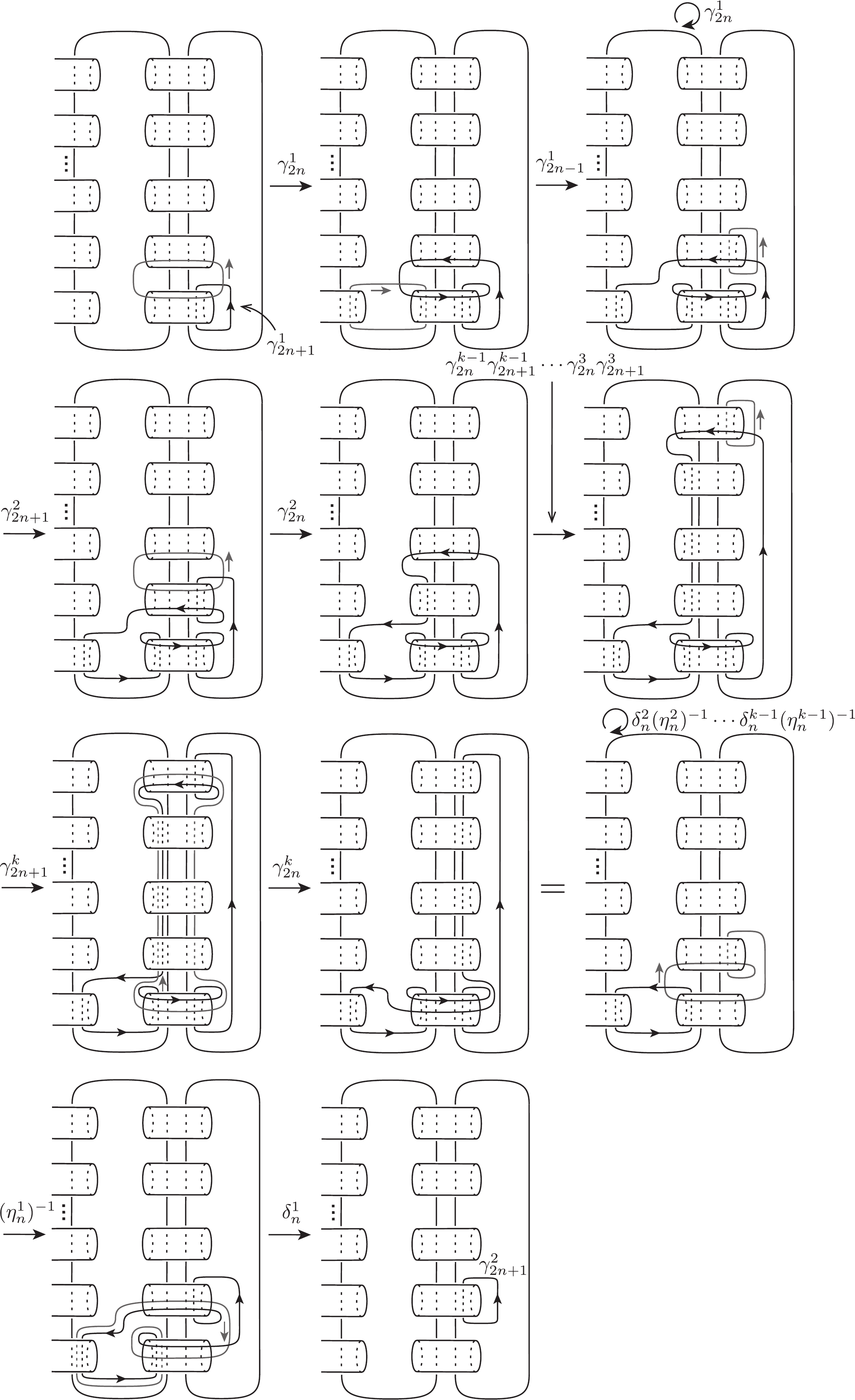}
\caption{The proof of $\widetilde{r}_{n}\widetilde{s}_n(\gamma _{2n+1}^1)=\widetilde{r}_{n}\widetilde{h}_{2n}t_{\gamma _{2n-1}^1}t_{\gamma _{2n}^1}(\gamma _{2n+1}^1)=\gamma _{2n+1}^2$, where we express a right-handed Dehn twist $t_{\gamma }$ along $\gamma $ by $\gamma $ in this figure.}\label{fig_proof_zeta_prod}
\end{figure}

\subsection{Main theorem for the balanced superelliptic handlebody groups}\label{section_main_thm_smod}
Recall that 
\begin{itemize}
\item $\widetilde{s}_i=\widetilde{h}_{2i}\widetilde{h}_{2i-1}\widetilde{t}_i^{-1}$,\quad $\widetilde{r}_i=\widetilde{h}_{2i}^{-1}\widetilde{h}_{2i-1}=t_{\delta _{i}^1}t_{\eta _{i}^1}^{-1}t_{\delta _{i}^2}t_{\eta _{i}^2}^{-1}\cdots t_{\delta _{i}^{k-1}}t_{\eta _{i}^{k-1}}^{-1}$, 
\item $\widetilde{h}_{2i-1}=t_{\gamma _{2i-1}^1}t_{\gamma _{2i}^1}t_{\gamma _{2i-1}^2}t_{\gamma _{2i}^2}\cdots t_{\gamma _{2i-1}^{k-1}}t_{\gamma _{2i}^{k-1}}t_{\gamma _{2i-1}^{k}}$, 
\item $\widetilde{h}_{2i}=t_{\gamma _{2i}^{k}}t_{\gamma _{2i+1}^{k}}t_{\gamma _{2i}^{k-1}}t_{\gamma _{2i+1}^{k-1}}\cdots t_{\gamma _{2i}^{2}}t_{\gamma _{2i+1}^{2}}t_{\gamma _{2i}^{1}}$\quad for $1\leq i\leq n$, 
\item $\widetilde{t}_j=t_{\gamma _{2j-1}^1}t_{\gamma _{2j-1}^2}\cdots t_{\gamma _{2j-1}^{k}}$ for $1\leq j \leq n+1$, and
\item $\displaystyle \widetilde{r}=t_{\mu _1^1}t_{\mu _n^k}\prod _{\substack{1\leq i\leq n\\ 2\leq l\leq k-1}}t_{\mu _i^l}\prod _{\substack{1\leq i\leq n-1\\ 1\leq l\leq k-1}}t_{\nu _i^l} \prod _{\substack{1\leq i\leq n-1\\ 1\leq l\leq k-2}}t_{\xi _i^l}^{-1}R$  
\end{itemize}
(see also Figures~\ref{fig_scc_c_il}, \ref{fig_scc_delta-eta}, \ref{fig_handlebody-homeo2}, \ref{fig_scc_nu}, and \ref{fig_rotation1}). 
The main theorem for the balanced superelliptic handlebody group is as follows. 

\begin{thm}\label{thm_pres_SH}
For $n\geq 1$ and $k\geq 3$ with $g=n(k-1)$, $\SH $ admits the presentation with generators $\widetilde{s}_i$ for $1\leq i\leq n$, $\widetilde{r}_i$ for $1\leq i\leq n$, $\widetilde{t}_{i}$ for $1\leq i\leq n+1$, and $\widetilde{r}$, and the following defining relations: 
\begin{enumerate}
\item commutative relations
\begin{enumerate}
\item $\widetilde{\alpha }_i \rightleftarrows \widetilde{\beta }_j$ \quad for $j-i\geq 2$ and $\alpha , \beta \in \{ s, r\}$,
\item $\widetilde{\alpha }_{i} \rightleftarrows \widetilde{t}_{j}$ \quad for $j\not =i, i+1$ and $\alpha \in \{ s, r\}$, 
\item $\widetilde{t}_i\rightleftarrows \widetilde{t}_{j}$ \quad for $1\leq i<j\leq n+1$, 
\item $\widetilde{s}_i \rightleftarrows \widetilde{r}$ \quad for $1\leq i\leq n$, 
\item $\widetilde{t}_i \rightleftarrows \widetilde{r}$ \quad for $1\leq i\leq n+1$, 
\end{enumerate}
\item conjugation relations
\begin{enumerate}
\item $\widetilde{\alpha }_i\widetilde{\alpha }_{i+1}\widetilde{\alpha }_i=\widetilde{\alpha }_{i+1}\widetilde{\alpha }_i\widetilde{\alpha }_{i+1}$ \quad for $1\leq i\leq n-1$ and $\alpha \in \{ s, r\}$, 
\item $\widetilde{s}_i^{\varepsilon }\widetilde{s}_{i+1}^{\varepsilon }\widetilde{r}_{i}=\widetilde{r}_{i+1}\widetilde{s}_i^{\varepsilon }\widetilde{s}_{i+1}^{\varepsilon }$ \quad for $1\leq i\leq n-1$ and $\varepsilon \in \{ 1, -1\}$,
\item $\widetilde{r}_i\widetilde{r}_{i+1}\widetilde{s}_{i}=\widetilde{s}_{i+1}\widetilde{r}_i\widetilde{r}_{i+1}$ \quad for $1\leq i\leq n-1$,
\item $\widetilde{r}_i\widetilde{r}\widetilde{s}_{i}=\widetilde{r}\widetilde{s}_{i}\widetilde{r}_{i}^{-1}$ \quad for $1\leq i\leq n$,
\item $\widetilde{s}_i^{\varepsilon }\widetilde{t}_{i}=\widetilde{t}_{i+1}\widetilde{s}_{i}^{\varepsilon }$ \quad for $1\leq i\leq n$ and $\varepsilon \in \{ 1, -1\}$,
\item $\widetilde{r}_i\widetilde{t}_{i}=\widetilde{t}_{i+1}\widetilde{r}_{i}$ \quad for $1\leq i\leq n$,
\item $\widetilde{t}_i\widetilde{s}_{i}^2\widetilde{r}_i=\widetilde{r}_i\widetilde{s}_i^2\widetilde{t}_{i+1}$ \quad for $1\leq i\leq n$,
\end{enumerate}
\item $\widetilde{r}^2=\widetilde{t}_1\widetilde{t}_2\cdots \widetilde{t}_{n+1}$, 
\item $(\widetilde{r}_1\widetilde{r}_2\cdots \widetilde{r}_n\widetilde{s}_n\cdots \widetilde{s}_2\widetilde{s}_1\widetilde{t}_{1})^k=1$, 
\item $\widetilde{t}_{n+1}\cdots \widetilde{t}_2\widetilde{t}_1\bigl( (\widetilde{s}_n\widetilde{s}_{n-1}\cdots \widetilde{s}_1)(\widetilde{s}_n\widetilde{s}_{n-1}\cdots \widetilde{s}_2)\cdots (\widetilde{s}_n\widetilde{s}_{n-1})\widetilde{s}_{n}\bigr) ^2=1$,  
\item 
\begin{enumerate}
\item $(\widetilde{r}_1\widetilde{r}_2\cdots \widetilde{r}_n\widetilde{s}_n\cdots \widetilde{s}_2\widetilde{s}_1\widetilde{t}_{1})\rightleftarrows \widetilde{\alpha }_1$ \quad for $\alpha \in \{ s, r\}$,
\item $(\widetilde{r}_1\widetilde{r}_2\cdots \widetilde{r}_n)\widetilde{t}_{n+1}=\widetilde{t}_1(\widetilde{r}_1\widetilde{r}_2\cdots \widetilde{r}_n)$,
\item $\widetilde{r}(\widetilde{r}_1\widetilde{r}_2\cdots \widetilde{r}_n\widetilde{s}_n\cdots \widetilde{s}_2\widetilde{s}_1\widetilde{t}_{1})=(\widetilde{r}_1\widetilde{r}_2\cdots \widetilde{r}_n\widetilde{s}_n\cdots \widetilde{s}_2\widetilde{s}_1\widetilde{t}_{1})^{-1}\widetilde{r}$.
\end{enumerate}
\end{enumerate}

\end{thm}

We remark that $\zeta =\widetilde{r}_1\widetilde{r}_2\cdots \widetilde{r}_n\widetilde{s}_n\cdots \widetilde{s}_2\widetilde{s}_1\widetilde{t}_{1}$ by Lemma~\ref{lem_zeta_prod}. 
We can check that the next lemma holds: 

\begin{lem}\label{lem_image_gamma}
The following equations hold without orientations of curves: 
\begin{enumerate}
\item $(\widetilde{s}_i(\gamma _{2i-1}^l), \widetilde{s}_i(\gamma _{2i+1}^l))=(\gamma _{2i+1}^l, \gamma _{2i-1}^l)$ \quad for $1\leq i\leq n$ and $1\leq l\leq k$, 
\item $\widetilde{r}_i(\gamma _{2i-1}^l)=\gamma _{2i+1}^{l+1}$ \quad for $1\leq i\leq n$ and $1\leq l\leq k-1$, 
\item $\widetilde{r}(\gamma _{2i-1}^l)=\gamma _{2i-1}^{k-l+1}$ \quad for $1\leq i\leq n$ and $1\leq l\leq k$. 
\end{enumerate}
\end{lem}

Remark that by the exact sequence~\ref{exact_SH_handlebody} in Lemma~\ref{lem_exact_SH_handlebody}, if $\theta (f)=1\in \LH $ for $f\in \SH$, then there exists $0\leq l\leq k-1$ such that $f=\zeta ^l$. 
In particular, for $f\in \SH $ with $\theta (f)=1$, $f=\zeta ^l$ if and only if $f(\gamma _{i}^1)=\gamma _i^{l+1}$ for some $1\leq i\leq 2n+1$. 
We use this property in the following lemmas. 

\begin{lem}\label{lem_lift_comm-rel}
The relations~(1) (a)--(e) in Theorem~\ref{thm_pres_SH} hold in $\SH $. 
\end{lem}

\begin{proof}
The relations~(1) (a)--(c) in Theorem~\ref{thm_pres_SH} are obtained from the disjointness of their supports.  
For the relations~(1) (d), we have $\widetilde{s}_i^{-1}\widetilde{r}^{-1}\widetilde{s}_i\widetilde{r}(\gamma _{2i-1}^1)=\widetilde{s}_i^{-1}\widetilde{r}^{-1}\widetilde{s}_i(\gamma _{2i-1}^{k})=\widetilde{s}_i^{-1}\widetilde{r}^{-1}(\gamma _{2i+1}^{k})=\widetilde{s}_i^{-1}(\gamma _{2i+1}^{1})=\gamma _{2i-1}^{1}$ by Lemma~\ref{lem_image_gamma}. 
Hence, $\widetilde{s}_i^{-1}\widetilde{r}^{-1}\widetilde{s}_i\widetilde{r}=1$ in $\SH $. 
Since $\widetilde{t}_i$ fixes the curve $\gamma _{2j-1}^l$, we have $\widetilde{t}_i^{-1}\widetilde{r}^{-1}\widetilde{t}_i\widetilde{r}=1$ in $\SH $
Therefore, we have completed the proof of Lemma~\ref{lem_lift_comm-rel}.  
\end{proof}

\begin{lem}\label{lem_lift_conj-rel}
The relations~(2) (a)--(g) in Theorem~\ref{thm_pres_SH} hold in $\SH $. 
\end{lem}

\begin{proof}
Since $\widetilde{t}_i$ fixes a curve $\gamma _{2j-1}^l$, as an argument similar to the bottom of the proof of Lemma~\ref{lem_lift_comm-rel}, the relations~(2) (e)--(g) in Theorem~\ref{thm_pres_SH} hold in $\SH $. 
The relations~(2) (a)--(c) in Theorem~\ref{thm_pres_SH} also hold in $\SH $ by using Lemma~\ref{lem_image_gamma}. 
For instance, for the relation~(2) (a) when $\alpha =r$, we have $\widetilde{r}_{i+1}^{-1}\widetilde{r}_i^{-1}\widetilde{r}_{i+1}^{-1}\widetilde{r}_i\widetilde{r}_{i+1}\widetilde{r}_i(\gamma _{2i-1}^1)=\widetilde{r}_{i+1}^{-1}\widetilde{r}_i^{-1}\widetilde{r}_{i+1}^{-1}\widetilde{r}_i(\gamma _{2i+3}^3)=\widetilde{r}_{i+1}^{-1}\widetilde{r}_i^{-1}\widetilde{r}_{i+1}^{-1}(\gamma _{2i+3}^3)=\widetilde{r}_{i+1}^{-1}(\gamma _{2i-1}^1)=\gamma _{2i-1}^1$. 
Finally, the relations~(2) (d) in Theorem~\ref{thm_pres_SH} is equivalent to the relation $\widetilde{r}\widetilde{s}_{i}\widetilde{r}_{i}\widetilde{s}_{i}^{-1}\widetilde{r}^{-1}\widetilde{r}_i=1$. 
Then we have $\widetilde{r}\widetilde{s}_{i}\widetilde{r}_{i}\widetilde{s}_{i}^{-1}\widetilde{r}^{-1}\widetilde{r}_i(\gamma _{2i-1}^1)=\widetilde{r}\widetilde{s}_{i}\widetilde{r}_{i}\widetilde{s}_{i}^{-1}\widetilde{r}^{-1}(\gamma _{2i+1}^2)=\widetilde{r}\widetilde{s}_{i}\widetilde{r}_{i}\widetilde{s}_{i}^{-1}(\gamma _{2i+1}^{k-1})=\widetilde{r}\widetilde{s}_{i}\widetilde{r}_{i}(\gamma _{2i-1}^{k-1})=\widetilde{r}\widetilde{s}_{i}(\gamma _{2i+1}^{k})=\widetilde{r}(\gamma _{2i-1}^{k})=\gamma _{2i-1}^{1}$. 
Therefore, the relations(2) (d) in Theorem~\ref{thm_pres_SH} holds in $\SH $ and we have completed the proof of Lemma~\ref{lem_lift_conj-rel}.  
\end{proof}

\begin{lem}\label{lem_lift_rel-square-r}
The relation~(3) in Theorem~\ref{thm_pres_SH} holds in $\SH $. 
\end{lem}

\begin{proof}
Let ${\xi _i^\prime }^{l}$ for $1\leq i\leq n-1$ and $1\leq l\leq k-2$ be a simple closed curve on $\Sigma _g$ as in Figure~\ref{fig_scc_xi^prime}. 
Remark that $R(\xi _i^l)={\xi _i^\prime }^{k-l-1}$. 
By the definition of $\widetilde{r}$, we have
\begin{eqnarray*}
\widetilde{r}^2&=&\biggl( t_{\mu _1^1}t_{\mu _n^k}\Bigl( \prod _{\substack{1\leq i\leq n\\ 2\leq l\leq k-1}}t_{\mu _i^l}\Bigr) \Bigl( \prod _{\substack{1\leq i\leq n-1\\ 1\leq l\leq k-1}}t_{\nu _i^l}\Bigr) \Bigl( \prod _{\substack{1\leq i\leq n-1\\ 1\leq l\leq k-2}}t_{\xi _i^l}^{-1}\Bigr) R\biggr) ^2\\
&=& t_{\mu _1^1}t_{\mu _n^k}\Bigl( \prod _{\substack{1\leq i\leq n\\ 2\leq l\leq k-1}}t_{\mu _i^l}\Bigr) \Bigl( \prod _{\substack{1\leq i\leq n-1\\ 1\leq l\leq k-1}}t_{\nu _i^l}\Bigr) \Bigl( \prod _{\substack{1\leq i\leq n-1\\ 1\leq l\leq k-2}}t_{\xi _i^l}^{-1}\Bigr) \\
&&\cdot  t_{\mu _1^k}t_{\mu _n^1}\Bigl( \prod _{\substack{1\leq i\leq n\\ 2\leq l\leq k-1}}t_{\mu _i^l}\Bigr) \Bigl( \prod _{\substack{1\leq i\leq n-1\\ 1\leq l\leq k-1}}t_{\nu _i^l}\Bigr) \Bigl( \prod _{\substack{1\leq i\leq n-1\\ 1\leq l\leq k-2}}t_{{\xi _i^\prime }^l}^{-1}\Bigr) \\
&=& t_{\mu _1^1}t_{\mu _1^k}t_{\mu _n^1}t_{\mu _n^k}\Bigl( \prod _{\substack{1\leq i\leq n\\ 2\leq l\leq k-1}}t_{\mu _i^l}^2\Bigr) \Bigl( \prod _{\substack{1\leq i\leq n-1\\ 1\leq l\leq k-1}}t_{\nu _i^l}^2\Bigr) \Bigl( \prod _{\substack{1\leq i\leq n-1\\ 1\leq l\leq k-2}}t_{\xi _i^l}^{-1}t_{{\xi _i^\prime }^l}^{-1}\Bigr) \\
&=& (t_{\mu _1^1}\cdots t_{\mu _1^k})(t_{\mu _n^1}\cdots t_{\mu _n^k})\Bigl( \prod _{\substack{1\leq i\leq n-1}}t_{\nu _i^1}t_{\nu _i^{k-1}}\Bigr) \\
&&\cdot \Bigl( \prod _{\substack{1\leq i\leq n-1\\ 1\leq l\leq k-2}}t_{\mu _{i}^{l+1}}t_{\mu _{i+1}^{l+1}}t_{\nu _i^l}t_{\nu _i^{l+1}}\Bigr) \Bigl( \prod _{\substack{1\leq i\leq n-1\\ 1\leq l\leq k-2}}t_{\xi _i^l}^{-1}t_{{\xi _i^\prime }^l}^{-1}\Bigr) \\
&=& \widetilde{t}_1\widetilde{t}_{n+1}\Bigl( \prod _{\substack{1\leq i\leq n-1}}t_{\gamma _{2i+1}^1}t_{\gamma _{2i+1}^{k}}\Bigr) \Bigl( \prod _{\substack{1\leq i\leq n-1\\ 1\leq l\leq k-2}}t_{\mu _{i}^{l+1}}t_{\mu _{i+1}^{l+1}}t_{\nu _i^l}t_{\nu _i^{l+1}}t_{\xi _i^l}^{-1}t_{{\xi _i^\prime }^l}^{-1}\Bigr) .
\end{eqnarray*}
By the lantern relatons, we have $t_{\mu _{i}^{l+1}}t_{\mu _{i+1}^{l+1}}t_{\nu _i^l}t_{\nu _i^{l+1}}t_{\xi _i^l}^{-1}t_{{\xi _i^\prime }^l}^{-1}=t_{\gamma _{2i+1}^{l+1}}$ for $1\leq i\leq n-1$ and $1\leq l\leq k-2$. 
Therefore, the relation $\widetilde{r}^2=\widetilde{t}_1\widetilde{t}_2\cdots \widetilde{t}_{n+1}$ holds in $\SH $ and we have completed the proof of Lemma~\ref{lem_lift_rel-square-r}.  
\end{proof}

\begin{figure}[h]
\includegraphics[scale=0.85]{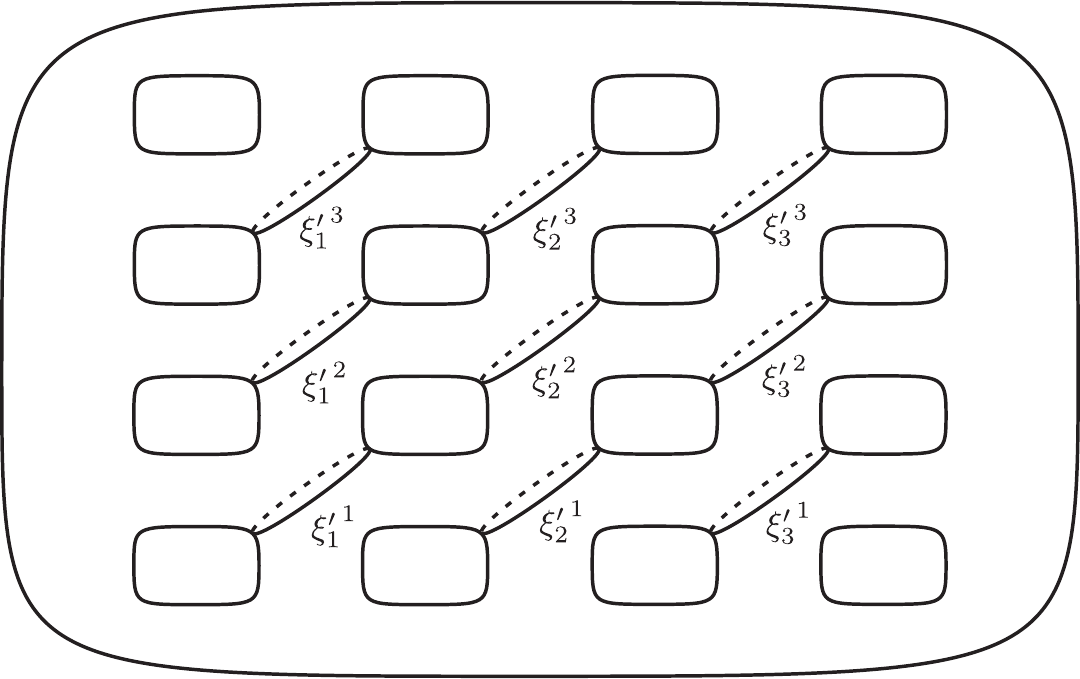}
\caption{Simple closed curve ${\xi _i^\prime }^{l}$ for $1\leq i\leq n-1$ and $1\leq l\leq k-2$ when $n=5$ and $k=5$.}\label{fig_scc_xi^prime}
\end{figure}

\begin{lem}\label{lem_rel_conj-zeta-s}
The commutative relations $\zeta \rightleftarrows \widetilde{\alpha }_i$ for $\alpha \in \{ s, r\}$ and $2\leq i\leq n$ and $\zeta \rightleftarrows \widetilde{t}_i$ for $1\leq i\leq n+1$ are obtained from the relations~(1), (2), and (6) (b) in Theorem~\ref{thm_pres_SH} and $\zeta =\widetilde{r}_1\widetilde{r}_2\cdots \widetilde{r}_{n}\widetilde{s}_n\cdots \widetilde{s}_2\widetilde{s}_1\widetilde{t}_{1}$. 
\end{lem}

\begin{proof}
Recall that $\zeta =\widetilde{r}_1\widetilde{r}_2\cdots \widetilde{r}_{n}\widetilde{s}_n\cdots \widetilde{s}_2\widetilde{s}_1\widetilde{t}_{1}$ by Lemma~\ref{lem_zeta_prod}. 
We can check that the relations $\zeta \rightleftarrows \widetilde{\alpha }_i$ for $\alpha \in \{ s, r\}$ and $2\leq i\leq n$ and $\zeta \rightleftarrows \widetilde{t}_2$ for $1\leq i\leq n+1$ are obtained from the relations~(1), (2), and (6) in Theorem~\ref{thm_pres_SH}. 
For the relation $\zeta \rightleftarrows \widetilde{t}_1$, we have $(\widetilde{r}_1\widetilde{r}_2\cdots \widetilde{r}_{n}\underline{\widetilde{s}_n\cdots \widetilde{s}_2\widetilde{s}_1\widetilde{t}_{1}})\widetilde{t}_{1}\overset{\text{(2)(e)}}{\underset{}{=}}\underline{(\widetilde{r}_1\widetilde{r}_2\cdots \widetilde{r}_{n}\widetilde{t}_{n+1}}\widetilde{s}_n\cdots \widetilde{s}_2\widetilde{s}_1)\widetilde{t}_{1}\overset{\text{(6)(b)}}{\underset{}{=}}\widetilde{t}_{1}(\widetilde{r}_1\widetilde{r}_2\cdots \widetilde{r}_{n}\widetilde{s}_n\cdots \widetilde{s}_2\widetilde{s}_1\widetilde{t}_{1})$. 
Therefore, we have completed the proof of Lemma~\ref{lem_rel_conj-zeta-s}. 
\end{proof}

\begin{proof}[Proof of Theorem~\ref{thm_pres_SH}]
We apply Lemma~\ref{presentation_exact} to the exact sequence~\ref{exact_SH_handlebody} in Lemma~\ref{lem_exact_SH_handlebody} and the presentation for $\LH $ in Theorem~\ref{thm_pres_LH}. 
By Lemmas~\ref{lift-t_{i,i+1}}, \ref{lift-s_i-r_i}, and \ref{lift-r}, the generators of the presentation for $\SH $ are $\widetilde{s}_i$, $\widetilde{r}_i$ for $1\leq i\leq n$, $\widetilde{t}_{j}$ for $1\leq j\leq n+1$, $\widetilde{r}$, and $\zeta $. 
By Lemma~\ref{lem_zeta_prod}, we have $\zeta =\widetilde{r}_1\widetilde{r}_2\cdots \widetilde{r}_{n}\widetilde{s}_n\cdots \widetilde{s}_2\widetilde{s}_1\widetilde{t}_{1}$. 

By Lemmas~\ref{lem_lift_comm-rel}, \ref{lem_lift_conj-rel}, and \ref{lem_lift_rel-square-r}, the relations~(1), (2), and (3) in Theorem~\ref{thm_pres_SH}, that are lifts of the relations~(1), (2), and (3) in Theorem~\ref{thm_pres_LH} with respect to $\theta $. 
Since $\widetilde{r}_1\widetilde{r}_2\cdots \widetilde{r}_{n}\widetilde{s}_n\cdots \widetilde{s}_2\widetilde{s}_1\widetilde{t}_{1}=\zeta $, the relation~(4) in Theorem~\ref{thm_pres_SH} 
holds in $\SH $ and coincides with the relation $\zeta ^k=1$. 
By Lemma~\ref{lem_image_gamma}, we have $\left( \widetilde{t}_{n+1}\cdots \widetilde{t}_2\widetilde{t}_1\bigl( (\widetilde{s}_n\widetilde{s}_{n-1}\cdots \widetilde{s}_1)(\widetilde{s}_n\widetilde{s}_{n-1}\cdots \widetilde{s}_2)\cdots (\widetilde{s}_n\widetilde{s}_{n-1})\widetilde{s}_{n}\bigr) ^2\right) (\gamma _{2n+1}^l)=\gamma _{2n+1}^l$. 
Hence the relation~(5) in Theorem~\ref{thm_pres_SH} holds in $\SH $. 

The conjugation of $\zeta =\widetilde{r}_1\widetilde{r}_2\cdots \widetilde{r}_{n}\widetilde{s}_n\cdots \widetilde{s}_2\widetilde{s}_1\widetilde{t}_{1}$ by $\widetilde{\alpha }_i$ for $\alpha \in \{ s,r \}$ and $1\leq i\leq n$ (resp. $\widetilde{t}_i$ for $1\leq i\leq n+1$) coincides with $\zeta $ by Lemma~\ref{lem_image_gamma}. 
These conjugation (commutative) relations are obtained from the relations~(1), (2), (6) (a), and (b) in Theorem~\ref{thm_pres_SH} by Lemma~\ref{lem_rel_conj-zeta-s}. 
Finally, we can check that $\left( \widetilde{r}(\widetilde{r}_1\widetilde{r}_2\cdots \widetilde{r}_{n}\widetilde{s}_n\cdots \widetilde{s}_2\widetilde{s}_1\widetilde{t}_{1})\widetilde{r}^{-1}\right) (\gamma _{2n+1}^{1})=\gamma _{2n+1}^{k}$ by Lemma~\ref{lem_image_gamma}. 
 Therefore, the relation~(6)~(c) holds in $\SH $ and we have completed the proof of Theorem~\ref{thm_pres_SH}. 
\end{proof}

\subsection{The first homology groups of the balanced superelliptic handlebody groups}\label{section_abel-smod}

In this section, we will prove Theorem~\ref{thm_abel_smod}.  
For conveniences, we denote the equivalence class in $H_1(\SH )$ of an element $h$ in $\SH $ by $h$. 

\begin{proof}[Proof of Theorem~\ref{thm_abel_smod}]
The relations~(2) (a), (b), and (c) in Theorem~\ref{thm_pres_SH} are equivalent to the relations $\widetilde{s}_i=\widetilde{s}_{i+1}$ and $\widetilde{r}_i=\widetilde{r}_{i+1}$ for $1\leq i\leq n-1$ in $H_1(\SH )$. 
The relation~(2) (d) in Theorem~\ref{thm_pres_SH} is equivalent to the relation $\widetilde{r}_i^2=1$ for $1\leq i\leq n$ in $H_1(\SH )$. 
The relations~(2) (e), (f), (g), and (6) (b) in Theorem~\ref{thm_pres_SH} are equivalent to the relation $\widetilde{t}_{i}=\widetilde{t}_{i+1}$ for $1\leq i\leq n$ in $H_1(\SH )$. 
Up to these relations, the relations~(3), (4), (5), and (6) (c) in Theorem~\ref{thm_pres_SH} are equivalent to the relations $\widetilde{r}^2\widetilde{t}_1^{-n-1}=1$, $\widetilde{t}_1^k=(\widetilde{r}_1\widetilde{s}_1)^{-kn}$, $\widetilde{t}_1^{n+1}\widetilde{s}_1^{n(n+1)}=1$, and $\widetilde{t}_1^2=(\widetilde{r}_1\widetilde{s}_1)^{-2n}$ in $H_1(\SH )$, respectively. 
When $k$ is odd, the relation $\widetilde{t}_1^k=(\widetilde{r}_1\widetilde{s}_1)^{-kn}$ is equivalent to the relation $\widetilde{t}_1=(\widetilde{r}_1\widetilde{s}_1)^{-n}$ by the relation $\widetilde{t}_1^2=(\widetilde{r}_1\widetilde{s}_1)^{-2n}$. 
Thus by an argument similar to the proof of Theorem~\ref{thm_abel_smod}, we have $H_1(\SH )\cong \Z [\widetilde{s}_1]\oplus \Z _2[\widetilde{r}_1]\oplus \Z _2[X]$ for odd $k$, where $X=\widetilde{r}(\widetilde{r}_1\widetilde{s}_1)^{\frac{n(n+1)}{2}}$. 

Assume that $k\geq 4$ is even. 
As a presentation for an abelian group, we have
\begin{eqnarray*}
&&H_1(\SH )\\
&\cong &\left< \widetilde{s}_1, \widetilde{r}_1, \widetilde{t}_1, \widetilde{r} \middle| \widetilde{r}_1^2=1, \widetilde{r}^2\widetilde{t}_1^{-n-1}=1, \widetilde{t}_1^2(r_1s_1)^{2n}=1, \widetilde{t}_1^{n+1}\widetilde{s}_1^{n(n+1)}=1\right> \\
&\cong &\left< \widetilde{s}_1, \widetilde{r}_1, \widetilde{t}_1, \widetilde{r} \middle| \widetilde{r}_1^2=1, \widetilde{r}^2\widetilde{t}_1^{-n-1}=1, \widetilde{t}_1^2s_1^{2n}=1, \widetilde{t}_1^{n+1}\widetilde{s}_1^{n(n+1)}=1\right> \\
&\cong &\left< \widetilde{s}_1, \widetilde{r}_1, \widetilde{t}_1, \widetilde{r}, X \middle| \widetilde{r}_1^2=1, \widetilde{r}^2\widetilde{t}_1^{-n-1}=1, (\widetilde{t}_1s_1^{n})^2=1, \widetilde{t}_1^{n+1}\widetilde{s}_1^{n(n+1)}=1, X=\widetilde{t}_1\widetilde{s}_1^{n}\right> \\
&\cong &\left< \widetilde{s}_1, \widetilde{r}_1, \widetilde{r}, X \middle| \widetilde{r}_1^2=1, \widetilde{r}^2X^{-n-1}\widetilde{s}_1^{n(n+1)}=1, X^2=1, X^{n+1}=1\right> .
\end{eqnarray*}
When $n$ is odd, the last presentation is equivalent to the following presentations:  
\begin{eqnarray*}
&&\left< \widetilde{s}_1, \widetilde{r}_1, \widetilde{r}, X \middle| \widetilde{r}_1^2=1, \widetilde{r}^2\widetilde{s}_1^{n(n+1)}=1, X^2=1\right> \\ 
&\cong &\left< \widetilde{s}_1, \widetilde{r}_1, \widetilde{r}, X, Y \middle| \widetilde{r}_1^2=1, \widetilde{r}^2\widetilde{s}_1^{n(n+1)}=1, X^2=1, Y=\widetilde{r}\widetilde{s}_1^{\frac{n(n+1)}{2}}\right> \\ 
&\cong &\left< \widetilde{s}_1, \widetilde{r}_1, X, Y \middle| \widetilde{r}_1^2=1, Y^2=1, X^2=1 \right> \\ 
&\cong &\Z [\widetilde{s}_1]\oplus \Z _2[\widetilde{r}_1]\oplus \Z _2[X]\oplus \Z _2[Y],
\end{eqnarray*}
where $X=\widetilde{t}_1\widetilde{s}_1^{n}$ and $Y=\widetilde{r}\widetilde{s}_1^{\frac{n(n+1)}{2}}$. 
When $n$ is even, since $X=1$ in $H_1(\SH )$, we have 
\begin{eqnarray*}
H_1(\SH )
&\cong &\left< \widetilde{s}_1, \widetilde{r}_1, \widetilde{r} \middle| \widetilde{r}_1^2=1, \widetilde{r}^2\widetilde{s}_1^{n(n+1)}=1 \right> \\
&\cong &\left< \widetilde{s}_1, \widetilde{r}_1, \widetilde{r} \middle| \widetilde{r}_1^2=1, \widetilde{r}^2\widetilde{s}_1^{n(n+1)}=1 \right> \\
&\cong &\left< \widetilde{s}_1, \widetilde{r}_1, \widetilde{r}, Y\middle| \widetilde{r}_1^2=1, \widetilde{r}^2\widetilde{s}_1^{n(n+1)}=1, Y=\widetilde{r}\widetilde{s}_1^{\frac{n(n+1)}{2}} \right> \\
&\cong &\left< \widetilde{s}_1, \widetilde{r}_1, Y\middle| \widetilde{r}_1^2=1, Y^2=1 \right> \\
&\cong &\Z [\widetilde{s}_1]\oplus \Z _2[\widetilde{r}_1]\oplus \Z _2[Y],
\end{eqnarray*}
where $Y=\widetilde{r}\widetilde{s}_1^{\frac{n(n+1)}{2}}$. 
Therefore, we have completed the proof of Theorem~\ref{thm_abel_smod}. 
\end{proof}



\par
{\bf Acknowledgement:} The authors would like to express their gratitude to Susumu Hirose for helpful discussions and comments. 
The first author was supported by JSPS KAKENHI Grant Numbers JP19K23409 and 21K13794.

\end{document}